\documentclass[11 pt]{amsart}
\usepackage{amsmath}
\usepackage{latexsym,amsfonts,amssymb,mathrsfs}
\usepackage{mathdots}
\usepackage{color}
\usepackage [all,2cell,color]{xy}

\usepackage{tikz-cd}
 \usepackage[utf8]{inputenc}
 \DeclareFontFamily{U}{min}{}
 \DeclareFontShape{U}{min}{m}{n}{<-> udmj30}{}

\UseAllTwocells
\usepackage[
textwidth=3cm,
textsize=small,
colorinlistoftodos]
{todonotes}

\usetikzlibrary{positioning} 
\usetikzlibrary{decorations.pathreplacing}
\usetikzlibrary{arrows, decorations.markings} 
\pgfarrowsdeclarecombine{twotriang}{twotriang}{stealth}{stealth}
{stealth}{stealth}
\tikzset{mono/.style={>-stealth}} 
\tikzset{epi/.style={-twotriang}} 
\tikzset{arrow/.style={->}}
\tikzset{arrowshorter/.style={->, shorten <=2pt, shorten >=2pt}}
\tikzset{twoarrowlonger/.style={double,double distance=1.5pt,
shorten <=5pt,shorten >=6pt,
decoration={markings,mark=at position -4pt with {\arrow[scale=1.75]{>}}},
preaction={decorate}}} 

\tikzset{mapstikz/.style={-stealth, 
decoration={markings,mark=at position 0pt with {\arrow[scale=0.5]{|}}}, preaction={decorate}}}
\tikzset{dot/.style={circle,draw,fill,inner sep=1pt}}
\theoremstyle{plain}   
\newtheorem{thm}{Theorem}[section] 
\makeatletter\let\c@thm\c@thm\makeatother
\newtheorem{cor}{Corollary}[section]
\makeatletter\let\c@cor\c@thm\makeatother
\newtheorem{lem}{Lemma}[section]
\makeatletter\let\c@lem\c@thm\makeatother
\newtheorem{prop}{Proposition}[section]
\makeatletter\let\c@prop\c@thm\makeatother

\makeatletter\let\c@claim\c@thm\makeatother

\theoremstyle{definition}
\newtheorem{defn}{Definition}[section]
\makeatletter\let\c@defn\c@thm\makeatother
\newtheorem{const}{Construction}[section]
\makeatletter\let\c@const\c@thm\makeatother
\newtheorem{notn}{Notation}[section]
\makeatletter\let\c@notn\c@thm\makeatother

\theoremstyle{remark}
\newtheorem{rmk}{Remark}[section]
\makeatletter\let\c@rmk\c@thm\makeatother
\newtheorem{ex}{Example}[section]
\makeatletter\let\c@ex\c@thm\makeatother

\makeatletter\let\c@observationn\c@thm\makeatother

\makeatletter
\let\c@equation\c@thm
\numberwithin{equation}{section}
\makeatother

\usepackage[capitalise]{cleveref}

\newcommand{\newrefformat}[2]{}

\crefname{lem}{Lemma}{Lemmas}
\crefname{thm}{Theorem}{Theorems}
\crefname{defn}{Definition}{Definitions}
\crefname{notn}{Notation}{Notations}
\crefname{const}{Construction}{Constructions}
\crefname{prop}{Proposition}{Propositions}
\crefname{rmk}{Remark}{Remarks}
\crefname{cor}{Corollary}{Corollaries}
\crefname{equation}{Display}{Displays}
\crefname{ex}{Example}{Examples}

\newcommand{\cC}{\mathcal{C}}
\newcommand{\cD}{\mathcal{D}}

\newcommand{\cH}{\mathcal{H}}

\newcommand{\cP}{\mathcal{P}}

\newcommand{\cS}{\mathcal{S}}
\newcommand{\cT}{\mathcal{T}}

\newcommand{\cV}{\mathcal{V}}
\newcommand{\cW}{\mathcal{W}}

\newcommand{\cat}{\cC\!\mathit{at}}

\newcommand{\set}{\cS\!\mathit{et}}

\newcommand{\sset}{\mathit{s}\set}
\newcommand{\D}{\cD}

\newcommand{\sdot}{S_{\bullet}}
\newcommand{\sdotn}[1]{S_{#1}}
\newcommand{\pcat}{\mathcal{P}}
\newcommand{\epi}{\twoheadrightarrow}
\newcommand{\mono}{\rightarrowtail}

\newcommand{\adc}{\mathcal{DC}at_{aug}}
\newcommand{\pdc}{\mathcal{DC}at_*}
\newcommand{\asdc}{\mathcal{DC}at_{aug}^{st}}
\newcommand{\sdc}{\mathcal{DC}at^{st}}
\newcommand{\dc}{\mathcal{DC}at}
\newcommand{\psdc}{\mathcal{DC}at^{st}_*} 
\newcommand{\untwoseg}{\mathcal{U}2\mathcal{S}eg}

\newcommand{\Sq}{Sq\,}
\newcommand{\Ob}{Ob\,}
\newcommand{\Hor}{Hor\,}
\newcommand{\Ver}{Ver\,}
\newcommand{\Horz}{\cH or_0}
\newcommand{\Verz}{\cV er_0}
\newcommand{\Horo}{\cH or_1}
\newcommand{\Vero}{\cV er_1}
\newcommand{\aug}{A}

\DeclareMathOperator{\id}{id}
\DeclareMathOperator{\Mor}{mor}
\DeclareMathOperator{\Fun}{Fun}

\DeclareMathOperator{\Hom}{Hom}
\DeclareMathOperator{\tr}{tr}

\newcommand{\Pplus}{P^{\vartriangleleft}}
\newcommand{\Pminus}{P^{\vartriangleright}}

\DeclareFontFamily{OT1}{pzc}{}
\DeclareFontShape{OT1}{pzc}{m}{it}{<-> s * [1.10] pzcmi7t}{}
\DeclareMathAlphabet{\mathpzc}{OT1}{pzc}{m}{it}
\newcommand{\R}{\mathbb{R}}

\tikzset{ %style for graph nodes with label
    vnode/.style={circle, radius=2pt, minimum size=4pt, draw, fill, inner sep=0, label={[below,text height=5mm]:#1}}
}

\tikzset{
    boxy/.style={baseline={([yshift=0.5ex]current bounding box.center)}}
}

\usetikzlibrary{shapes}
\tikzset{
    vellipsefirstone/.style={draw, ellipse, minimum width=0.6*#1, minimum height=1.5*#1, blue, thick}
}
\tikzset{
    vellipsesecondone/.style={draw, ellipse, minimum width=0.6*#1, minimum height=1.5*#1, orange, thick}
}
\tikzset{
    vellipsethirdone/.style={draw, ellipse, minimum width=0.6*#1, minimum height=1.5*#1, green!40!black!60, thick}
}

\tikzset{
    vellipsefirsttwo/.style={draw, ellipse, minimum width=2*#1, minimum height=1.5*#1, blue, thick}
}
\tikzset{
    vellipsesecondtwo/.style={draw, ellipse, minimum width=2*#1, minimum height=1.5*#1, orange, thick}
}
\tikzset{
    vellipsefirstthree/.style={draw, ellipse, minimum width=3.5*#1, minimum height=1.5*#1, blue, thick}
}

\newdir{ >}{{}*!/-5pt/@{>}}

\begin{document}
\title{2-Segal sets and the Waldhausen construction}

\author{Julia E. Bergner}
\address{Department of Mathematics, University of Virginia, Charlottesville, VA 22904, USA}
\email{bergnerj@member.ams.org}
\author{Ang\'{e}lica M. Osorno}
\address{Department of Mathematics, Reed College, Portland, OR 97202, USA}
\email{aosorno@reed.edu}
\author{Viktoriya Ozornova}
\address{Mathematical Institute, University of Bonn, 53115 Bonn, Germany}
\email{ozornova@math.uni-bonn.de}
\author{Martina Rovelli}
\address{UPHESS BMI FSV, \'{E}cole Polytechnique F\'{e}d\'{e}rale de Lausanne, CH-1015 Lausanne, Switzerland}
\email{martina.rovelli@epfl.ch}
\author{Claudia I. Scheimbauer}
\address{Max Planck Institute for Mathematics, 53111 Bonn, Germany}
\email{scheimbauer@mpim-bonn.mpg.de}

\date{\today}



\thanks{The first-named author was partially supported by NSF CAREER award DMS-1352298. The second-named author was partially supported by a grant from the Simons Foundation (\#359449, Ang\'elica Osorno). The fifth-named author was partially supported by the grant P2EZP2\textunderscore 159113 from the Swiss National Science Foundation.}

\begin{abstract}
It is known by results of Dyckerhoff-Kapranov and of G\'alvez--Carrillo-Kock-Tonks that the output of the Waldhausen $\sdot$-construction has a unital 2-Segal structure.  Here, we prove that a certain $\sdot$-functor defines an equivalence between the category of augmented stable double categories and the category of unital 2-Segal sets.  The inverse equivalence is described explicitly by a path construction.  We illustrate the equivalence for the known examples of partial monoids, cobordism categories with genus constraints and graph coalgebras.
\end{abstract}

\maketitle

In \cite{waldhausen}, Waldhausen gave a definition of the algebraic $K$-theory of certain categories using the $\sdot$-construction. The input categories for this construction, now called Waldhausen categories, have specified cofibrations and weak equivalences, subject to some axioms, and generalize more classical notions such as exact categories.  The essential step is to construct a simplicial space whose $k$th space of simplices is the classifying space of the groupoid of diagrams of a certain shape. 

More recently, there have been many generalizations of this construction using several flavours of $(\infty, 1)$-categories as input, such as \cite{BarKthy}, \cite{BGT}, and \cite{FLP}.
Analyzing the structure of the output of such a construction in detail, Dyckerhoff and Kapranov \cite{DK} and, independently, G\'alvez--Carrillo, Kock, and Tonks \cite{GalvezKockTonks}, realized that for some inputs (e.g.~an exact category), it is not just a simplicial space, but has additional structure which generalizes that of a category up to homotopy.
This structure is referred to as a \emph{unital 2-Segal space} by the former authors, and as a \emph{decomposition space} by the latter.  

Although the two sets of authors come at the definition of unital 2-Segal space from two different perspectives and are motivated by different examples, it is significant that both groups of authors identify the output of the $\sdot$-construction as a key example.
One can therefore ask whether all unital $2$-Segal spaces arise via an $\sdot$-construction for a suitably generalized input category.

In this paper, we restrict ourselves to the discrete case of unital 2-Segal sets, which are certain simplicial sets, rather than simplicial spaces.  We briefly describe this structure in this context; a more precise definition is given in the next section.

A 1-\emph{Segal set} is a simplicial set $X$ such that the Segal maps 
\[ X_n \rightarrow \underbrace{X_1 \times_{X_0} \cdots \times_{X_0} X_1}_n \]
are isomorphisms for $n \geq 2$.  This condition allows us to think of $X$ as having an object set $X_0$, a morphism set $X_1$, and a composition which can be defined by the span
\[ X_1\times_{X_0} X_1 \leftarrow X_2 \to X_1, \]
since the first arrow is an isomorphism.  Indeed, a simplicial set is a 1-Segal set if and only if it is isomorphic to the nerve of a category.

In contrast, a 2-Segal set is a simplicial set $X$ such that certain maps
\[ X_n \rightarrow \underbrace{X_2 \times_{X_1} \cdots \times_{X_1} X_2}_{n-1} \]
are isomorphisms for $n \geq 3$.  In this setting, we still have an object set $X_0$ and a morphism set $X_1$, but we no longer have composition of all morphisms, since the first map in the span
\[ X_1\times_{X_0} X_1 \leftarrow X_2 \to X_1 \] 
is no longer necessarily invertible.  However, we can think of a 2-Segal set as having a multi-valued composition, where an element of $X_1 \times_{X_0} X_1$ is lifted to a preimage in $X_2$, which is in turn sent to its image in $X_1$.  Thus, two potentially composable morphisms could have no composite at all (if the preimage in $X_2$ is empty) or multiple composites (if the preimage has multiple elements).  The invertibility of the 2-Segal maps given above is used to prove that this multi-valued composition is associative.  We think of this structure as that of a \emph{multi-valued category}.   

To understand what the 2-Segal condition does, let us look more precisely at how the maps are defined.  Unlike the case of 1-Segal maps, where for every $n \geq 2$ there is a single map, 2-Segal maps are parametrized by triangulations of regular $(n+1)$-gons for $n\geq 3$. For example, if $X$ is a simplicial set, the two triangulations
\begin{center}
\begin{tikzpicture}[scale=1.5, baseline=(a)]
\node (a) at (0,0.5) {};
\draw (0,0) node[dot] {} node[below left] {0} --  (1,0) node[dot] {} node[below right] {1} -- (1,1) node[dot] {} node[above right] {2} -- (0,1) node[dot] {} node[above left] {3} -- cycle;
\draw (1,0) -- (0,1);
\draw (0.5,-0.5) node {$\mathcal T_1$};
\end{tikzpicture}
\hspace{1cm}\mbox{and}\hspace{1cm}
\begin{tikzpicture}[scale=1.5, baseline=(a)]
\node (a) at (0,0.5) {};
\draw (0,0) node[dot] {} node[below left] {0} --  (1,0) node[dot] {} node[below right] {1} -- (1,1) node[dot] {} node[above right] {2} -- (0,1) node[dot] {} node[above left] {3} -- cycle;
\draw (0,0) -- (1,1);
\draw (0.5,-0.5) node {$\mathcal T_2$};
\end{tikzpicture}
\end{center}
of the square induce two different maps $X_3 \rightarrow X_2 \times_{X_1} X_2$; indeed, different face maps are used to define the pullbacks corresponding to the two triangulations.  The fact that $X_3$ is isomorphic to both pullbacks gives associativity of the partially defined composition.  We specify these maps more precisely in the following section.

We also restrict ourselves to \emph{unital} 2-Segal sets, for which composition with identity morphisms always exists and is unique.  It is a result of Dyckerhoff and Kapranov that the category of unital 2-Segal sets is equivalent to the category of multi-valued categories \cite{DK}.
%[\S 3.3]

Now let us consider the $\sdot$-construction in this context. When applied to an exact category, the output of the $S_\bullet$-construction is a simplicial space obtained by taking the geometric realization of the groupoid of diagrams of a certain shape in every degree.  It is a result of both Dyckerhoff-Kapranov and G\'alvez--Carrillo-Kock-Tonks that this simplicial space is $2$-Segal.  Roughly speaking, a $3$-simplex arises from a cartesian square in the exact category, and one of the two 2-Segal maps in degree 3 extracts the cospan 
\[
  \begin{tikzpicture}[scale=0.8]
    \def\l{2cm}
    \begin{scope}
   \draw[fill] (0,0) circle (1pt) node (b0){};
   \draw[fill] (-\l,0) circle (1pt) node (b3){};
   \draw[fill] (0,\l) circle (1pt) node (b1){};

   \draw[arrow] (b1)--node[anchor=west](x01){}(b0);
   \draw[arrow] (b3)--node[anchor=north](x03){}(b0);
   
   \end{scope}
  \end{tikzpicture}
\]
of this square. As the cartesian square is determined up to isomorphism by its cospan, the corresponding 2-Segal map is a weak equivalence.

To get a 2-Segal set, a naive guess would be to take the \emph{set} of all diagrams of the required shape, rather than the classifying space of the groupoid of such diagrams of a particular shape.  Although this construction produces a simplicial set, it is not $2$-Segal. The obstruction boils
down to the fact that the cartesian square that completes a cospan in an exact category is only determined up to isomorphism.  Since this square is not defined uniquely, the 2-Segal map fails to be an isomorphism.

In this paper, we identify the optimal amount of structure so that a discrete version of the $\sdot$-construction can be defined and is a $2$-Segal set. Although exact categories do not fit into this discrete context, they inspire the structure and the properties we are looking for.
For instance, the input object should have a collection of distinguished squares with the property that every cospan can be completed to a distinguished square in a unique way. These ideas lead to the notion of \emph{stable pointed double category}, which we now briefly explain.

A double category is a category internal to categories. More informally, it consists of the following data, subject to some axioms: a set of objects; two different morphism sets which we suggestively call ``horizontal'' and ``vertical'' morphisms; and ``squares'', which have ``horizontal'' source and target morphisms and ``vertical'' source and target morphisms. We depict a square by
\[
  \begin{tikzpicture}[scale=0.8]
    \def\l{2cm}
    \begin{scope}
   \draw[fill] (0,0) circle (1pt) node (b0){};
   \draw[fill] (-\l,\l) circle (1pt) node (b2){};
   \draw[fill] (-\l,0) circle (1pt) node (b3){};
   \draw[fill] (0,\l) circle (1pt) node (b1){};

   \draw[epi] (b1)--node[anchor=west](x01){}(b0);
   \draw[mono] (b2)--node[anchor=south](x12){}(b1);
   \draw[epi] (b2)--node[anchor=east](x23){}(b3);
   \draw[mono] (b3)--node[anchor=north](x03){}(b0);

   \draw[twoarrowlonger] (-0.8*\l, 0.8*\l)--node[above, xshift=0.1cm]{}(-0.2*\l,0.2*\l);
   
   \draw[right of=b0, xshift=-0.75cm, yshift=-0.08cm] node (sc){.};
   \end{scope}
  \end{tikzpicture}
\]

To define a simplicial object via an $\sdot$-construction, we need the input to be \emph{pointed}.  In the framework of exact categories, we assume a zero object, but in the double category setup we only ask for an object $*$ which is initial for the horizontal category and terminal for the vertical category. For such a pointed double category $\mathcal D$ we let $S_n(\mathcal D)$ be the set of diagrams of the form
$$
\raisebox{-0.5\height}{\includegraphics{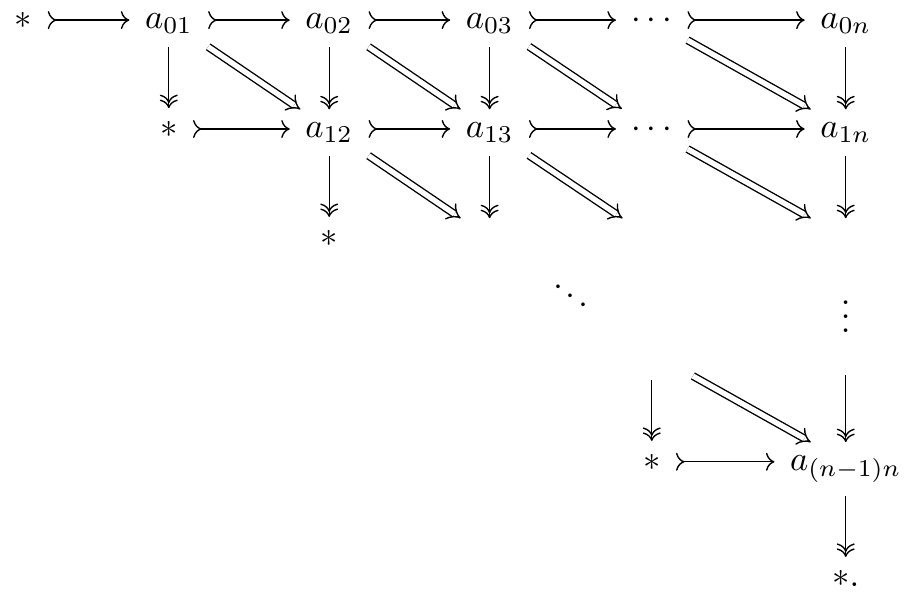}}
$$
As $n$ varies, we obtain a simplicial set whose face maps are given by deleting a row and a column, and composing appropriately.

If the pointed double category $\mathcal D$ is \emph{stable}, meaning that any square is uniquely determined by the span composed by its horizontal and vertical sources, and, simultaneously, by the cospan composed by its horizontal and vertical targets, then $\sdot(\mathcal D)$ is a 2-Segal set.  

Observe that $S_0(\mathcal{D}) = \{*\}$;  we call a 2-Segal set with this property \emph{reduced}. To obtain 2-Segal sets which are not reduced, we replace the singleton set $\{*\}$ by a subset of objects called the \emph{augmentation}.  When taking the $\sdot$-construction, we require the elements along the diagonal to be in the augmentation set.

The definitions of double categories and the conditions on them which we require can be found in \cref{sec:doublecats}, and the $\sdot$-construction is described explicitly in \cref{sec:wald}.
The following is the statement of our main result, \cref{thm:main_thm}. 

\theoremstyle{plain}
\newtheorem*{thm:main}{Main Theorem}
\begin{thm:main}
The generalized $\sdot$-construction defines an equivalence of categories between the category of augmented stable double categories and the category of unital 2-Segal sets. 
\end{thm:main}

The inverse functor can also be described explicitly via a path construction or d\'ecalage functor, as we explain in \cref{sec:path}. 
In this paper,  we illustrate the equivalence of the above theorem by three examples of 2-Segal sets which do not arise naturally from the ordinary $\sdot$-construction. The first one encodes the structure of Segal's partial monoids \cite{Segal}; the second one is borrowed from work of the fifth-named author and Valentino \cite{SV}, and encodes 2-dimensional cobordisms with genus constraints.  The third one is a 2-Segal set associated to a graph, and is a more basic version of a 2-Segal space encoding the combinatorics of graphs, related to the Schmitt coalgebra \cite{GalvezKockTonks_comb}, \cite{Schmitt}. 
We describe these 2-Segal sets in \cref{sec:twosegex}.  In \cref{sec:examples}, we return to these examples and give an explicit description of the associated augmented stable double category that results from applying the path construction functor to each.  

As mentioned earlier, the $\sdot$-construction of an exact category does not fit into the framework described in this paper, because it fails to be stable in the discrete setting. In a future paper, we will establish an equivalence of homotopy theories between unital 2-Segal spaces and double Segal spaces which are augmented and stable in a sense which appropriately generalizes the usage in the current paper.  The latter give an appropriate model for a homotopical version of augmented stable double categories.  

Such structures are likely to arise not only from exact categories, but also from stable $(\infty, 1)$-categories, for which an $\sdot$-construction has been already defined in \cite{BarKthy}, \cite{BGT}, and \cite{FLP}.  We expect our result to recover these known constructions. 
    
\subsection*{Acknowledgements}
We would like to thank the organizers of the Women in Topology II Workshop and the Banff International Research Station for providing a wonderful opportunity for collaborative research.  Conversations with T.\ Dyckerhoff, I.\ G\'alvez--Carrillo, and A.\ Tonks were helpful. The fifth-named author would like to thank A.\ Valentino for conversations which led to the example of cobordisms with genus constraints. We also thank the referee for a detailed report which led to many improvements to this paper.

\section{Some background on 2-Segal sets}

Before introducing 2-Segal sets, we recall the more familiar definition of 1-Segal sets.  

\begin{defn}\label{defn:1Segal}
A simplicial set $X$ is a \emph{1-Segal set} if, for all $n\geq 2$, the map
\[ X_n \rightarrow \underbrace{X_1 \times_{X_0} \cdots \times_{X_0} X_1}_n, \]
induced by the maps $[1]\cong\{i \leq i+1\} \hookrightarrow [n]$ in the category $\Delta$, for all $0 \leq i <n$, is a bijection.
\end{defn}

\begin{rmk} When unspecified, pullbacks of $X_1$ over $X_0$ follow the convention 
$$X_1 \times_{X_0} X_1:=\textrm{lim}\left[X_1\stackrel{d_0}\longrightarrow X_0\stackrel{d_1}\longleftarrow X_1 \right].$$
\end{rmk} 

The maps in \cref{defn:1Segal} can be defined for any simplicial object, not just for simplicial sets, and of particular interest has been the case of simplicial spaces, in which these maps are required to be weak equivalences rather than isomorphisms, and the pullbacks are taken to be homotopy pullbacks.  Rezk calls such simplicial spaces \emph{Segal spaces} in \cite{rezk}, taking the name from similar maps used by Segal in his work on $\Gamma$-spaces in \cite{segalgamma}.  Dyckerhoff and Kapranov use the term \emph{1-Segal spaces} to distinguish them from more general $n$-Segal spaces, and we follow their usage here.   While we restrict ourselves to the case of simplicial sets in this paper, and defer homotopical variants to later work, the following proposition points to the importance of 1-Segal spaces in defining homotopical categories.

\begin{prop}\label{prop:1segalnerve}
A simplicial set is 1-Segal if and only if it is isomorphic to the nerve of a category.
\end{prop}

\begin{rmk}\label{fundcat}
Given a 1-Segal set $X$, the corresponding category can be recovered using the fundamental category functor $\tau_1$. This functor is left adjoint to the nerve functor, and moreover, it is a left inverse (i.e., the counit is an isomorphism); see \cite[\S II.4]{GabrielZisman} for more details.  For a general simplicial set $X$, the set of objects of $\tau_1 X$ is $X_0$, and the morphisms are given by equivalence classes of strings of elements in $X_1$, where the equivalence relation is generated by elements of $X_2$. If $X$ is a 1-Segal set, every equivalence class is uniquely represented by an element in $X_1$. Thus, in this particular case, the sets of objects and morphisms of $\tau_1 X$ are given by $X_0$ and $X_1$, respectively, and the source, target, and identity maps are provided by $d_1$, $d_0$, and $s_0$. Finally, composition is given by the composite
\[X_1\times_{X_0} X_1 \underset{\cong}{\xrightarrow{(d_2,d_0)^{-1}}} X_2 \xrightarrow{d_1} X_1.\]  
Here, we use the fact that $X$ is 1-Segal, and hence that the map $(d_2, d_0)$ is an isomorphism, to obtain a single composition map $X_1 \times_{X_0} X_1 \rightarrow X_1$.  We frequently define composition maps via such diagrams in this paper.
\end{rmk}

We consider the following 2-dimensional generalization of 1-Segal sets as given by \cite{DK}.  

\begin{defn}
A \emph{2-Segal set} is a simplicial set such that, for every $n \geq 3$ and every triangulation $\cT$ of a regular $(n+1)$-gon, the induced map
\[ X_n \rightarrow \underbrace{X_2 \times_{X_1} \cdots \times_{X_1} X_2}_{n-1}, \]
which we call the $\cT$-\emph{Segal map}, is an isomorphism of sets. 
\end{defn} 

Let us explain further how these triangulations of polygons induce such maps.  To do so, we make use of the following notation.  If $S$ is a subset of $\{0,1, \ldots, n\}$, and $X$ is a simplicial set, then we denote by $X_S$ the set of $|S|$-simplices of $X$.  This notation is useful to specify face maps; for example we can denote $d_0 \colon X_2 \rightarrow X_1$ by $X_{\{0,1,2\}} \rightarrow X_{\{1,2\}}$.    

For $n\geq 3$, consider a regular $(n+1)$-gon with a cyclic labelling of its vertices by the set $\{0,1,\dots, n\}$. A triangulation of such an $(n+1)$-gon determines $n-1$ subsets of $\{0,\dots, n\}$ with exactly three elements (the triangles), and the iterated pullback in the $\cT$-Segal map is induced by taking triangles which agree along a 1-dimensional face (a two-element subset).  

For example, when $n=3$, the two triangulations
\begin{center}
\begin{tikzpicture}[scale=1.5, baseline=(a)]
\node (a) at (0,0.5) {};
\draw (0,0) node[dot] {} node[below left] {0} --  (1,0) node[dot] {} node[below right] {1} -- (1,1) node[dot] {} node[above right] {2} -- (0,1) node[dot] {} node[above left] {3} -- cycle;
\draw (1,0) -- (0,1);
\draw (0.5,-0.5) node {$\mathcal T_1$};
\end{tikzpicture}
\hspace{1cm}\mbox{and}\hspace{1cm}
\begin{tikzpicture}[scale=1.5, baseline=(a)]
\node (a) at (0,0.5) {};
\draw (0,0) node[dot] {} node[below left] {0} --  (1,0) node[dot] {} node[below right] {1} -- (1,1) node[dot] {} node[above right] {2} -- (0,1) node[dot] {} node[above left] {3} -- cycle;
\draw (0,0) -- (1,1);
\draw (0.5,-0.5) node {$\mathcal T_2$};
\end{tikzpicture}
\end{center}
of the square 
determine the two diagrams 
$$\begin{tikzcd}[column sep= small]
& X_3 \arrow[swap]{dl}{d_2} \arrow{dr}{d_0}\\
X_{\{0,1,3\}} \arrow[swap]{dr}{d_0}& & X_{\{1,2,3\}} \arrow{dl}{d_1}\\
& X_{\{1,3\}}
\end{tikzcd}
\quad\mbox{and}\quad  
\begin{tikzcd}[column sep= small]
&X_3 \arrow[swap]{dl}{d_3} \arrow{dr}{d_1}\\
X_{\{0,1,2\}} \arrow[swap]{dr}{d_1}& & X_{\{0,2,3\}} \arrow{dl}{d_2}\\
& X_{\{0,2\}},
\end{tikzcd}
$$
which in turn give two maps
$$f_{\mathcal T_1} \colon X_3 \longrightarrow X_{\{0,1,3\}} {\underset{X_{\{1,3\}}}{\times}} X_{\{1,2,3\}}
\quad\mbox{and}\quad  
f_{\mathcal T_2} \colon X_3 \longrightarrow X_{\{0,1,2\}} {\underset{X_{\{0,2\}}}{\times}} X_{\{0,2,3\}}.
$$

The following result follows from the definition of 2-Segal set.

\begin{prop}\label{prop:3iso}
Let $X$ be a simplicial set and $Y$ a 2-Segal set. 
\begin{enumerate}
\item \label{unique} A simplicial map $g\colon X \to Y$ is uniquely determined by $g_i$ for $i=0,1,2$.
\item\label{existence} Given $g_i\colon X_i \to Y_i$ for $i=0,1,2,3$ compatible with the simplicial structure, there exists an extension to a simplicial map $g\colon X\to Y$.
\item\label{isomorphism} If $X$ is also a 2-Segal set, then $g$ is an isomorphism of simplicial sets if and only if $g_i$ is an isomorphism for $i=0,1,2$.
\end{enumerate}
\end{prop}

\begin{proof}
Given $n\geq 3$, consider a triangulation $\cT$ of the $(n+1)$-gon. Then commutativity of the diagram
\[
\xymatrixcolsep{5pc}\xymatrix{
X_n \ar[r]^-{g_n} \ar[d]_{f_\cT} & Y_n \ar[d]_{\cong}^{f_{\cT}}\\
X_2 \times_{X_1} \cdots \times_{X_1} X_2 \ar[r]_{g_2\times \cdots \times g_2} & Y_2 \times_{Y_1} \cdots \times_{Y_1} Y_2,
}
\]
where the vertical arrows are the $\cT$-Segal maps, shows that $g_n$ is completely determined by $g_1$ and $g_2$, proving (\ref{unique}). The same diagram proves (\ref{isomorphism}): if  the two $\cT$-Segal maps are isomorphisms and $g_1$ and $g_2$ are isomorphisms, then the bottom map is also an isomorphism, and hence so is $g_n$.

To prove \eqref{existence}, we first prove that for $n\geq 3$ and any two triangulations $\cT$ and $\cT'$ of the $(n+1)$-gon, the diagram
\[
\xymatrix{
&X_2 \times_{X_1} \cdots \times_{X_1} X_2 \ar[rr]^{g_2\times \cdots \times g_2} && Y_2 \times_{Y_1} \cdots \times_{Y_1} Y_2 \ar[dr]_-{\cong}^-{f_{\cT}^{-1}}\\
X_n  \ar[dr]_-{f_{\cT'}}\ar[ur]^-{f_\cT} &&&& Y_n \\
&X_2 \times_{X_1} \cdots \times_{X_1} X_2 \ar[rr]_{g_2\times \cdots \times g_2} && Y_2 \times_{Y_1} \cdots \times_{Y_1} Y_2 \ar[ur]^-{\cong}_-{f_{\cT'}^{-1}}
}
\]
commutes. When $n=3$ there are only two triangulations, and the statement is true since we are given $g_3$, which is equal to both composites as shown in the diagram above. For $n\geq 4$, if two triangulations differ only by one edge, so that we can obtain one from the other by changing the diagonal of a sub-quadrilateral in the $(n+1)$-gon,  the result follows from the fact that the 2-Segal maps $f_\cT$ and $f_{\cT'}$ factor as
\[
\xymatrixcolsep{4pc}\xymatrixrowsep{3pc}\xymatrix{
&&X_2 \times_{X_1} \cdots \times_{X_1} X_2 \\
X_n  \ar[drr]_-{f_{\cT'}}\ar[urr]^-{f_\cT} \ar[rr] &&X_2 \times_{X_1} \cdots \times_{X_1} X_3 \times_{X_1} \cdots \times_{X_1} X_2 \ar[u]_{\id \times \cdots \times f_{\overline{\cT}}\times \cdots  \times \id} \ar[d]^{\id \times \cdots \times f_{\overline{\cT'}}\times \cdots  \times \id}\\
&&X_2 \times_{X_1} \cdots \times_{X_1} X_2 ,
}
\]
where $\overline{\cT}$ and $\overline{\cT'}$ are the two triangulations of the sub-quadrilateral specified above. Since any two triangulations can be related by a sequence of such one-edge flips, the result follows for arbitrary triangulations of the $(n+1)$-gon.

We can thus define $g_n$ as the composite $f^{-1}_{\cT} \circ (g_2 \times \cdots \times g_2) \circ f_{\cT}$ for any triangulation $\cT$. It remains to show that these maps are compatible with the simplicial structure. We prove the compatibility with $d_n\colon X_n \to X_{n-1}$; the other faces and the degeneracies are proved similarly. Let $\cT$ be a triangulation of the $(n+1)$-gon that contains the triangle $\{ 0,n-1,n\}$, and let $\cT'$ be the corresponding triangulation of the $n$-gon, obtained by removing that precise triangle. Then we have the commutative diagram
\[
\xymatrix{
X_n\ar[r]^-{f_{\cT}}\ar[d]_{d_n}& {\overbrace{X_2 \times_{X_1} \cdots \times_{X_1} X_2}^{n-1}} \ar[rr]^{g_2\times \cdots \times g_2} \ar[d] && {\overbrace{Y_2 \times_{Y_1} \cdots \times_{Y_1} Y_2}^{n-1}} \ar[d] \ar[r]_-{\cong}^-{f_{\cT}^{-1}}& Y_n \ar[d]^{d_n}\\
X_{n-1}\ar[r]_-{f_{\cT'}}& {\underbrace{X_2 \times_{X_1} \cdots \times_{X_1} X_2}_{n-2}} \ar[rr]_{g_2\times \cdots \times g_2} && {\underbrace{Y_2 \times_{Y_1} \cdots \times_{Y_1} Y_2}_{n-2}} \ar[r]^-{\cong}_-{f_{\cT'}^{-1}}& Y_{n-1},
}
\]
where the two vertical maps in the middle are projections onto all but one of the factors, corresponding to the triangle that was removed.
\end{proof}

We thus obtain the following consequence; for more details on coskeletal simplicial sets, see, for example \cite[\S IV.3.2]{GJ}.

\begin{cor}
If $Y$ is a 2-Segal set then it is 3-coskeletal.
\end{cor}

\begin{proof}
Let $\Delta_{\leq 3}$ denote the truncation of $\Delta$ containing only the objects $[i]$ for $i \leq 3$, and let $\sset_{\leq 3}$ denote the category of functors $\Delta^{\textrm{op}}_{\leq 3} \rightarrow \set$.  Note that (\ref{unique}) and (\ref{existence}) in \cref{prop:3iso} imply that for $X$ any simplicial set and $Y$ a 2-Segal set, the map
\[\Hom_{\sset}(X,Y) \longrightarrow \Hom_{\sset_{\leq 3}}(\tr_3 X, \tr_3 Y)\]
is an isomorphism. Thus any 2-Segal set $Y$ is 3-coskeletal.
\end{proof}

\begin{rmk}
For an interval $[x_0, x_n]$, we can consider a ``triangulation'' $\cT$ given by vertices $\{x_0 < x_1 < \cdots < x_n\}$,
$$\begin{tikz}
\draw (0,0) node[dot] {} node[anchor=north]{$x_0$} -- (1,0) node[dot]{} node[anchor=north]{$x_1$} -- (2,0) node[dot]{} node[anchor=north]{$x_2$} -- node[anchor=north]{$\cdots$} (4,0) node[dot]{} node[anchor=north]{$x_{n-1}$} -- (5,0) node[dot] {} node[anchor=north]{$x_n.$};
\end{tikz}$$
The combinatorics of the subdivision induce the usual 1-Segal maps
\[ X_n \rightarrow X_1 \times_{X_0} \cdots \times_{X_0} X_1. \]
Thus, we can see how 2-Segal sets give a ``higher-dimensional" generalization of 1-Segal sets.
\end{rmk}

One can also check that any 1-Segal set is also 2-Segal; a proof can be found in \cite{DK}. 

We now consider some additional properties on 2-Segal sets.

\begin{defn}
A 2-Segal set $X$ is \emph{reduced} if $X_0$ consists of a single point.
\end{defn}

The following definition is taken from \cite{DK}.

\begin{defn}  \label{defn:unital} 
A 2-Segal set is \emph{unital} if for all $n\geq 2$ and $0\leq i \leq n-1$, the diagram
\begin{center}
\begin{tikzcd}
X_{n-1} \arrow[r, "\alpha_i"] \arrow[d, "s_i"'] & X_0 \arrow[d, "s_0"]\\
X_n \arrow[r, "\beta_i"']  & X_1
\end{tikzcd}
\end{center}
is a pullback, where $s_0$ and $s_i$ are degeneracy maps and the maps $\alpha_i, \beta_i$ are induced by the following maps:
\[
 \begin{aligned}
\alpha^i \colon [0]&\to &&[n-1]\\
0&\mapsto &&i, \\
\beta^i \colon [1] &\to &&[n]\\
0&\mapsto&& i\\
1&\mapsto &&i+1.
 \end{aligned}
\]
\end{defn}

The conditions for unitality of a 2-Segal set can be reduced by the following lemma. The proof is analogous to that of a similar reduction in \cite{GalvezKockTonks}.
\begin{lem}\label{lem:easyunital}
A 2-Segal set $X$ is unital if and only if
\begin{center}
\begin{tikzcd}
X_{1} \arrow[r, "d_0"] \arrow[d, "s_1"'] & X_0 \arrow[d, "s_0"]\\
X_2 \arrow[r, "d_0"']  & X_1
\end{tikzcd}
\quad\mbox{and}\quad
\begin{tikzcd}
X_{1} \arrow[r, "d_1"] \arrow[d, "s_0"'] & X_0 \arrow[d, "s_0"]\\
X_2 \arrow[r, "d_2"']  & X_1
\end{tikzcd}
\end{center}
are pullback diagrams. 
\end{lem}

\begin{notn}
Let $\untwoseg$ and $\untwoseg_*$ denote the full subcategories of $\sset$ consisting of unital $2$-Segal sets and reduced unital 2-Segal sets, respectively.
\end{notn}

A useful way to check that a simplicial set is 2-Segal is the Path Space Criterion, which can be found in \cite{DK} and in a similar form in \cite{GalvezKockTonks}.
%Theorem 6.3.2 of \cite{DK} and in a similar form in Theorem 1.5.5 of \cite{GalvezKockTonks}. 

Recall the endofunctors of $\Delta$ which are used to define the path space construction,
 \[
  \begin{aligned}
   i \colon \Delta  \to & \Delta\\
   [n]  \mapsto & [0]*[n] 
   \end{aligned}
   \qquad \text{and} \qquad
   \begin{aligned}
   f \colon \Delta  \to & \Delta\\
   [n] \mapsto & [n]*[0],
  \end{aligned}
 \]
where $*$ denotes the join operation, which for linear posets is simply given by concatenation.  Here, the functor names are meant to suggest adjoining an initial and final object, respectively. Note that there are natural transformations $i\Rightarrow \id_{\Delta}$ and $f\Rightarrow \id_{\Delta}$ induced by the maps $\delta^0\colon [n] \to [0]*[n]=[n+1]$ and $\delta^{n+1}\colon [n] \to [n]*[0]=[n+1]$, respectively.

%[\S 6]
\begin{defn}[\cite{DK}]\label{defn:pathspace}
Given a simplicial set $X$, its \emph{path spaces} are the simplicial sets $\Pplus(X)=X\circ i^{\textrm{op}}$ and $\Pminus(X)=X\circ f^{\textrm{op}}$. 
\end{defn}

The natural transformations above induce maps of simplicial sets 
\[
d_0\colon \Pplus (X) \to X \quad \text{and} \quad d_{top} \colon \Pminus(X)\to X
\]
that are natural in $X$. These simplicial sets are also often called \emph{d\'ecalages}, for example in \cite{GalvezKockTonks}.

The following Path Space Criterion relates $1$-Segal sets and $2$-Segal sets.
%[Theorem 6.3.2]
\begin{thm}[\cite{DK}, \cite{GalvezKockTonks}]\label{thm:PathSpaceCriterion}
A simplicial set $X$ is $2$-Segal if and only if its path spaces $\Pplus X$ and $\Pminus X$ are $1$-Segal sets.
\end{thm}

\section{Three examples} \label{sec:twosegex}
In this section, we give explicit descriptions of three examples of 2-Segal sets.  The first is that of partial monoids; we follow the treatment of Segal in \cite{Segal}.

\begin{ex}\label{ex: partial monoids}
A partial monoid is a set $M$ together with a subset $M_2 \subseteq M\times M$ which is the domain of an associative multiplication. In other words, there is a map $\cdot \colon M_2\to M$ such that
\begin{enumerate}
\item for all $m,m',m''\in M$, we have that $(m\cdot m') \cdot m''$ is defined if an only if $m\cdot (m'\cdot m'')$ is defined, and if they are both defined, then we have {\em associativity} 
\[(m\cdot m') \cdot m'' = m\cdot (m'\cdot m''),\] 

\item there is a {\em unit} $1\in M$ such that for every $m\in M$, we have that $(1,m)\in M_2$ and $(m,1)\in M_2$, and $1\cdot m=m\cdot 1=m$.
\end{enumerate}

In \cite{Segal}, Segal defined the nerve of a partial monoid as the following simplicial set. Let $M_0=\{1\}$, and for $k\geq 1$, let $M_k \subseteq M^{\times k}$ be the subset of \emph{composable $k$-tuples}, which are elements $(m_1,\ldots, m_k)\in M^{\times k}$ such that $(m_1\cdots m_i,m_{i+1})\in M_2$ for every $1\leq i < k$. The face maps are given by composition, and the degeneracy maps are defined using insertion of the unit $1$. More precisely, for $0\leq i\leq k$,
\begin{align*}
d_i:  (m_1,\ldots, m_k) & \longmapsto \begin{cases}
(m_2,\ldots, m_k), & i=0,\\
(m_1,\ldots , m_i\cdot m_{i+1}, \ldots, m_k), & 0< i < k,\\
(m_1,\ldots, m_{k-1}), & i=k;
\end{cases}\\
s_i: (m_1,\ldots, m_k) & \longmapsto (m_1,\ldots, m_i, 1, m_{i+1},\ldots,  m_k).
\end{align*}
We claim that $M_\bullet$ is a 2-Segal set. Indeed, let $\cT$ be any triangulation of a polygon with $n$ vertices. We need to check that the induced
 $\cT$-Segal map
$$f_\cT \colon  M_n \rightarrow \underbrace{M_2\times_{M_1} \cdots \times_{M_1} M_2}_{n-1}$$
is a bijection. The maps used in the fiber product on the right-hand side are face maps. 

Let us first illustrate this bijection for an example. Consider the following triangulation of a pentagon:
$$\begin{tikzpicture}[scale=1.5]
\draw (0:1) node[dot] (0) {} node[anchor=west] {0}
-- (72: 1) node[dot] (1) {} node[anchor=south west] {1}
-- (144:1) node[dot] (2) {} node[anchor=south east] {2}
-- (216:1) node[dot] (3) {} node[anchor=north east] {3}
-- (288:1) node[dot] (4) {} node[anchor= north west] {4.}
-- cycle;
\draw (0) -- (2) -- (4);
\end{tikzpicture}$$

An element in $M_4$ is a composable 4-tuple $(m_1, m_2, m_3, m_4)$, i.e.,~we have that $(m_1, m_2), (m_1m_2, m_3),$ and $(m_1m_2m_3, m_4)\in M_2$. 
The $\cT$-Segal map lands in the fiber product
 $${M_2} \, {\underset{M_1}{^{d_1}\times^{d_2}}} M_2\,  {\underset{M_1}{^{d_0}\times^{d_1}}} M_2,$$
and it sends $(m_1,m_2,m_3,m_4)$ to 
\[ ((m_1,m_2),(m_1\cdot m_2,m_3\cdot m_4),(m_3,m_4)). \] 
An arbitrary element in the fiber product above is a triple $$(m_1, m_2), (n_1, n_2), (m_3, m_4)$$
of elements in $M_2$ such that $n_1=m_1\cdot m_2$ and $n_2=m_3\cdot m_4$. We can think of such an element as a decoration of the triangulation:
$$\begin{tikzpicture}[scale=1.5]\small
\draw (0:1) node[dot] (0) {} node[anchor=west] {0}
-- node[anchor=south west] {$m_1$} (72: 1) node[dot] (1) {} node[anchor=south west] {1}
-- node[above] {$m_2$} (144:1) node[dot] (2) {} node[anchor=south east] {2}
-- node[anchor=east] {$m_3$} (216:1) node[dot] (3) {} node[anchor=north east] {3}
-- node[below] {$m_4$} (288:1) node[dot] (4) {} node[anchor= north west] {4.}
-- node[anchor=north west] {$(m_1\cdot m_2)\cdot (m_3\cdot m_4)$} (360:1);
\draw (0) -- node[above] {$m_1\cdot m_2$} (2) -- node[anchor= west] {$m_3\cdot m_4$} (4);
\end{tikzpicture}$$
Since $(n_1, n_2) = (m_1\cdot m_2, m_3\cdot m_4)\in M_2$, we have that $(m_1\cdot m_2)\cdot (m_3\cdot m_4)$ is well-defined and by associativity, $(m_1\cdot m_2\cdot m_3, m_4)\in M_2$. In particular, all other necessary products are well-defined and therefore $(m_1,m_2,m_3,m_4)\in M_4$.

The essential ingredient of this argument is to relate an element in the fiber product appearing in the $\cT$-Segal map with a decoration of the triangulation. In particular, such an element determines the labels $m_1,\ldots, m_n$ decorating the outer edges of the $n$-gon except for the last one (the 0$n$ edge). The diagonals in the interior are decorated by iterated products of some of these elements. Finally, the 0$n$ edge is decorated by an iterated product of all of the elements; it is a classical result going back to Catalan \cite{Catalan} that triangulations of an $(n+1)$-gon are in one-to-one correspondence with the ways of bracketing a product of $n$ elements. Associativity of the multiplication then implies that $(m_1,\ldots, m_n)\in M_n$.

Moreover, the 2-Segal set $M_\bullet$ is unital. The map $\beta_i \colon M_{n} \to M_1$ of \cref{defn:unital} sends $(m_1,\dots , m_{n})\in M_{n}$ to $m_{i+1}$. Thus, an element of the pullback of the diagram
\[
\xymatrix{
 & M_0 \ar[d]^-{s_0}\\
M_{n} \ar[r]_-{\beta_i}  & M_1
}
\]
is an element of $M_{n}$ of the form $(m_1,\dots, m_i, 1, m_{i+2},\dots, m_{n})$. The map from $M_{n-1}$ to the pullback induced by the universal property, which sends $(m_1,\dots, m_{n-1})$ to $(m_1,\dots, m_i, 1, m_{i+1},\dots, m_{n-1})$, is then an isomorphism.
\end{ex}

The second example is that of cobordisms with a genus constraint, which is taken from and treated in more detail in work of the fifth-named author and Valentino in \cite{SV}. 

\begin{ex} \label{ex:cob} 
Fix a non-negative integer $g$, and consider 2-dimensional cobordisms with the constraint that its genus is less than or equal to $g$.   Following the definition of (the nerve of) the usual 2-dimensional cobordism category, we consider the following simplicial set $2\mathrm{Cob}^{\leq g}$.
\begin{itemize}
\item Let the elements in $(2\mathrm{Cob}^{\leq g})_0$ be 1-dimensional closed manifolds, which can be depicted by
$$\begin{tikzpicture}[scale=0.1]
\draw (0,0) circle (3);
\draw (7,0) circle (3);
\draw (14,0) node{$\cdots$};
\draw (21,0) circle (3);
\end{tikzpicture}$$
and are just disjoint unions of circles.

\item Let the elements of $(2\mathrm{Cob}^{\leq g})_k$ be diffeomorphism classes of 2-di\-men\-sion\-al cobordisms $\Sigma$ between 1-dimensional closed manifolds $\partial_{\textrm{in}} \Sigma$ and $\partial_{\textrm{out}} \Sigma$ with genus less than or equal to $g$, together with a decomposition thereof into $k$ cobordisms $\Sigma_1, \ldots, \Sigma_k$. Here, a diffeomorphism of such decomposed cobordisms must restrict to the individual composed cobordisms $\Sigma_i$ for $1\leq i \leq k$. We write $(\Sigma_1, \ldots, \Sigma_k)$ for an element in $(2\mathrm{Cob}^{\leq g})_k$, since the individual cobordisms fully determine the $k$-simplex.
\end{itemize}

For example, if $g=0$, the left picture is allowed as a 3-simplex of $2\mathrm{Cob}^{\leq 0}$, whereas the second one is not:
$$
\raisebox{-0.5\height}{\includegraphics{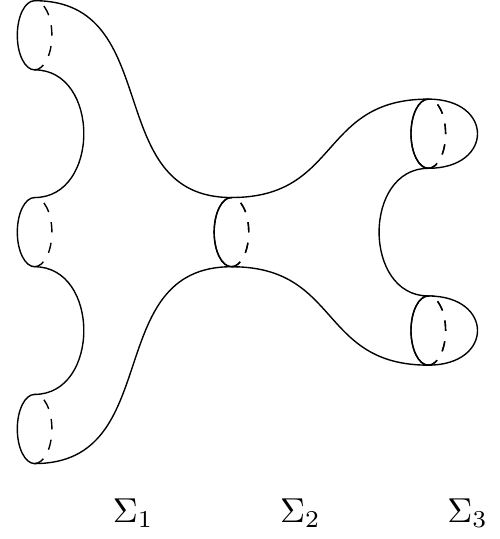}}
\hspace{1.5cm}
\raisebox{-0.5\height}{\includegraphics{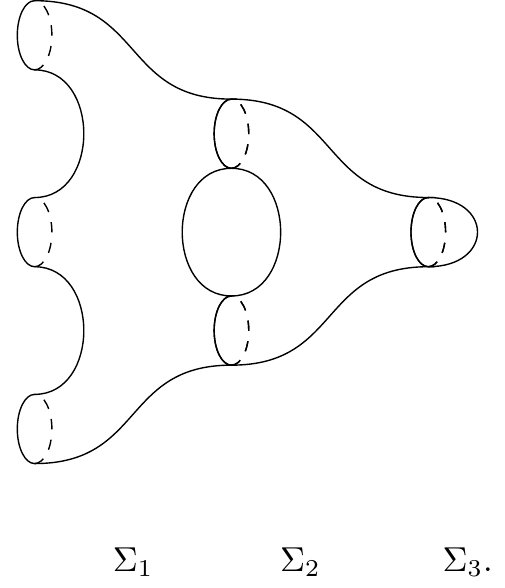}}
$$

Observe that $2\mathrm{Cob}^{\leq g}$ is not the nerve of a category, because not all pairs with compatible outgoing and incoming boundary components compose, as illustrated by the picture above.  However, it is a unital 2-Segal set which is a simplicial subset of the nerve of the usual 2-dimensional cobordism category $2\mathrm{Cob}$. 
\end{ex}

Our next example is inspired by the example of the 2-Segal space of graphs described by G\'alvez--Carrillo, Kock, and Tonks in \cite{GalvezKockTonks_comb}, from which one can obtain the Hopf algebra of graphs. Our example instead looks at a 2-Segal set corresponding to a single graph.

\begin{ex} \label{ex:graph}
Let $G$ be a graph consisting of a set $v(G)$ of vertices of $G$ and a set of edges between vertices.  We associate to $G$ a simplicial set $X$ as follows.
\begin{enumerate}
    \item The set $X_0$ has a single element which we denote by $\varnothing$.
    
    \item \label{x1} The set $X_1$ is the set of all subgraphs of $G$.
    
    \item \label{xn} Any $X_n$ has elements $(H; S_1, \ldots, S_n)$ where $H$ is a subgraph of $G$ and the sets $S_1, \ldots, S_n$ form a partition of the set $v(H)$ of vertices into $n$ disjoint (but possibly empty) sets. 

    \item The face maps $d_i \colon X_n \rightarrow X_{n-1}$ are defined in the following way.
    \begin{enumerate}
        \item If $i=0$, then
        \[ d_0(H; S_1, \ldots, S_n) = (H^{(0)}; S_2, \ldots, S_n) \]
        where $H^{(0)}$ denotes the full subgraph of $H$ on the vertices $v(H) \setminus S_1$.
        
        \item If $i=n$, then 
        \[ d_n(H; S_1, \ldots, S_n) = (H^{(n)}; S_1, \ldots, S_{n-1}) \]
        where $H^{(n)}$ denotes the full subgraph of $H$ on the vertices $v(H) \setminus S_n$.
        
        \item If $0<i<n$, then 
        \[ d_i(H; S_1, \ldots, S_n) =(H; S_1, \ldots S_{i-1}, S_i \cup S_{i+1}, S_{i+2}, \ldots, S_n). \]
    \end{enumerate}
    
    \item The degeneracy maps $s_i \colon X_n \rightarrow X_{n+1}$ are given by %\vonote{I think if we start numbering at 1, here the $\varnothing$ should be between i and i+1}
    \[ s_i(H; S_1, \ldots, S_n) = (H; S_1, \ldots S_i, \varnothing, S_{i+1}, \ldots, S_n). \]
    \end{enumerate}

Observe that condition \eqref{x1} is actually encompassed by condition \eqref{xn}, if we regard a graph as being partitioned into a single set. We use this uninteresting partition in the depictions of particular examples that follow.  

One can check that this simplicial set is 2-Segal but not 1-Segal. 
Let us investigate a specific example.  Consider the graph 
\begin{center}
\begin{tikzpicture}
\def\l{1cm} %length for the edges of the graph
%the graph itself
\draw (0,0) node (G){$G:$};
\draw (\l,0)  node[vnode=$a$](a){};
\draw (2*\l,0) node[vnode=$b$](b){};
\draw (3*\l,0) node[vnode=$c.$](c){};
\draw (a)--(b);
\draw (b)--(c);
\end{tikzpicture}
\end{center}
Then we can depict the set of 1-simplices as: 

\begin{tikzpicture}
\def\l{1cm} 

%1-simplices
 \draw (0,0) node {$X_1:$};
 \draw (\l,0) node {$\varnothing$};
 \draw (\l,0) node[vellipsefirstone=\l]{};
 
\begin{scope}[xshift=2*\l]
%length for the edges of the graph
\draw (0,0) node [vnode=$a$](a){};
\draw (0,0) node[vellipsefirstone=\l]{};
\end{scope}

 \begin{scope}[xshift=3*\l]
\draw (0,0) node [vnode=$b$](b){};
\draw (0,0) node[vellipsefirstone=\l]{};
\end{scope}

 \begin{scope}[xshift=4*\l]
\draw (0,0) node [vnode=$c$](c){};
\draw (0,0) node[vellipsefirstone=\l]{};
\end{scope}

\begin{scope}[xshift=4*\l]
\draw (\l,0) node [vnode=$a$](a){};
\draw (2*\l,0) node [vnode=$b$](b){};
\draw (1.5*\l,0) node[vellipsefirsttwo=\l]{};
\end{scope}

\begin{scope}[xshift=6.5*\l]
\draw (\l,0) node [vnode=$a$](a){};
\draw (2*\l,0) node [vnode=$b$](b){};
\draw (a)--(b);
\draw (1.5*\l,0) node[vellipsefirsttwo=\l]{};
\end{scope}

\begin{scope}[xshift=9.0*\l]
\draw (\l,0) node [vnode=$a$](a){};
\draw (2*\l,0) node [vnode=$c$](c){};
\draw (1.5*\l,0) node[vellipsefirsttwo=\l]{};
\end{scope}

\begin{scope}[yshift=-2*\l,xshift=\l]
\draw (\l,0) node [vnode=$b$](b){};
\draw (2*\l,0) node [vnode=$c$](c){};
\draw (1.5*\l,0) node[vellipsefirsttwo=\l]{};
\end{scope}

\begin{scope}[yshift=-2*\l, xshift=3.5*\l]
\draw (\l,0) node [vnode=$b$](b){};
\draw (2*\l,0) node [vnode=$c$](c){};
\draw (b)--(c);
\draw (1.5*\l,0) node[vellipsefirsttwo=\l]{};
\end{scope}
 
\begin{scope}[yshift=-2*\l, xshift=6*\l]
\draw[fill] (\l,0) node [vnode=$a$](a){};
\draw (2*\l,0) node [vnode=$b$](b){};
\draw (3*\l,0) node [vnode=$c$](c){};
\draw (2*\l,0) node[vellipsefirstthree=\l]{};
\end{scope}

\begin{scope}[yshift=-4*\l,xshift=\l]
\draw (\l,0) node [vnode=$a$](a){};
\draw (2*\l,0) node [vnode=$b$](b){};
\draw (3*\l,0) node [vnode=$c$](c){};
\draw (a)--(b);
\draw (2*\l,0) node[vellipsefirstthree=\l]{};
\end{scope}

\begin{scope}[yshift=-4*\l, xshift=5.5*\l]
\draw (\l,0) node [vnode=$a$](a){};
\draw (2*\l,0) node [vnode=$b$](b){};
\draw (3*\l,0) node [vnode=$c$](c){};
\draw (b)--(c);
\draw (2*\l,0) node[vellipsefirstthree=\l]{};
\end{scope}

\begin{scope}[yshift=-6*\l,xshift=3*\l]
\draw (\l,0) node [vnode=$a$](a){};
\draw (2*\l,0) node [vnode=$b$](b){};
\draw (3*\l,0) node [vnode=$c.$](c){};
\draw (a)--(b);
\draw (b)--(c);
\draw (2*\l,0) node[vellipsefirstthree=\l]{};
\end{scope}
\end{tikzpicture} 

We do not list all the 2-simplices, but illustrate with some examples of face maps.  We illustrate partitions by colored circles; we use blue for the first element of the partition, orange for the second, and, for 3-simplices, green for the third.
In most examples the partitions are ordered from left to right.  First, we show the effect of all the face maps on a representative element of $X_2$:
\begin{center}
\begin{tikzpicture}
\def\l{1cm} %length for the edges of the graph
%the original 2-simplex
\begin{scope}

%the subgraph 
\draw[fill] (\l,0) node [vnode=$a$](a){};
\draw (2*\l,0) node [vnode=$b$](b){};
\draw (3*\l,0) node [vnode=$c$](c){};
\draw (a)--(b);

%partitioning
\draw (\l,0) node[vellipsefirstone=\l]{};
\draw (2.5*\l, 0) node[vellipsesecondtwo=\l]{};
\end{scope}

%d0
\begin{scope}[yshift={-1.7*\l}, xshift={-3*\l}]
\draw (\l,0) node [vnode=$b$](b){};
\draw (2*\l,0) node [vnode=$c$](c){};
\draw (1.5*\l, 0) node[vellipsefirsttwo=\l]{};
\end{scope}

%d1
\begin{scope}[yshift=-2.2*\l]
\draw (\l,0) node [vnode=$a$](a){};
\draw (2*\l,0) node [vnode=$b$](b){};
\draw (3*\l,0) node [vnode=$c$](c){};
\draw (a)--(b);
\draw (2*\l, 0) node[vellipsefirstthree=\l]{};
\end{scope}

%d2
\begin{scope}[yshift=-1.7*\l, xshift=4*\l]
\draw (\l,0) node [vnode=$a.$](a){};
\draw (\l,0) node[vellipsefirstone=\l]{};
\end{scope}

%arrows
\draw[|->, thick] (0.5*\l, -0.1*\l)--node[above](d0){$d_0$}(-1.0*\l, -1.0*\l);
\draw[|->, thick] (2*\l, -0.8*\l)--node[left](d1){$d_1$}(2*\l, -1.3*\l);
\draw[|->, thick] (3.5*\l, -0.6*\l)--node[above](d2){$d_2$}(4.4*\l, -1.3*\l);
\end{tikzpicture}
\end{center}
Likewise, we depict the degeneracies of a particular 1-simplex as follows:
\begin{center}
\begin{tikzpicture}
\def\l{1cm} %length for the edges of the graph
%the original 1-simplex
\begin{scope}

%the subgraph 
\draw (\l,0) node [vnode=$a$](a){};
\draw (2*\l,0) node [vnode=$b$](b){};
\draw (a)--(b);
\draw (1.5*\l,0) node[vellipsefirsttwo=\l]{};
\end{scope}

%s0
\begin{scope}[yshift={-2*\l}, xshift={-2*\l}]
\draw (\l,0) node (emp){$\varnothing$};
\draw (2*\l,0) node [vnode=$a$](a){};
\draw (3*\l,0) node [vnode=$b$](b){};
\draw (a)--(b);
%partitioning
\draw (\l,0) node[vellipsefirstone=\l]{};
\draw (2.5*\l, 0) node[vellipsesecondtwo=\l]{};
\end{scope}

%s1
\begin{scope}[yshift=-2*\l, xshift=2*\l]
\draw[fill] (\l,0) node [vnode=$a$](a){};
\draw (2*\l,0) node [vnode=$b$](b){};
\draw (3*\l,0) node(emp){$\varnothing.$};
\draw (a)--(b);
\draw (1.5*\l,0) node[vellipsefirsttwo=\l]{};
\draw (3*\l, 0) node[vellipsesecondone=\l]{};
\end{scope}

%arrows
\draw[|->, thick] (0.4*\l, -0.2*\l)--node[above, yshift=0.1cm](s0){$s_0$}(-0.2*\l, -1.0*\l);
\draw[|->, thick] (2.6*\l, -0.2*\l)--node[above](s1){$s_1$}(3.5*\l, -1.0*\l);
\end{tikzpicture}
\end{center}
More interestingly, we have the following face maps of a 3-simplex of $X$:
\begin{center}
\begin{tikzpicture}
\def\l{1cm} %length for the edges of the graph
%the original 2-simplex
\begin{scope}

%the subgraph 
\draw[fill] (\l,0) node [vnode=$a$](a){};
\draw (2*\l,0) node [vnode=$b$](b){};
\draw (3*\l,0) node [vnode=$c$](c){};
\draw (b)--(c);

%partitioning
\draw (\l,0) node[vellipsefirstone=\l]{};
\draw (2*\l, 0) node[vellipsesecondone=\l]{};
\draw (3*\l, 0) node[vellipsethirdone=\l]{};
\end{scope}

%d0
\begin{scope}[yshift={-1.7*\l}, xshift={-3*\l}]
\draw (\l,0) node [vnode=$b$](b){};
\draw (2*\l,0) node [vnode=$c$](c){};
\draw (b)--(c);
\draw (\l,0) node[vellipsefirstone=\l]{};
\draw (2*\l, 0) node[vellipsesecondone=\l]{};
\end{scope}

%d1
\begin{scope}[yshift=-2.5*\l]
\draw[fill] (\l,0) node [vnode=$a$](a){};
\draw (2*\l,0) node [vnode=$b$](b){};
\draw (3*\l,0) node [vnode=$c$](c){};
\draw (b)--(c);
\draw (1.5*\l,0) node[vellipsefirsttwo=\l]{};
\draw (3*\l, 0) node[vellipsesecondone=\l]{};
\end{scope}

%d2
\begin{scope}[yshift=-2.5*\l, xshift=3.3*\l]
\draw[fill] (\l,0) node [vnode=$a$](a){};
\draw (2*\l,0) node [vnode=$b$](b){};
\draw (3*\l,0) node [vnode=$c$](c){};
\draw (b)--(c);
\draw (\l,0) node[vellipsefirstone=\l]{};
\draw(2.5*\l, 0) node[vellipsesecondtwo=\l]{};
\end{scope}

%d3
\begin{scope}[yshift={-1.7*\l}, xshift={6.3*\l}]
\draw (\l,0) node [vnode=$a$](a){};
\draw (2*\l,0) node [vnode=$b.$](b){};
\draw (\l,0) node[vellipsefirstone=\l]{};
\draw (2*\l, 0) node[vellipsesecondone=\l]{};
\end{scope}

%arrows
\draw[|->, thick] (0.5*\l, -0.1*\l)--node[above](d0){$d_0$}(-1.35*\l, -1.0*\l);
\draw[|->, thick] (2*\l, -0.8*\l)--node[left](d1){$d_1$}(2*\l, -1.8*\l);
\draw[|->, thick] (3.5*\l, -0.6*\l)--node[above](d2){$d_2$}(4.7*\l, -1.5*\l);
\draw[|->, thick] (4*\l, 0*\l)--node[above](d3){$d_3$}(7.7*\l, -1.0*\l);
\end{tikzpicture}
\end{center}

We can start to see from these face maps how one might reconstruct a 3-simplex from its faces.

Indeed, the fact that the associated simplicial set $X$ is 2-Segal can be illustrated by the diagrams
\begin{center}
\begin{tikzpicture}
\def\l{1cm} %length for the edges of the graph
%the original 2-simplex
\begin{scope}

%the subgraph 
\draw[fill] (\l,0) node [vnode=$a$](a){};
\draw (2*\l,0) node [vnode=$b$](b){};
\draw (3*\l,0) node [vnode=$c$](c){};
\draw (b)--(c);

%partitioning
\draw (\l,0) node[vellipsefirstone=\l]{};
\draw (2*\l, 0) node[vellipsesecondone=\l]{};
\draw (3*\l, 0) node[vellipsethirdone=\l]{};
\end{scope}

%d0
\begin{scope}[yshift=-3*\l, xshift=\l]
\draw (\l,0) node [vnode=$b$](b){};
\draw (2*\l,0) node [vnode=$c$](c){};
\draw (b)--(c);
\draw (\l,0) node[vellipsefirstone=\l]{};
\draw (2*\l, 0) node[vellipsesecondone=\l]{};
\end{scope}

%d2
\begin{scope}[ xshift=4*\l]
\draw[fill] (\l,0) node [vnode=$a$](a){};
\draw (2*\l,0) node [vnode=$b$](b){};
\draw (3*\l,0) node [vnode=$c$](c){};
\draw (b)--(c);
\draw (\l,0) node[vellipsefirstone=\l]{};
\draw(2.5*\l, 0) node[vellipsesecondtwo=\l]{};
\end{scope}

\begin{scope}[yshift=-3*\l, xshift=5*\l]
\draw (\l,0) node [vnode=$b$](b){};
\draw (2*\l,0) node [vnode=$c$](c){};
\draw (b)--(c);
\draw (1.5*\l, 0) node[vellipsefirsttwo=\l]{};
\end{scope}

%arrows
\draw[|->, thick] (2*\l, -1*\l)--node[left](d0){$d_0$}(2*\l, -2.1*\l);
\draw[|->, thick] (3.5*\l, 0)--node[above](d2){$d_2$}(4.5*\l, 0);
\draw[|->, thick] (3.5*\l, -3*\l)--node[below](d1){$d_1$}(5.4*\l, -3*\l);
\draw[|->, thick] (6*\l, -1*\l)--node[right](d0r){$d_0$}(6*\l, -2.2*\l);
\end{tikzpicture}
\end{center}
and
\begin{center}
\begin{tikzpicture}
\def\l{1cm} %length for the edges of the graph
%the original 2-simplex
\begin{scope}

%the subgraph 
\draw[fill] (\l,0) node [vnode=$a$](a){};
\draw (2*\l,0) node [vnode=$b$](b){};
\draw (3*\l,0) node [vnode=$c$](c){};
\draw (b)--(c);

%partitioning
\draw (\l,0) node[vellipsefirstone=\l]{};
\draw (2*\l, 0) node[vellipsesecondone=\l]{};
\draw (3*\l, 0) node[vellipsethirdone=\l]{};
\end{scope}

%d1
\begin{scope}[yshift=-3*\l]
\draw[fill] (\l,0) node [vnode=$a$](a){};
\draw (2*\l,0) node [vnode=$b$](b){};
\draw (3*\l,0) node [vnode=$c$](c){};
\draw (b)--(c);
\draw (1.5*\l,0) node[vellipsefirsttwo=\l]{};
\draw (3*\l, 0) node[vellipsesecondone=\l]{};
\end{scope}

%d3
\begin{scope}[xshift={4*\l}]
\draw (\l,0) node [vnode=$a$](a){};
\draw (2*\l,0) node [vnode=$b$](b){};
\draw (\l,0) node[vellipsefirstone=\l]{};
\draw (2*\l, 0) node[vellipsesecondone=\l]{};
\end{scope}

\begin{scope}[xshift={4*\l}, yshift=-3*\l]
\draw (\l,0) node [vnode=$a$](a){};
\draw (2*\l,0) node [vnode=$b.$](b){};
\draw (1.5*\l,0) node[vellipsefirsttwo=\l]{};
\end{scope}

\draw[|->, thick] (2*\l, -1*\l)--node[left](d1){$d_1$}(2*\l, -2.1*\l);
\draw[|->, thick] (3.5*\l, 0)--node[above](d3){$d_3$}(4.5*\l, 0);
\draw[|->, thick] (3.5*\l, -3*\l)--node[below](d2){$d_2$}(4.4*\l, -3*\l);
\draw[|->, thick] (5.5*\l, -1*\l)--node[right](d1r){$d_1$}(5.5*\l, -2.2*\l);
\end{tikzpicture}
\end{center}

However, $X$ is not 1-Segal, as can be illustrated by the diagrams
\begin{center}
\begin{tikzpicture}
\def\l{1cm} %length for the edges of the graph
\begin{scope}
\draw (\l,0) node [vnode=$a$](a){};
\draw (2*\l,0) node [vnode=$b$](b){};
\draw (\l,0) node[vellipsefirstone=\l]{};
\draw (2*\l, 0) node[vellipsesecondone=\l]{};
%faces
\draw (-0.3*\l,-0.7*\l) node [vnode=$a$](a2){};
\draw (3.3*\l,-0.7*\l) node [vnode=$b$](b2){};
%partitions for faces
\draw (-0.3*\l, -0.7*\l) node[vellipsefirstone=\l]{};
\draw (3.3*\l,-0.7*\l) node[vellipsefirstone=\l]{};
%arrows
\draw[|->, shorten >=0.35cm, shorten <=0.35cm, thick] (a)--node[above, xshift=-0.1cm]{$d_2$}(a2);
\draw[|->, shorten >=0.35cm, shorten <=0.35cm, thick] (b)--node[above, xshift=0.1cm]{$d_0$}(b2);
\end{scope}

\draw (4.2*\l, 0) node{$\quad$and$\quad$};

\begin{scope}[xshift=5.5*\l]
\draw (\l,0) node [vnode=$a$](a){};
\draw (2*\l,0) node [vnode=$b$](b){};
\draw (a)--(b);
\draw (\l,0) node[vellipsefirstone=\l]{};
\draw (2*\l, 0) node[vellipsesecondone=\l]{};
%faces
\draw (-0.3*\l,-0.7*\l) node [vnode=$a$](a2){};
\draw (3.3*\l,-0.7*\l) node [vnode=$b.$](b2){};
%partitions for faces
\draw (-0.3*\l, -0.7*\l) node[vellipsefirstone=\l]{};
\draw (3.3*\l,-0.7*\l) node[vellipsefirstone=\l]{};
%arrows
\draw[|->, shorten >=0.35cm, shorten <=0.35cm, thick] (a)--node[above, xshift=-0.1cm]{$d_2$}(a2);
\draw[|->, shorten >=0.35cm, shorten <=0.35cm, thick] (b)--node[above, xshift=0.1cm]{$d_0$}(b2);
\end{scope}
\end{tikzpicture}
\end{center}

In fact, this example shows that, unlike for the other two examples given in this section, the 1-Segal map 
\[ X_2 \rightarrow X_1 \times_{X_0} X_1 \]
is not injective.

One can check, however, that these 2-Segal sets arising from graphs are always unital.
\end{ex}

\section{Augmented stable double categories}  \label{sec:doublecats}

Our aim is to give a description of unital $2$-Segal sets in terms of categorical structures to which we can apply some version of Waldhausen $\sdot$-construction.   Here, we work with structures that have two different kinds of morphisms on the same set of objects, which are better described in terms of double categories. Double categories were first defined by Ehresmann \cite{Ehresmann}, and good introductory accounts can be found in \cite{FiorePaoliPronk} and \cite{GP}. We begin by recalling the definition. 

\begin{defn}  \label{doublecat}
 A \emph{(small) double category} is an internal category in the category of small categories.
\end{defn}

\begin{rmk}\label{doublecatremark}
Roughly speaking, being an internal category in the category of small categories means that a double category consists of a ``category of objects" and a ``category of morphisms".  However, this definition obscures the symmetry between the two kinds of morphisms. Hence, it is useful to unpack this definition further to see that a double category $\cD$ consists of sets of objects $\Ob{\cD}$, horizontal morphisms $\Hor{\cD}$, vertical morphisms $\Ver{\cD}$, and squares $\Sq{\cD}$, which are connected by various source, target, identity and composition maps. 

In particular, $(\Ob{\cD},\Hor{\cD})$ and $(\Ob{\cD},\Ver{\cD})$ form categories which we denote by $\Horz{\cD}$ and $\Verz{\cD}$, respectively. The subscript 0 indicates that each of these categories can be considered as a category of objects in a category internal to categories, as we explain below.  We denote horizontal and vertical morphisms by 
\[\begin{tikzpicture}\draw[mono] (0,0) -- (1,0);\end{tikzpicture} \quad\mbox{and}\quad \begin{tikzpicture}[baseline=(a)]\draw[epi] (0,0) --node (a) {} (0,-1);\end{tikzpicture}, 
\]
respectively. Squares are depicted by diagrams of the following form:
\begin{center}
  \begin{tikzpicture}[scale=0.8]
    \def\l{2cm}
    \begin{scope}
   \draw[fill] (0,0) circle (1pt) node (b0){};
   \draw[fill] (-\l,\l) circle (1pt) node (b2){};
   \draw[fill] (-\l,0) circle (1pt) node (b3){};
   \draw[fill] (0,\l) circle (1pt) node (b1){};

   \draw[epi] (b1)--node[anchor=west](x01){$k$}(b0);
   \draw[mono] (b2)--node[anchor=south](x12){$f$}(b1);
   \draw[epi] (b2)--node[anchor=east](x23){$j$}(b3);
   \draw[mono] (b3)--node[anchor=north](x03){$g$}(b0);

   \draw[twoarrowlonger] (-0.8*\l, 0.8*\l)--node[above, xshift=0.1cm]{$\alpha$}(-0.2*\l,0.2*\l);
   
   \draw[right of=b0, xshift=-0.75cm, yshift=-0.08cm] node (sc){.};
   \end{scope}
  \end{tikzpicture}
 \end{center}
There are horizontal and vertical source and target maps 
\[
s_h, t_h\colon \Sq{\cD}\to \Ver{\cD} \mbox{ and } s_v, t_v\colon \Sq{\cD}\to \Hor{\cD}
\]
given by, for example, $s_h(\alpha)=j$.  Note that the horizontal source and target of a square are vertical morphisms, and similarly the vertical source and target of a square are horizontal morphisms. In a double category we have horizontal composition of squares
\[\circ_{h}\colon \Sq{\cD}\times_{\Ver{\cD}} \Sq{\cD}\to \Sq{\cD}\] 
and vertical composition of squares
\[\circ_{v}\colon \Sq{\cD}\times_{\Hor{\cD}} \Sq{\cD}\to \Sq{\cD}.\]
These compositions are compatible with the compositions in $\Horz{\cD}$ and $\Verz{\cD}$ via source and target maps.  The compatibility allows depicting horizontal and vertical composition as 
  \begin{center}
  \begin{tikzpicture}[scale=0.8, baseline=(x01)]
    \def\l{2cm}
\begin{scope}
    \begin{scope}
   \draw[fill] (0,0) circle (1pt) node (b0){};
   \draw[fill] (-\l,\l) circle (1pt) node (b2){};
   \draw[fill] (-\l,0) circle (1pt) node (b3){};
   \draw[fill] (0,\l) circle (1pt) node (b1){};

   \draw[epi] (b1)--node[anchor=west](x01){}(b0);
   \draw[mono] (b2)--node[anchor=south](x12){}(b1);
   \draw[epi] (b2)--node[anchor=east](x23){}(b3);
   \draw[mono] (b3)--node[anchor=north](x03){}(b0);

   \draw[twoarrowlonger] (-0.8*\l, 0.8*\l)--(-0.2*\l,0.2*\l);
   \end{scope}
    \begin{scope}[xshift=2cm]
   \draw[fill] (0,0) circle (1pt) node (b0){};
   \draw[fill] (-\l,\l) circle (1pt) node (b2){};
   \draw[fill] (-\l,0) circle (1pt) node (b3){};
   \draw[fill] (0,\l) circle (1pt) node (b1){};

   \draw[epi] (b1)--node[anchor=west](x01){}(b0);
   \draw[mono] (b2)--node[anchor=south](x12){}(b1);
   \draw[epi] (b2)--node[anchor=east](x23){}(b3);
   \draw[mono] (b3)--node[anchor=north](x03){}(b0);

   \draw[twoarrowlonger] (-0.8*\l, 0.8*\l)--(-0.2*\l,0.2*\l);
   \end{scope}
\end{scope}
  \end{tikzpicture}

\quad\mbox{and}\quad

\begin{tikzpicture}[scale=0.8, baseline=(b1)]
    \def\l{2cm}
\begin{scope}
    \begin{scope}
   \draw[fill] (0,0) circle (1pt) node (b0){};
   \draw[fill] (-\l,\l) circle (1pt) node (b2){};
   \draw[fill] (-\l,0) circle (1pt) node (b3){};
   \draw[fill] (0,\l) circle (1pt) node (b1){};

   \draw[epi] (b1)--node[anchor=west](x01){}(b0);
   \draw[mono] (b2)--node[anchor=south](x12){}(b1);
   \draw[epi] (b2)--node[anchor=east](x23){}(b3);
   \draw[mono] (b3)--node[anchor=north](x03){}(b0);

   \draw[twoarrowlonger] (-0.8*\l, 0.8*\l)--(-0.2*\l,0.2*\l);
   \end{scope}
    
\end{scope}
\begin{scope}[yshift=-2cm]
    \begin{scope}
   \draw[fill] (0,0) circle (1pt) node (b0){};
   \draw[fill] (-\l,\l) circle (1pt) node (b2){};
   \draw[fill] (-\l,0) circle (1pt) node (b3){};
   \draw[fill] (0,\l) circle (1pt) node (b1){};

   \draw[epi] (b1)--node[anchor=west](x01){}(b0);
   \draw[mono] (b2)--node[anchor=south](x12){}(b1);
   \draw[epi] (b2)--node[anchor=east](x23){}(b3);
   \draw[mono] (b3)--node[anchor=north](x03){}(b0);

   \draw[twoarrowlonger] (-0.8*\l, 0.8*\l)--(-0.2*\l,0.2*\l);
   \draw[right of=b0, xshift=-0.75cm, yshift=-0.08cm] node (sc){.};
   \end{scope}

\end{scope}
  \end{tikzpicture}
 \end{center}
 
 There are also horizontal and vertical identity squares, respectively, which can be depicted by the following diagrams: 
  \begin{center}
  \begin{tikzpicture}[scale=0.8]
    \def\l{2cm}
    \begin{scope}
   \draw[fill] (0,0) circle (1pt) node (b0){\,};
   \draw[fill] (-\l,\l) circle (1pt) node (b2){};
   \draw[fill] (-\l,0) circle (1pt) node (b3){};
   \draw[fill] (0,\l) circle (1pt) node (b1){};

   \draw[epi] (b1)--node[anchor=west](x01){$j$}(b0);
   \draw[mono] (b2)--node[anchor=south](x12){$\id^h$}(b1);
   \draw[epi] (b2)--node[anchor=east](x23){$j$}(b3);
   \draw[mono] (b3)--node[anchor=north](x03){$\id^h$}(b0);

   \draw[twoarrowlonger] (-0.8*\l, 0.8*\l)--(-0.2*\l,0.2*\l);
   \draw (-0.3*\l, 0.6*\l)node(middle){$\id^h_j$};
   \end{scope}

      \draw (1.5, 0.5*\l) node{ and };
   
    \begin{scope}[xshift=5cm]
   \draw[fill] (0,0) circle (1pt) node (b0){\,\, .};
   \draw[fill] (-\l,\l) circle (1pt) node (b2){};
   \draw[fill] (-\l,0) circle (1pt) node (b3){};
   \draw[fill] (0,\l) circle (1pt) node (b1){};

   \draw[epi] (b1)--node[anchor=west](x01){$\id^v$}(b0);
   \draw[mono] (b2)--node[anchor=south](x12){$f$}(b1);
   \draw[epi] (b2)--node[anchor=east](x23){$\id^v$}(b3);
   \draw[mono] (b3)--node[anchor=north](x03){$f$}(b0);

   \draw[twoarrowlonger] (-0.8*\l, 0.8*\l)--(-0.2*\l,0.2*\l);
   \draw (-0.3*\l, 0.6*\l)node(middle){$\id^v_f$};
   
   \end{scope}
  \end{tikzpicture}
 \end{center}
 
This data is further subject to the following axioms:
 \begin{itemize} 
 \item associativity of the compositions $\circ_{h}$ and $\circ_v$; 
 \item unitality of $\circ_h$ and $\circ_v$: the identity squares $\id^h$ and $\id^v$ act as the identity for horizontal and vertical composition of squares, respectively; and  
 \item interchange law: given squares as in the following diagram, the order of composition does not matter, i.e.,~we can first compose either horizontally or vertically: 
  \begin{center}
  \begin{tikzpicture}[scale=0.8]
    \def\l{2cm}
\begin{scope}
    \begin{scope}
   \draw[fill] (0,0) circle (1pt) node (b0){};
   \draw[fill] (-\l,\l) circle (1pt) node (b2){};
   \draw[fill] (-\l,0) circle (1pt) node (b3){};
   \draw[fill] (0,\l) circle (1pt) node (b1){};

   \draw[epi] (b1)--node[anchor=west](x01){}(b0);
   \draw[mono] (b2)--node[anchor=south](x12){}(b1);
   \draw[epi] (b2)--node[anchor=east](x23){}(b3);
   \draw[mono] (b3)--node[anchor=north](x03){}(b0);

   \draw[twoarrowlonger] (-0.8*\l, 0.8*\l)--(-0.2*\l,0.2*\l);
   \end{scope}
    \begin{scope}[xshift=2cm]
   \draw[fill] (0,0) circle (1pt) node (b0){};
   \draw[fill] (-\l,\l) circle (1pt) node (b2){};
   \draw[fill] (-\l,0) circle (1pt) node (b3){};
   \draw[fill] (0,\l) circle (1pt) node (b1){};

   \draw[epi] (b1)--node[anchor=west](x01){}(b0);
   \draw[mono] (b2)--node[anchor=south](x12){}(b1);
   \draw[epi] (b2)--node[anchor=east](x23){}(b3);
   \draw[mono] (b3)--node[anchor=north](x03){}(b0);

   \draw[twoarrowlonger] (-0.8*\l, 0.8*\l)--(-0.2*\l,0.2*\l);
   \end{scope}
\end{scope}
\begin{scope}[yshift=-2cm]
    \begin{scope}
   \draw[fill] (0,0) circle (1pt) node (b0){};
   \draw[fill] (-\l,\l) circle (1pt) node (b2){};
   \draw[fill] (-\l,0) circle (1pt) node (b3){};
   \draw[fill] (0,\l) circle (1pt) node (b1){};

   \draw[epi] (b1)--node[anchor=west](x01){}(b0);
   \draw[mono] (b2)--node[anchor=south](x12){}(b1);
   \draw[epi] (b2)--node[anchor=east](x23){}(b3);
   \draw[mono] (b3)--node[anchor=north](x03){}(b0);

   \draw[twoarrowlonger] (-0.8*\l, 0.8*\l)--(-0.2*\l,0.2*\l);
   \end{scope}
    \begin{scope}[xshift=2cm]
   \draw[fill] (0,0) circle (1pt) node (b0){};
   \draw[fill] (-\l,\l) circle (1pt) node (b2){};
   \draw[fill] (-\l,0) circle (1pt) node (b3){};
   \draw[fill] (0,\l) circle (1pt) node (b1){};

   \draw[epi] (b1)--node[anchor=west](x01){}(b0);
   \draw[mono] (b2)--node[anchor=south](x12){}(b1);
   \draw[epi] (b2)--node[anchor=east](x23){}(b3);
   \draw[mono] (b3)--node[anchor=north](x03){}(b0);

   \draw[twoarrowlonger] (-0.8*\l, 0.8*\l)--(-0.2*\l,0.2*\l);
   
      \draw[right of=b0, xshift=-0.75cm, yshift=-0.08cm] node (sc){.};
   \end{scope}
\end{scope}
  \end{tikzpicture}
 \end{center}
 \end{itemize}

From a double category $\cD$ we can extract two categories $\Horo{\cD}$ and $\Vero{\cD}$ corresponding to horizontal and vertical composition of squares, respectively. For example, the category $\Horo{\cD}$ has $\Ver{\cD}$ as the set of objects and $\Sq{\cD}$ as the set of morphisms, with source, target, identity, and composition given by $s_h$, $t_h$, $\id^h$, and $\circ_h$, respectively. In fact, the categories $\Verz{\cD}$ and $\Vero{\cD}$ are the ``categories of objects'' and ``categories of morphisms'', respectively, of the category internal to categories $\cD$. Similarly, the pair $(\Horz{\cD},\Horo{\cD})$ also forms a category internal to categories, which is the \emph{transpose} of $\cD$.

The following diagram, adapted from \cite{GP}, gives a useful depiction of some the sets and categories that are associated to a double category and which have been described above.
\[ \xymatrix{
\Ob{\cD} \ar[r] \ar[d] & \Hor{\cD} \ar@<1ex>[l]^{t_h} \ar@<-1ex>[l]_{s_h} \ar[d]  & \Horz{\cD} \ar[d]\\
\Ver{\cD} \ar[r]  \ar@<1ex>[u]^{t_v} \ar@<-1ex>[u]_{s_v}  & \Sq{\cD} \ar@<1ex>[l]^{t_h} \ar@<-1ex>[l]_{s_h}  \ar@<1ex>[u]^{t_v} \ar@<-1ex>[u]_{s_v}  & \Horo{\cD} \ar@<1ex>[u] \ar@<-1ex>[u] \\
\Verz{\cD} \ar[r]  & \Vero{\cD} \ar@<1ex>[l] \ar@<-1ex>[l] 
}\]
\end{rmk}

\begin{ex}
If $\mathcal E$ is an exact category, the full subcategories of admissible monomorphisms and admissible epimorphisms assemble to the data of a double category, where the collection of distinguished squares is given by all commutative squares which are simultaneously pullbacks and pushouts.
\end{ex}

We can also consider appropriate functors between double categories. 
\begin{defn}
A \emph{double functor} between double categories is an internal functor.
\end{defn}

In other words, a double functor consists of an assignment of objects, horizontal and vertical morphisms, and squares, all of which are compatible with all source, target, composition, and unit maps.

To define a meaningful version of the Waldhausen $\sdot$-construction, we need to impose further conditions on double categories.  Recall that an exact category has a zero object, namely an object which is both initial and terminal.  Since categories with zero objects are often referred to as pointed, we use this terminology for our generalization to double categories.

\begin{defn}
A double category $\cD$ is \emph{pointed} if it is equipped with a fixed distinguished object $*$ which is both initial in $\Horz{\cD}$ and terminal in $\Verz{\cD}$.  A double functor $F\colon \cD \to \cD'$ between pointed double categories is \emph{pointed} if it sends the distinguished object of $\cD$ to the distinguished object of $\cD'$.
\end{defn} 

In order to obtain all 2-Segal sets, and not just those with a single 0-simplex, we will have need of a more general notion than pointedness.   

\begin{defn}
An \emph{augmentation} of a double category $\cD$ consists of a set of objects $\aug$ satisfying the condition that for every object $d$ of $\cD$ there are unique morphisms $a\mono d $ in $\Hor{\cD}$ and $d \epi a'$ in $\Ver{\cD}$ such that $a$ and $a'$ are in $\aug$.

An \emph{augmented double category} is a double category equipped with a choice of augmentation.   A double functor between augmented double categories is \emph{augmented} if objects in the augmentation of the source are sent to objects in the augmentation of the target.
\end{defn}

\begin{rmk}
Note that if the augmentation set is a singleton, so $\aug=\{*\}$, the definition reduces to being pointed by the single object $*$.
\end{rmk}

The following result is a reformulation of the notion of augmentation in categorical terms. 

\begin{prop}\label{augpb}
A subset $\aug$ of $\Ob{\cD}$ is an augmentation for the double category $\cD$  if and only if there exist maps
\begin{align*}
f\colon \Ob{\cD} & \longrightarrow \Hor{\cD} & g\colon \Ob{\cD} & \longrightarrow \Ver{\cD}\\
\delta^h\colon \Ob{\cD} & \longrightarrow \aug & \delta^v\colon \Ob{\cD} & \longrightarrow \aug,
\end{align*}
such that $t_h\circ f = \id$ and $s_v \circ g=\id$, and the two diagrams
\[\xymatrixrowsep{.8pc}\xymatrix{
\Ob{\cD}\ar@{-->}[r]^{\delta^h} \ar@{-->}[dd]_{f}&\aug\ar@{^{(}->}[dd]&& \Ob{\cD}\ar@{-->}[r]^{\delta^v} \ar@{-->}[dd]_{g}&\aug\ar@{^{(}->}[dd]\\
&&\text{and}\\
\Hor{\cD}\ar[r]_-{s_h}&\Ob{\cD}&& \Ver{\cD}\ar[r]_-{t_v}&\Ob{\cD}
}
\]
commute and are pullback diagrams.

Equivalently, $A$ is an augmentation for $\cD$ if and only if the maps
$$t_h\circ \mbox{pr}_1: \Hor{\cD}\times_{\Ob{\cD}} \aug \to \Ob{\cD} \quad\mbox{and}\quad s_h\circ \mbox{pr}_2: \Ver{\cD}\times_{\Ob{\cD}} \aug \to \Ob{\cD}$$
are bijections.
\end{prop}

\begin{rmk}
In particular, the maps $f$ and $\delta^h$ of \cref{augpb} make the nerve of the category $\Horz{\cD}$ into an augmented simplicial set with an extra degeneracy, more precisely, with a backward contracting homotopy,
\[\begin{tikzcd}
A \arrow[r, hook, bend left =50, ""] & \Ob{\cD} \arrow[l, "\delta^h"] \arrow[r, arrowshorter]  \arrow [r, bend left =50, "f"] & \Hor{\cD} \arrow [r, bend left =50, "f\times \id"] \arrow[l, arrow, shift left=1ex] \arrow[l, arrow, shift right=1ex] \arrow[r, arrowshorter, shift left=0.75ex] \arrow[r, arrowshorter, shift right=0.75ex] & 
\Hor{\cD} \times_{\Ob{\cD}} \Hor{\cD} \arrow[l, arrow] \arrow[l, arrow, shift left=1.5ex] \arrow[l, arrow, shift right=1.5ex] \arrow[r, phantom, "\cdots"] & {},
\end{tikzcd}\]
where $f\times \id: \Hor{\cD} \cong \Ob{\cD}\times_{\Ob{\cD}} \Hor{\cD} \to \Hor{\cD} \times_{\Ob{\cD}} \Hor{\cD}$, and similarly for the other extra degeneracy maps.

Similarly, $g$ and $\delta^v$ make the nerve of the category $\Verz{\cD}$ into an augmented simplicial set with a forward contracting homotopy given by $g$,
\[\begin{tikzcd}
A \arrow[r, hook, bend right =50, ""] & \Ob{\cD} \arrow[l, "\delta^v"] \arrow[r, arrowshorter]  \arrow [r, bend right =50, "g", swap] & \Ver{\cD} \arrow [r, bend right =50, swap, "\id\times g"] \arrow[l, arrow, shift left=1ex] \arrow[l, arrow, shift right=1ex] \arrow[r, arrowshorter, shift left=0.75ex] \arrow[r, arrowshorter, shift right=0.75ex] & 
\Ver{\cD} \times_{\Ob{\cD}} \Ver{\cD} \arrow[l, arrow] \arrow[l, arrow, shift left=1.5ex] \arrow[l, arrow, shift right=1.5ex] \arrow[r, phantom, "\cdots"] & {}.
\end{tikzcd}\]
\end{rmk}

In the definition of the $\sdot$-construction for exact categories, certain squares are required to be both pullbacks and pushouts.  In other words, such a square is determined, up to isomorphism, by its span; it is similarly determined by its cospan.  In a double category, this condition does not make sense, since in general the horizontal and vertical morphisms in a double category need not be morphisms in a common category, so the span or cospan of a square may not be given by a diagram in a category. Furthermore, pullbacks and pushouts are only determined up to isomorphism, but we need uniqueness on the nose in the discrete setting. 

\begin{defn}
A double category is \emph{stable} if every square is uniquely determined by its span of source morphisms and, independently, by its cospan of target morphisms.  More precisely, we require the maps
\[
 \begin{aligned}
  (s_h, s_v)\colon \Sq\cD \to \Ver{\cD} \times_{\Ob{\cD}} \Hor{\cD} \\
  (t_h, t_v)\colon \Sq\cD \to \Ver{\cD} \times_{\Ob{\cD}} \Hor{\cD}
 \end{aligned}
\]
depicted by
 \begin{center}
  \begin{tikzpicture}[scale=0.6]
    \def\l{2cm}
    \begin{scope}
   \draw[fill] (0,0) circle (1pt) node (b0){};
   \draw[fill] (-\l,\l) circle (1pt) node (b2){};
   \draw[fill] (-\l,0) circle (1pt) node (b3){};
   \draw[fill] (0,\l) circle (1pt) node (b1){};

   \draw[epi] (b1)--node[anchor=west](x01){}(b0);
   \draw[mono] (b2)--node[anchor=south](x12){}(b1);
   \draw[epi] (b2)--node[anchor=east](x23){}(b3);
   \draw[mono] (b3)--node[anchor=north](x03){}(b0);

   \draw[twoarrowlonger] (-0.8*\l, 0.8*\l)--(-0.2*\l,0.2*\l);
   \end{scope}
\draw[|->] (0.1*\l, 0.5*\l)--node[above](op1){}(0.4*\l,0.5*\l);
    \begin{scope}[xshift=1.5*\l]
   \draw[fill] (-\l,\l) circle (1pt) node (b2){};
   \draw[fill] (-\l,0) circle (1pt) node (b3){};
  \draw[fill] (0,\l) circle (1pt) node (b1){};

    \draw[mono] (b2)--node[anchor=south](x12){}(b1);
   \draw[epi] (b2)--node[anchor=east](x23){}(b3);
   \end{scope}
\draw (2.5*\l, 0.5*\l) node{and};
    \begin{scope}[xshift=4.5*\l]
   \draw[fill] (0,0) circle (1pt) node (b0){};
   \draw[fill] (-\l,\l) circle (1pt) node (b2){};
   \draw[fill] (-\l,0) circle (1pt) node (b3){};
   \draw[fill] (0,\l) circle (1pt) node (b1){};

   \draw[epi] (b1)--node[anchor=west](x01){}(b0);
   \draw[mono] (b2)--node[anchor=south](x12){}(b1);
   \draw[epi] (b2)--node[anchor=east](x23){}(b3);
   \draw[mono] (b3)--node[anchor=north](x03){}(b0);

   \draw[twoarrowlonger] (-0.8*\l, 0.8*\l)--(-0.2*\l,0.2*\l);
   \end{scope}

\draw[|->] (4.6*\l, 0.5*\l)--node[above](op1){}(4.9*\l,0.5*\l);
    \begin{scope}[xshift=6*\l]
   \draw[fill] (0,0) circle (1pt) node (b0){};
   \draw[fill] (-\l,0) circle (1pt) node (b3){};
   \draw[fill] (0,\l) circle (1pt) node (b1){};

   \draw[epi] (b1)--node[anchor=west](x01){}(b0);
   \draw[mono] (b3)--node[anchor=north](x03){}(b0);
   \end{scope}
\end{tikzpicture}
 \end{center}
to be bijections. 
\end{defn}
\begin{rmk}
The double category associated to an exact category is not stable, except for trivial examples. Indeed, a cartesian square is determined by the corresponding span only up to isomorphism.
\end{rmk}

An important consequence of stability, which we now show, is that the set of squares and their horizontal and vertical compositions are essentially determined by the rest of the data.

Let $\cH$ and $\cV$ be categories with the same set of objects $\Ob$, and $\Sq$ a set together with maps 
\[
s_h, t_h\colon \Sq\to\Mor \cV \quad \mbox{ and } \quad s_v, t_v\colon \Sq\to\Mor \cH,
\]
which we call horizontal and vertical source and target maps, respectively. As suggested by the name, the set $\Sq$ is a proposed set of squares for a double category. Assume that these maps are compatible with the source and target maps in $\cH$ and $\cV$, as in the definition of a double category (see \cref{doublecatremark}). For example, $s_{\cH}s_v=s_{\cV}s_h$. In other words, we are given some of the data and axioms for a double category as detailed in \cref{doublecatremark}, missing precisely horizontal and vertical compositions of squares and existence of horizontal and vertical identity squares. Furthermore, assume that the stability condition holds, i.e.,~the two maps
\begin{equation}\label{eqn_stability}\xymatrix{\Mor\cH\times_{\Ob} \Mor\cV&\ar[l]_-{(t_v,t_h)}\Sq\ar[r]^-{(s_h,s_v)}&\Mor\cV\times_{\Ob} \Mor\cH}
\end{equation}
are bijections.

Then there are two horizontal ``compositions'' of squares defined using composition in $\cH$ and two vertical ``compositions''\footnote{A priori, these maps are not fully compatible with source and target, and hence do not deserve to be called composition.} of squares defined using composition in $\cV$, as follows.

We define the first horizontal composition of squares using the bijection in \eqref{eqn_stability} given by $(s_h, s_v)$ and the second using the bijection in \eqref{eqn_stability} given by $(t_v, t_h)$:
\begin{equation}\label{h1h2}
\begin{tikzcd}[column sep=1.5cm]
\circ_h^1: \Sq \times_{\Mor \cV} \Sq \arrow[r, "{(s_h, s_v)\times s_v}"]& (\Mor\cV \times_{\Ob} \Mor \cH) \times_{\Ob} \Mor\cH \arrow{d}{\id \times \circ_{\cH}} \\
& \Mor\cV \times_{\Ob} \Mor\cH \arrow[r, "{(s_h, s_v)^{-1}}", "\cong"'] &[-1.7cm] \Sq , \\
\circ_h^2: \Sq \times_{\Mor \cV} \Sq \arrow[r, "{t_v \times (t_v,t_h)}"]& \Mor\cH \times_{\Ob} (\Mor \cH \times_{\Ob} \Mor\cV) \arrow{d}{ \circ_{\cH}\times \id} \\
& \Mor\cH \times_{\Ob} \Mor\cV \arrow[r, "{(t_v, t_h)^{-1}}", "\cong"'] & \Sq . 
\end{tikzcd}
\end{equation}
Pictorially, the two maps can be understood as follows:
 \begin{center}
   
  \begin{tikzpicture}[scale=0.5]
    \def\l{2cm}
    \begin{scope}
   \draw[fill] (0,0) circle (1pt) node (b0){};
   \draw[fill] (-\l,\l) circle (1pt) node (b2){};
   \draw[fill] (-\l,0) circle (1pt) node (b3){};
   \draw[fill] (0,\l) circle (1pt) node (b1){};
   
   \draw[fill] (\l,\l) circle (1pt) node (b4){};
   \draw[fill] (\l,0) circle (1pt) node (b5){};

   \draw[epi] (b1)--node[anchor=west](x01){}(b0);
   \draw[mono] (b2)--node[anchor=south](x12){}(b1);
   \draw[epi] (b2)--node[anchor=east](x23){}(b3);
   \draw[mono] (b3)--node[anchor=north](x03){}(b0);
   
   \draw[epi] (b4) -- (b5);
   \draw[mono] (b0) -- (b5);
   \draw[mono] (b1) -- (b4); 

   \draw[twoarrowlonger] (-0.8*\l, 0.8*\l)--(-0.2*\l,0.2*\l);
   \draw[twoarrowlonger] (0.2*\l, 0.8*\l)--(0.8*\l,0.2*\l);
   
   \end{scope}
\draw[|->] (1.25*\l, 0.5*\l)--node[above](op1){}(1.75*\l,0.5*\l);
    \begin{scope}[xshift=3*\l]
%   \draw[fill] (0,0) circle (1pt) node (b0){};
   \draw[fill] (-\l,\l) circle (1pt) node (b2){};
   \draw[fill] (-\l,0) circle (1pt) node (b3){};
   \draw[fill] (0,\l) circle (1pt) node (b1){};
   
   \draw[fill] (\l,\l) circle (1pt) node (b4){};

%   \draw[epi] (b1)--node[anchor=west](x01){}(b0);
    \draw[mono] (b2)--node[anchor=south](x12){}(b1);
   \draw[epi] (b2)--node[anchor=east](x23){}(b3);
      \draw[mono] (b1) -- (b4);
   \end{scope}
\draw[|->] (4.5*\l, 0.5*\l)--node[above](op1){}(5*\l,0.5*\l);   

    \begin{scope}[xshift=6.5*\l]
%   \draw[fill] (0,0) circle (1pt) node (b0){};
   \draw[fill] (-\l,\l) circle (1pt) node (b2){};
   \draw[fill] (-\l,0) circle (1pt) node (b3){};
%   \draw[fill] (0,\l) circle (1pt) node (b1){};
   
   \draw[fill] (\l,\l) circle (1pt) node (b4){};

%   \draw[epi] (b1)--node[anchor=west](x01){}(b0);
    \draw[mono] (b2)--node[anchor=south](x12){}(b4);
   \draw[epi] (b2)--node[anchor=east](x23){}(b3);
    \end{scope}
\draw[|->] (7.75*\l, 0.5*\l)--node[above](op1){$\cong$} (8.25*\l,0.5*\l);

\begin{scope}[xshift=9.75*\l]
%   \draw[fill] (0,0) circle (1pt) node (b0){};
   \draw[fill] (-\l,\l) circle (1pt) node (b2){};
   \draw[fill] (-\l,0) circle (1pt) node (b3){};
%   \draw[fill] (0,\l) circle (1pt) node (b1){};
   
    \draw[fill] (\l,\l) circle (1pt) node (b4){};
    \draw[fill] (\l,0) circle (1pt) node (b5){};

%   \draw[epi] (b1)--node[anchor=west](x01){}(b0);
    \draw[mono] (b2)--node[anchor=south](x12){}(b4);
    \draw[epi] (b2)--node[anchor=east](x23){}(b3);
    \draw[epi] (b4) -- (b5);
    \draw[mono] (b3)--node[anchor=south](x12){}(b5);
    
    \draw[twoarrowlonger] (-0.8*\l, 0.8*\l) -- (0.8*\l,0.2*\l);
    \end{scope}
\end{tikzpicture}
\begin{tikzpicture}[scale=0.5]
    \def\l{2cm}
    \begin{scope}
% \draw (-1.5*\l, 0.5*\l) node {$\circ_h^2:$};

   \draw[fill] (0,0) circle (1pt) node (b0){};
   \draw[fill] (-\l,\l) circle (1pt) node (b2){};
   \draw[fill] (-\l,0) circle (1pt) node (b3){};
   \draw[fill] (0,\l) circle (1pt) node (b1){};
   
   \draw[fill] (\l,\l) circle (1pt) node (b4){};
   \draw[fill] (\l,0) circle (1pt) node (b5){};

   \draw[epi] (b1)--node[anchor=west](x01){}(b0);
   \draw[mono] (b2)--node[anchor=south](x12){}(b1);
   \draw[epi] (b2)--node[anchor=east](x23){}(b3);
   \draw[mono] (b3)--node[anchor=north](x03){}(b0);
   
   \draw[epi] (b4) -- (b5);
   \draw[mono] (b0) -- (b5);
   \draw[mono] (b1) -- (b4); 

   \draw[twoarrowlonger] (-0.8*\l, 0.8*\l)--(-0.2*\l,0.2*\l);
   \draw[twoarrowlonger] (0.2*\l, 0.8*\l)--(0.8*\l,0.2*\l);
   
   \end{scope}
\draw[|->] (1.25*\l, 0.5*\l)--node[above](op1){}(1.75*\l,0.5*\l);
    \begin{scope}[xshift=3*\l]
   \draw[fill] (0,0) circle (1pt) node (b0){};
%   \draw[fill] (-\l,\l) circle (1pt) node (b2){};
   \draw[fill] (-\l,0) circle (1pt) node (b3){};
%   \draw[fill] (0,\l) circle (1pt) node (b1){};
   
   \draw[fill] (\l,\l) circle (1pt) node (b4){};
   \draw[fill] (\l,0) circle (1pt) node (b5){};
   
%   \draw[epi] (b1)--node[anchor=west](x01){}(b0);
     \draw[mono] (b3)--node[anchor=south](x12){}(b0);
   \draw[epi] (b4)--node[anchor=east](x23){}(b5);
       \draw[mono] (b0) -- (b5);
   \end{scope}
\draw[|->] (4.5*\l, 0.5*\l)--node[above](op1){}(5*\l,0.5*\l);   

    \begin{scope}[xshift=6.5*\l]
%   \draw[fill] (0,0) circle (1pt) node (b0){};
%   \draw[fill] (-\l,\l) circle (1pt) node (b2){};
   \draw[fill] (-\l,0) circle (1pt) node (b3){};
%   \draw[fill] (0,\l) circle (1pt) node (b1){};
   
   \draw[fill] (\l,\l) circle (1pt) node (b4){};
   \draw[fill] (\l,0) circle (1pt) node (b5){};
   
%   \draw[epi] (b1)--node[anchor=west](x01){}(b0);
    %  \draw[mono] (b3)--node[anchor=south](x12){}(b0);
   \draw[epi] (b4)--node[anchor=east](x23){}(b5);
       \draw[mono] (b3) -- (b5);
    \end{scope}
\draw[|->] (7.75*\l, 0.5*\l)--node[above](op1){$\cong$} (8.25*\l,0.5*\l);

\begin{scope}[xshift=9.75*\l]
%   \draw[fill] (0,0) circle (1pt) node (b0){};
   \draw[fill] (-\l,\l) circle (1pt) node (b2){};
   \draw[fill] (-\l,0) circle (1pt) node (b3){};
%   \draw[fill] (0,\l) circle (1pt) node (b1){};
   
    \draw[fill] (\l,\l) circle (1pt) node (b4){};
    \draw[fill] (\l,0) circle (1pt) node (b5){};

%   \draw[epi] (b1)--node[anchor=west](x01){}(b0);
    \draw[mono] (b2)--node[anchor=south](x12){}(b4);
    \draw[epi] (b2)--node[anchor=east](x23){}(b3);
    \draw[epi] (b4) -- (b5);
    \draw[mono] (b3)--node[anchor=south](x12){}(b5);
    
    \draw[twoarrowlonger] (-0.8*\l, 0.8*\l) -- (0.8*\l,0.2*\l);
    \end{scope}
\end{tikzpicture} 

 \end{center}
Similarly, we can define two vertical ``compositions'' $\circ_v^1$ and $\circ_v^2$ of squares using composition in $\cV$.

\begin{prop} \label{stableassociative}
Let $\cH$, $\cV$, $\Sq$, $s_v$, $t_v$, $s_h$, $t_h$ be as above. Assume further that
\[\circ_h^1=\circ_h^2 \quad \text{and}\quad \circ_v^1=\circ_v^2.\]
Then this data assembles into a stable double category, i.e., associativity and unitality of both compositions and the interchange law hold.
\end{prop}

\begin{proof}
Note that, by construction, $\circ_h^1$ is compatible with composition of vertical sources and with horizontal source, while $\circ_h^2$ is compatible with composition of vertical targets and with horizontal target. Thus, if $\circ_h^1=\circ_h^2$, this assignment is compatible with all source and target maps, and we can thus think of it as a composition $\circ_h$ that can be iterated. Since composition in $\cH$ is associative, stability implies associativity of composition, as both ways of associating correspond to squares with the same source span. A similar argument shows that $\circ_v=\circ_v^1=\circ_v^2$ is an associative composition compatible with source and target maps. 
 
Stability also implies the interchange law since we are comparing squares with the same source span. It remains to check the existence of horizontal and vertical identity squares. Given a vertical morphism $m$, by stability, there exist two squares as follows:
\begin{center}
\begin{tikzpicture}[baseline=(x01.base)]
    \def\l{1cm}
    \begin{scope}
% \draw (-1.5*\l, 0.5*\l) node {$\circ_h^2:$};

  \draw[fill] (0,0) circle (1pt) node (b0){};
  \draw[fill] (-\l,\l) circle (1pt) node (b2){};
  \draw[fill] (-\l,0) circle (1pt) node (b3){};
  \draw[fill] (0,\l) circle (1pt) node (b1){};
   
  \draw[epi] (b1)--node[anchor=west](x01){$m$} (b0);
  \draw[mono] (b2)--node[anchor=south](x12){$f$}(b1);
  \draw[epi] (b2)--node[anchor=east](x23){$k$}(b3);
  \draw[mono] (b3)--node[anchor=north](x03){$\id$}(b0);
   
  \draw[twoarrowlonger] (-0.8*\l, 0.8*\l)-- node[anchor=east] {$\alpha$} (-0.2*\l,0.2*\l);

  \end{scope}
\end{tikzpicture}
\hspace{1cm}and\hspace{1cm}
\begin{tikzpicture}[baseline=(x01.mid)]
  \def\l{1cm}
    \begin{scope}[xshift=2cm]
% \draw (-1.5*\l, 0.5*\l) node {$\circ_h^2:$};

  \draw[fill] (0,0) circle (1pt) node (b0){};

  \draw[fill] (0,\l) circle (1pt) node (b1){};
   
  \draw[fill] (\l,\l) circle (1pt) node (b4){};
  \draw[fill] (\l,0) circle (1pt) node (b5){} node[xshift=0.25cm] (labelb5){.};

  \draw[epi] (b1)--node[anchor=east](x01){$m$}(b0);

  \draw[epi] (b4) -- node[anchor=west] {$l$} (b5);
  \draw[mono] (b0) -- node [anchor=north] {$g$} (b5);
  \draw[mono] (b1) -- node[anchor=south] {$\id$} (b4); 

  \draw[twoarrowlonger] (0.2*\l, 0.8*\l)-- node[anchor=east,yshift=-0.1cm] {$\beta$} (0.8*\l,0.2*\l);
  \end{scope} 
\end{tikzpicture}
\end{center} 
Composing them, we obtain
\begin{center}
\begin{tikzpicture}[scale=0.5]
    \def\l{2cm}
    \begin{scope}
% \draw (-1.5*\l, 0.5*\l) node {$\circ_h^2:$};

  \draw[fill] (0,0) circle (2pt) node (b0){};
  \draw[fill] (-\l,\l) circle (2pt) node (b2){};
  \draw[fill] (-\l,0) circle (2pt) node (b3){};
  \draw[fill] (0,\l) circle (2pt) node (b1){};
   
  \draw[fill] (\l,\l) circle (2pt) node (b4){};
  \draw[fill] (\l,0) circle (2pt) node (b5){};

%   \draw[epi] (b1)--node[anchor=east](x01){$m$} (b0);
  \draw[mono] (b2)--node[anchor=south](x12){$f$}(b1);
  \draw[epi] (b2)--node[anchor=east](x23){$k$}(b3);
  \draw[mono] (b3)--node[anchor=north](x03){$\id$}(b0);
      
  \draw[epi] (b4) -- node[anchor=west] {$l$} (b5);
  \draw[mono] (b0) -- node [anchor=north] {$g$} (b5);
  \draw[mono] (b1) -- node[anchor=south] {$\id$} (b4); 
  
  \draw[twoarrowlonger] (-0.8*\l, 0.8*\l) -- (0.8*\l,0.2*\l);
   \end{scope} 
\end{tikzpicture}
\end{center}
which is a square whose span of sources is the pair $(k,f)$, so by stability this square must be $\alpha$, and therefore its vertical target $g= g\circ \id$ is the identity $\id$ and its horizontal target $m$ must be $l$. Similarly, its cospan of targets is $(g,l=m)$, so by stability this square must be $\beta$, and its vertical source $f=\id\circ f$ must be the identity and its horizontal source $k$ must be $m$. Finally, stability again shows that $\alpha=\beta$ acts as an identity for $\circ_h$.

A similar argument shows the existence of vertical identities.
\end{proof}

\begin{rmk}\label{rmk stable double functor}
Moreover, if $\cC$ is a double category and $\cD$ is a stable double category, when defining a double functor $F \colon \cC \to \cD$, the assignment on squares is uniquely determined by the assignment on vertical and horizontal morphisms.  However, it is not true that to define a double functor, it is enough to define functors on the horizontal and vertical categories that coincide on objects; the assignment on squares still requires some compatibility between the vertical and horizontal pieces.
\end{rmk}

For double categories which are stable and augmented the necessary data can be reduced even further.

\begin{lem}
\label{bijectionhorizontalvertical}
Let $\cD$ be an augmented stable double category. Then there is a bijection between the set of horizontal arrows and the set of vertical arrows, i.e.,
\[\Hor{\cD}\cong\Ver{\cD}.\]
\end{lem}

\begin{proof}
Given a horizontal morphism $x\mono y$, there is a unique vertical morphism $x\epi a$ with $a$ in the augmentation. Then, by stability, there is a unique square 
\[
  \begin{tikzpicture}[scale=0.6]
    \def\l{2cm}
    \begin{scope}
   \draw[fill] (0,0) node (b0){$z.$};
   \draw[fill] (-\l,\l) node (b2){$x$};
   \draw[fill] (-\l,0) node (b3){$a$};
   \draw[fill] (0,\l) node (b1){$y$};

   \draw[epi] (b1)--node[anchor=west](x01){}(b0);
   \draw[mono] (b2)--node[anchor=south](x12){}(b1);
   \draw[epi] (b2)--node[anchor=east](x23){}(b3);
   \draw[mono] (b3)--node[anchor=north](x03){}(b0);

   \draw[twoarrowlonger] (-0.8*\l, 0.8*\l)--(-0.2*\l,0.2*\l);
   \end{scope}
  \end{tikzpicture}
\]
The horizontal target of this square is a vertical morphism $y\epi z$, which is uniquely determined by the original horizontal morphism $x\mono y$. Note that the vertical target of the square is a horizontal morphism $a \mono z$ which is completely determined by $z$ since $a$ is in the augmentation. Conversely, if we start with a vertical morphism $y\epi z$, we can use a dual argument to recover the horizontal morphism $x\mono y$.
\end{proof}

Recall \cref{thm:PathSpaceCriterion}, which tells us that if $\cC$ is a category, then $\Pplus(N\cC)$ and $\Pminus(N\cC)$ are 1-Segal sets. The following proposition identifies the categories associated to the 1-Segal sets $\Pplus(N \Horz{\cD})$ and $\Pminus(N\Verz{\cD})$.

\begin{prop}\label{augstable are easy}
Let $\cD$ be an augmented stable double category. Then there are isomorphisms of categories
\[\tau_1\Pplus(N\Horz{\cD}) \cong \Horo{\cD} \quad\mbox{and}\quad \tau_1\Pminus(N\Verz{\cD}) \cong \Vero{\cD}\]
such that the following diagrams with the adjoint isomorphisms commute:
\begin{center}
\begin{tikzcd}[column sep=-2pt]
\Pplus(N\Horz{\cD}) \arrow{rr}{\cong} \arrow[swap]{dr}{d_0} & &N\Horo{\cD} \arrow{dl}{s_v} &[7pt] \Pminus(N\Verz{\cD}) \arrow{rr}{\cong} \arrow[swap]{dr}{d_{top}} & &N\Vero{\cD} \arrow{dl}{t_h}\\
& N\Horz{\cD} & & & N\Verz{\cD}.
\end{tikzcd}
\end{center}
Moreover, these isomorphisms are natural in $\cD$.
\end{prop}

\begin{proof}
We prove the statement for the diagram on the left, the other one being similar.  One can check that the set of objects of $\tau_1\Pplus(N\Horz{\cD})$ is precisely the set $\Hor\cD$.  Since the set of objects of $\Horo\cD$ is the set $\Ver\cD$, we have constructed in \cref{bijectionhorizontalvertical} a bijection between the object sets of these two categories.  By inspection, this bijection commutes with the source and target maps to $\Ob{\cD}$.  

Observe that the set of morphisms of $\tau_1\Pplus(N\Horz{\cD})$ is
$$(N\Horz{\cD})_2=\Hor{\cD} \times_{\Ob{\cD}} \Hor{\cD},$$
the set of pairs of composable horizontal morphisms. 
We use the same bijection from \cref{bijectionhorizontalvertical}, together with stability, to obtain a bijection to the sets of morphisms of $\Horo(\cD)$, which is the set $\Sq\cD$:
$$\Hor{\cD} \times_{\Ob{\cD}} \Hor{\cD} \cong \Ver{\cD} \times_{\Ob{\cD}} \Hor{\cD} \xleftarrow{\cong} \Sq\cD.$$
Note that in the first fibered product, the maps to $\Ob{\cD}$ are $(t_h,s_h)$ and in the second one, the maps are $(s_v, s_h)$, as required.
Pictorially, the bijection from \cref{bijectionhorizontalvertical} yields the left-hand vertical arrow in the following picture, and stability yields the displayed square:
\begin{center}
  \begin{tikzpicture}[scale=0.66]
    \def\l{2cm}
    \begin{scope}
   \draw[fill] (0,0) circle (1.5pt) node (b0){};
   \draw[fill] (-\l,\l) circle (1.5pt) node (b2){};
%   \draw[fill] (-\l,0) circle (1pt) node (b3){};
   \draw[fill] (0,\l) circle (1.5pt) node (b1){};
   
   \draw[fill] (\l,\l) circle (1.5pt) node (b4){};
   \draw[fill] (\l,0) circle (1.5pt) node (b5){} node[xshift=0.25cm] (labelb5){.};

   \draw[epi, dashed] (b1)--node[anchor=west](x01){}(b0);
   \draw[mono] (b2)--node[anchor=south](x12){}(b1);
%   \draw[epi, densely dotted] (b2)--node[anchor=east](x23){}(b3);
%   \draw[mono, dashed] (b3)--node[anchor=north](x03){}(b0);
   
   \draw[epi, densely dotted] (b4) -- (b5);
   \draw[mono, densely dotted] (b0) -- (b5);
   \draw[mono] (b1) -- (b4); 

%   \draw[twoarrowlonger, dashed] (-0.8*\l, 0.8*\l)--(-0.2*\l,0.2*\l);
   \draw[twoarrowlonger, densely dotted] (0.2*\l, 0.8*\l)--(0.8*\l,0.2*\l);
   
   \end{scope}
\end{tikzpicture}
\end{center}
These bijections are compatible with source, target, and composition and thus we obtain an isomorphism of categories. Commutativity of the desired diagram can be read off from the pictorial representation.
By inspection, the constructions of the isomorphisms are natural in $\cD$.
\end{proof}

\begin{notn}
We denote by $\pdc$ the category of pointed double categories and pointed double functors. We similarly denote by $\adc$ the category of augmented double categories and augmented double functors. We denote by $\psdc$ and $\asdc$ the full subcategories of $\pdc$ and $\adc$, respectively, consisting of stable objects.
\end{notn}

The aim of this paper is to prove that the categories $\psdc$ and $\asdc$ are equivalent to $\untwoseg_*$ and $\untwoseg$, respectively. 

We conclude this section with an example which is part of a family of augmented stable double categories which will be essential in our definition of the generalized $\sdot$-construction.

\begin{ex}\label{ex: W_2} Consider a double category, denoted by $\cW_2$, with objects $ij$ for $0\leq i \leq j\leq 2$, and generated by the following non-identity horizontal and vertical morphisms and squares:
$$
\begin{tikzpicture}   
   \draw (1,0.3) node(a00){$00$};
\draw (2,0.3) node(a01){$01$};
\draw (2,-0.7) node(a11) {$11$};
\draw (3, -0.7) node(a12){$12$};
\draw (3, 0.3) node(a02){$02$};
\draw (3,-1.7) node (a22){$22.$};

\draw[mono] (a00)--(a01);
\draw[epi] (a01)--(a11);
\draw[mono] (a11)--(a12);
\draw[mono] (a01)--(a02);
\draw[epi] (a12)--(a22);
\draw[epi] (a02)--(a12);

   \draw[twoarrowlonger] (2.2,0.1)--(2.8,-0.5);
\end{tikzpicture}
$$
Since the only non-identity square is uniquely determined by its pair of sources or by its pair of targets, $\cW_2$ is a stable double category.

Furthermore, the set $\aug=\{00, 11, 22\}$ defines an augmentation of $\cW_2$.  For example, consider the object $02$. There is a unique horizontal morphism with target $02$ and source in $\aug$, namely
the composite of the horizontal morphisms from $00$ to $01$ and from $01$ to $02$,
$$\begin{tikzpicture}
\draw (1,0.3) node(a00){$00$};
\draw (3, 0.3) node(a02){$02.$};
\draw[mono] (a00)--(a02);
\end{tikzpicture}$$
Similarly, the vertical morphism from $02$ to $22$ is the unique vertical morphism with source $02$ and target in $\aug$.
\end{ex}

\section{The generalized Waldhausen construction}\label{sec:wald}

We are now ready to define our generalized $\sdot$-construction, which is a functor
\[
\sdot\colon \adc \to \sset,
\]
inspired by Waldhausen's $\sdot$-construction \cite{waldhausen}.  

We begin by describing a cosimplicial object $\cW_{\bullet}$ in augmented stable double categories generalizing \cref{ex: W_2}.  The $\sdot$-construction will be defined by mapping this cosimplicial object into a given augmented stable double category.

\begin{defn}
Given any integer $n\geq 0$, define a double category $\cW_n$ as follows. Its objects are pairs
\[
\left\{ (i,j) \mid 0\leq i\leq j\leq n\right\};
\]
for simplicity of notation, we simply write $ij$ for the pair $(i,j)$.
A horizontal morphism is a triple $(i,j,k)$ with $i\leq j\leq k$, viewed as a map $ij\mono ik$. Similarly, a vertical morphism is also a triple $(i,j,k)$ with $i\leq k\leq j$, viewed as a map $ij\epi kj$. The squares are exactly those of the form
 \begin{center}
  \begin{tikzpicture}[scale=0.6]
    \def\l{2cm}
    \begin{scope}
   \draw[fill] (0,0) node (b0){$kl$};
   \draw[fill] (-\l,\l) node (b2){$ij$};
   \draw[fill] (-\l,0) node (b3){$kj$};
   \draw[fill] (0,\l) node (b1){$il$};

   \draw[epi] (b1)--node[anchor=west](x01){}(b0);
   \draw[mono] (b2)--node[anchor=south](x12){}(b1);
   \draw[epi] (b2)--node[anchor=east](x23){}(b3);
   \draw[mono] (b3)--node[anchor=north](x03){}(b0);

   \draw[twoarrowlonger] (-0.8*\l, 0.8*\l)--(-0.2*\l,0.2*\l);
   \end{scope}
  \end{tikzpicture}
 \end{center}
with $i\leq k\leq j\leq l$.
\end{defn}

The double category $\cW_n$ is augmented by the set $\aug=\{ii \mid 0\leq i\leq n\}$.   

\begin{ex}
 The double categories $\cW_n$ can be pictured for $0 \leq n\leq 4$ as follows:
\begin{center}
 \begin{tikzpicture}[scale=0.9]
  \draw (0,0) node(d0){$\cW_0\colon$};
\draw (1,0) node{$00$};

 \draw (0,-2) node(d1){$\cW_1\colon$};
 \begin{scope}[yshift=-2cm]
\begin{scope}[yshift=0.2cm]
\draw (1,0.3) node(a00){$00$};
\draw (2,0.3) node(a01){$01$};
\draw (2,-0.7) node(a11) {$11$};
\draw[mono] (a00)--(a01);
\draw[epi] (a01)--(a11);
\end{scope}
\end{scope}

 \draw (0,-5) node(d1){$\cW_2\colon $};
 \begin{scope}[yshift=-5cm]
\begin{scope}[yshift=0.7cm]
   \draw (1,0.3) node(a00){$00$};
\draw (2,0.3) node(a01){$01$};
\draw (2,-0.7) node(a11) {$11$};
\draw (3, -0.7) node(a12){$12$};
\draw (3, 0.3) node(a02){$02$};
\draw (3,-1.7) node (a22){$22$};
\draw[mono] (a00)--(a01);
\draw[epi] (a01)--(a11);
\draw[mono] (a11)--(a12);
\draw[mono] (a01)--(a02);
\draw[epi] (a12)--(a22);
\draw[epi] (a02)--(a12);

   \draw[twoarrowlonger] (2.2,0.1)--(2.8,-0.5);
\end{scope}
\end{scope}

 \draw (7,0) node(d1){$\cW_3\colon $};
 \begin{scope}[yshift=0cm, xshift=7cm]

\begin{scope}[yshift=1.1cm]
   \draw (1,0.3) node(a00){$00$};
\draw (2,0.3) node(a01){$01$};
\draw (2,-0.7) node(a11) {$11$};
\draw (3, -0.7) node(a12){$12$};
\draw (3, 0.3) node(a02){$02$};
\draw (3,-1.7) node (a22){$22$};
\draw (4, 0.3) node (a03){$03$};
\draw (4, -0.7) node (a13){$13$};
\draw (4, -1.7) node (a23){$23$};
\draw (4, -2.7) node (a33){$33$};

\draw[mono] (a00)--(a01);
\draw[mono] (a11)--(a12);
\draw[mono] (a01)--(a02);
\draw[mono] (a02)--(a03);
\draw[mono] (a12)--(a13);
\draw[mono] (a22)--(a23);

\draw[epi] (a01)--(a11);
\draw[epi] (a12)--(a22);
\draw[epi] (a02)--(a12);
\draw[epi] (a03)--(a13);
\draw[epi] (a13)--(a23);
\draw[epi] (a23)--(a33);

   \draw[twoarrowlonger] (2.2,0.1)--(2.8,-0.5);
   \draw[twoarrowlonger] (3.2,0.1)--(3.8,-0.5);
   \draw[twoarrowlonger] (3.2,-0.9)--(3.8,-1.5);
\end{scope}

\end{scope}

 \draw (7,-5) node(d1){$\cW_4\colon $};
 \begin{scope}[yshift=-5cm, xshift=7cm]

\begin{scope}[yshift=1.6cm]
   \draw (1,0.3) node(a00){$00$};
\draw (2,0.3) node(a01){$01$};
\draw (2,-0.7) node(a11) {$11$};
\draw (3, -0.7) node(a12){$12$};
\draw (3, 0.3) node(a02){$02$};
\draw (3,-1.7) node (a22){$22$};
\draw (4, 0.3) node (a03){$03$};
\draw (4, -0.7) node (a13){$13$};
\draw (4, -1.7) node (a23){$23$};
\draw (4, -2.7) node (a33){$33$};
\draw (5, 0.3) node (a04){$04$};
\draw (5, -0.7) node (a14){$14$};
\draw (5, -1.7) node (a24){$24$};
\draw (5, -2.7) node (a34){$34$};
\draw (5, -3.7) node (a44){$44.$};

\draw[mono] (a00)--(a01);
\draw[mono] (a11)--(a12);
\draw[mono] (a01)--(a02);
\draw[mono] (a02)--(a03);
\draw[mono] (a12)--(a13);
\draw[mono] (a22)--(a23);
\draw[mono] (a03)--(a04);
\draw[mono] (a13)--(a14);
\draw[mono] (a23)--(a24);
\draw[mono] (a33)--(a34);

\draw[epi] (a01)--(a11);
\draw[epi] (a12)--(a22);
\draw[epi] (a02)--(a12);
\draw[epi] (a03)--(a13);
\draw[epi] (a13)--(a23);
\draw[epi] (a23)--(a33);
\draw[epi] (a04)--(a14);
\draw[epi] (a14)--(a24);
\draw[epi] (a24)--(a34);
\draw[epi] (a34)--(a44);

   \draw[twoarrowlonger] (2.2,0.1)--(2.8,-0.5);
   \draw[twoarrowlonger] (3.2,0.1)--(3.8,-0.5);
   \draw[twoarrowlonger] (3.2,-0.9)--(3.8,-1.5);
   \draw[twoarrowlonger] (4.2,0.1)--(4.8,-0.5);
   \draw[twoarrowlonger] (4.2,-1.9)--(4.8,-2.5);
   \draw[twoarrowlonger] (4.2,-0.9)--(4.8,-1.5);
\end{scope}
\end{scope}
 \end{tikzpicture}
\end{center}

In each case, the pictured arrows and squares generate the double category $\cW_n$, and the augmentation consists of all elements on the diagonal.
\end{ex}

\begin{rmk}
We can alternatively construct $\cW_n$ as a certain sub-double category of a double category of functors as follows.  First, recall the double category of commutative squares in an arbitrary category $\cC$.  The vertical and horizontal categories are both given by $\cC$, and the set of squares is exactly the set of commutative squares in $\cC$.  Composition of squares and units come from composition of morphisms in $\cC$ and from identities in $\cC$, respectively. 

For each $n \geq 0$, consider the double category arising in this way from the category $\Fun([1],[n])$.  Since the categories $\Fun([1],[n])$ form a cosimplicial category as $n$ varies, their respective double categories also assemble to form a cosimplicial double category.

For each $n \geq 0$, let $\cW_n$ be the sub-double category of the double category of commutative squares in $\Fun([1], [n])$ consisting of:
\begin{itemize}
\item all the objects of $\Fun([1],[n])$; 
\item natural transformations which are the identity on the object $1$ of $[1]$ as vertical morphisms;  
\item natural transformations which are the identity on the object $0$ of $[1]$ as horizontal morphisms; and
\item all commutative squares between these morphisms.
\end{itemize}  

One can check that these properties do indeed define a double category.  Moreover, since the
cosimplicial structure comes from the post-composition of morphisms $[1]\to [n]$ with arbitrary morphisms in $\Delta$, it preserves the vertical and horizontal subcategories described above, thus making $\cW_{\bullet}$ into a cosimplicial double category.  Each double category $\cW_n$ has an augmentation given by the constant functors $[1] \to [n]$; note that the cosimplicial structure maps preserve the augmentation.
\end{rmk}

\begin{lem}
The collection $\cW_\bullet$ is a cosimplicial object in augmented stable double categories.
\end{lem}

\begin{proof}
We have shown that each $\cW_n$ is an augmented double category; it remains to show that it is stable. 
If we have a span 
\begin{center}
  \begin{tikzpicture}[scale=0.6]
    \def\l{2cm}
    \begin{scope}
   \draw[fill] (-\l,\l) node (b2){$ij$};
   \draw[fill] (-\l,0) node (b3){$kj$};
   \draw[fill] (0,\l) node (b1){$il$};

   \draw[mono] (b2)--node[anchor=south](x12){}(b1);
   \draw[epi] (b2)--node[anchor=east](x23){}(b3);

   \end{scope}
  \end{tikzpicture}
 \end{center}
and want to produce a square
   \begin{center}
  \begin{tikzpicture}[scale=0.6]
    \def\l{2cm}
    \begin{scope}
  \draw[fill] (0,0) node (b0){$xy,$};
   \draw[fill] (-\l,\l) node (b2){$ij$};
   \draw[fill] (-\l,0) node (b3){$kj$};
   \draw[fill] (0,\l) node (b1){$il$};

   \draw[epi, dashed] (b1)--(b0);
   \draw[mono] (b2)--(b1);
   \draw[epi] (b2)--(b3);
   \draw[mono, dashed] (b3)--(b0);

   \draw[twoarrowlonger, dashed] (-0.8*\l, 0.8*\l)--(-0.2*\l,0.2*\l);
   
   \end{scope}
  \end{tikzpicture}
 \end{center} 
it is necessary that $y=l$ for the right-hand side map to be a vertical morphism, and that $x=k$ for the lower map to be a horizontal morphism. Now there exists a unique square of the necessary form, namely
 \begin{center}
  \begin{tikzpicture}[scale=0.6]
    \def\l{2cm}
    \begin{scope}
   \draw[fill] (0,0) node (b0){$kl,$};
   \draw[fill] (-\l,\l) node (b2){$ij$};
   \draw[fill] (-\l,0) node (b3){$kj$};
   \draw[fill] (0,\l) node (b1){$il$};

   \draw[epi] (b1)--node[anchor=west](x01){}(b0);
   \draw[mono] (b2)--node[anchor=south](x12){}(b1);
   \draw[epi] (b2)--node[anchor=east](x23){}(b3);
   \draw[mono] (b3)--node[anchor=north](x03){}(b0);

   \draw[twoarrowlonger] (-0.8*\l, 0.8*\l)--(-0.2*\l,0.2*\l);
   \end{scope}
  \end{tikzpicture}
 \end{center}
since by properties of the original span we know that $i\leq k\leq j\leq l$. The argument for the cospan case is analogous.
\end{proof}

There are two other cosimplicial objects in double categories which will use later on; by construction one governs the horizontal and the other the vertical category of a double category. We will explain the horizontal case, the vertical one is similar.
\begin{defn}
For every $n \geq 0$, let $\cH_n$ be the (stable) double category whose object set is $\{0,1,\ldots, n\}$, whose horizontal category is $[n]$, whose vertical category is discrete, i.e., has no non-identity morphisms, and whose squares are only the identity squares. Similarly, let $\cV_n$ be the (stable) double category given by switching the horizontal and vertical directions in $\cH_n$.
\end{defn}

In what follows, we use the following property of $\cH_n$ and $\cV_n$, whose proof is left to the reader.

\begin{lem}
\label{lem:lemforPplusminusSdot} 
Ranging over all objects $[n]$ of $\Delta$, we obtain a cosimplicial stable double category $\cH_{\bullet}$.   Given a double category $\cD$, any double functor $\cH_n\to \cD$ is determined uniquely by a functor $[n]\to \Horz{\cD}$, and this identification is compatible with the cosimplicial structure. We obtain bijections
 \[
\Hom_{\dc}(\cH_n, \cD) \cong \Hom_{\cat}([n],\Horz{\cD})
 \]
which are natural in $n$, and thus assemble to an isomorphism of simplicial sets
\[
\Hom_{\dc}(\cH_{\bullet}, \cD) \cong N\Horz{\cD}.
\]
The analogous statements hold for $\cV_\bullet$ and $\Verz{\cD}$.%\hfill{$\square$}
\end{lem}

Recall from \cref{defn:pathspace} that the path object functor $\Pplus$ is defined using the join construction $[n] \mapsto [0]*[n]$.  Although $[0]*[n] \cong [n+1]$, it is preferable here to think of adjoining an extra object $0'$ and using the ordering 
\[
 [0]*[n]=\{0'\leq 0 \leq 1 \leq \ldots \leq n\}.
\] 
To clarify notation, we write $W([0]*[n])$ rather than $W_{n+1}$ for the corresponding augmented stable double category. 

Using this notation, we define a double functor 
\[ j_n\colon \cH_n\to \cW([0]*[n]) \]
by $j_n(k \leq l)=(0'k\mono 0'l)$. For example, for $n=2$, we indicate by boldface the image of $j_2$ in the diagram
\begin{center}
 \begin{tikzpicture}[scale=1.2]
\begin{scope}
   \draw (1,0.3) node(a00){$0'0'$};
\draw (2,0.3) node(a01){$\mathbf{0'0}$};
\draw (2,-0.7) node(a11) {$00$};
\draw (3, -0.7) node(a12){$01$};
\draw (3, 0.3) node(a02){$\mathbf{0'1}$};
\draw (3,-1.7) node (a22){$11$};
\draw (4, 0.3) node (a03){$\mathbf{0'2}$};
\draw (4, -0.7) node (a13){$02$};
\draw (4, -1.7) node (a23){$12$};
\draw (4, -2.7) node (a33){$22.$};

\draw[mono] (a00)--(a01);
\draw[mono] (a11)--(a12);
\draw[mono, thick] (a01)--(a02);
\draw[mono, thick] (a02)--(a03);
\draw[mono] (a12)--(a13);
\draw[mono] (a22)--(a23);

\draw[epi] (a01)--(a11);
\draw[epi] (a12)--(a22);
\draw[epi] (a02)--(a12);
\draw[epi] (a03)--(a13);
\draw[epi] (a13)--(a23);
\draw[epi] (a23)--(a33);

   \draw[twoarrowlonger] (2.2,0.1)--(2.8,-0.5);
   \draw[twoarrowlonger] (3.2,0.1)--(3.8,-0.5);
   \draw[twoarrowlonger] (3.2,-0.9)--(3.8,-1.5);
\end{scope}
 \end{tikzpicture}
\end{center} 
Observe that the double functors $j_n$ are natural in $n$, and  since $0'$ remains unchanged under the cosimplicial structure maps of $\cW([0]* [n])$, the functors $j_n$ indeed assemble to give a cosimplicial double functor $j_\bullet$.

We are now ready to use $\cW_\bullet$ to define the generalized $\sdot$-construction.

\begin{defn}\label{defn:sdot}
The \emph{Waldhausen $\sdot$-construction} is the functor
\[ \sdot \colon \adc \rightarrow \sset \] 
which %\mrnote{that?} \jbnote{Martina, you are even fussier about grammar than me!  I don't think it matters that much.}
takes an augmented double category $\cD$ to the simplicial set $\sdot(\cD)$ given by 
 \[
  \sdotn{n}(\cD):=\Hom_{\adc}(\cW_n, \cD).
 \]
\end{defn}

The first half of our main theorem is given by the following result.

\begin{thm}\label{prop: first half}
The Waldhausen $\sdot$-construction restricts to functors
$$\sdot \colon \asdc \rightarrow \untwoseg $$
and
$$\sdot \colon \psdc \rightarrow \untwoseg_* .$$
\end{thm}

We prove this theorem via two propositions. We first need to show that, for any augmented stable double category $\cD$,  the simplicial set $\sdot(\cD)$ is $2$-Segal.  We do so by using the Path Space Criterion (\cref{thm:PathSpaceCriterion}), and the fact that a $1$-Segal set is just the nerve of a category (\cref{prop:1segalnerve}). 

\begin{prop}\label{prop:PplusminusSdot}
 Let $\cD$ be an augmented stable double category. Then there are isomorphisms of simplicial sets
 \[
  \Pplus\sdot(\cD)\to N\Horz{\cD} \mbox{ and } \Pminus\sdot(\cD)\to N\Verz{\cD}.
 \]
 In particular, $\Pplus \sdot(\cD)$ and $\Pminus \sdot(\cD)$ are 1-Segal sets, and hence $\sdot(\cD)$ is a 2-Segal set.
\end{prop}

\begin{proof}
We establish the first isomorphism; the second one is analogous.

If $\cD$ is a augmented stable double category, then precomposition with the cosimplicial double functor $j_\bullet$ gives a map of simplicial sets
\[
 \Pplus\sdot(\cD)=\Hom_{\adc}(\cW([0]* \bullet), \cD)\to \Hom_{\dc}(\cH_{\bullet},\cD).
\]
By \cref{{lem:lemforPplusminusSdot}}, it is enough to prove that this map is an isomorphism of simplicial sets, which can be accomplished by proving that an augmented double functor $F \colon \cW([0]*[n]) \rightarrow \cD$ is uniquely determined by its restriction to the image of $j_n$.  To prove this statement for a fixed $n$, we use induction on $k$ to show that the row labeled by $k \in [0]*[n]$ is determined by the image of $j_n$.  For notational purposes, if $k=0$, then $k-1$ is taken to be $0'$. 

For the base case $k=0'$, observe that if we start with the image of $j_n$ only, then to obtain the rest of the $0'$ row we need only adjoin a horizontal map to be the image of the map $0'0'\mono 0'0$.   Since $\cD$ is an augmented double category and we want $F$ to be an augmented functor, the image of $0'0'\mono 0'0$ must be the unique horizontal morphism $a\mono F(0'0)$ with $a$ in the augmentation set.

Now, let $0\leq k < n$, and assume that $F$ is uniquely defined up to the row labeled by $k-1$. To complete the row labeled by $k$, we first send the map $(k-1)k\epi kk$ to the unique vertical morphism $F((k-1)k)\epi a'$ in $\cD$ with $a' \in \aug$ given by the augmentation.

Given $k+1\leq \ell \leq n$, to define the images of the objects $k\ell$, the morphisms $(k-1)\ell\epi k\ell$ and $k(\ell -1) \mono k\ell$, and the square with these morphisms as targets, we proceed by induction on $\ell$.  By stability in $\cD$, there is a unique square
\begin{center}
  \begin{tikzpicture}[scale=1.2]
    \def\l{2cm}
    \def\w{3cm}
    \begin{scope}
   \draw[fill] (0,0) node[minimum width=1cm, minimum height=0.5cm] (b0){$x,$};
   \draw[fill] (-\w,\l) node[minimum width=1cm, minimum height=0.5cm]  (b2){$F((k-1)(\ell -1))$};
   \draw[fill] (-\w,0) node[minimum width=1cm, minimum height=0.5cm]  (b3){$F(k(\ell -1))$};
   \draw[fill] (0,\l) node[minimum width=1cm, minimum height=0.5cm]  (b1){$F((k-1)\ell)$};

   \draw[epi] (b1)--node[anchor=west](x01){}(b0);
   \draw[mono] (b2)--node[anchor=south](x12){}(b1);
   \draw[epi] (b2)--node[anchor=east](x23){}(b3);
   \draw[mono] (b3)--node[anchor=north](x03){}(b0);

   \draw[twoarrowlonger] (-0.7*\w, 0.7*\l)--(-0.3*\w,0.3*\l);
   \end{scope}
  \end{tikzpicture}
 \end{center}
 which we use to extend $F$. 
 
Lastly, if $k=n$, the image of $(n-1)n \epi nn$ is uniquely specified by the augmentation.
 
We have constructed $F$ on the generating horizontal and vertical morphisms, and on the generating squares of $W([0]* [n])$, which is enough to define a double functor.  Thus, we have established the desired isomorphism. 

The final statement of the proposition follows from the fact that nerves of categories are 1-Segal sets; the fact that $\sdot(\cD)$ is 2-Segal then follows from the Path Space Criterion.
\end{proof}

Now that we have proved that the image of any augmented stable double category under the $\sdot$-construction is a $2$-Segal set, it remains to prove that it is also unital.

\begin{prop}\label{lem:UnitalitySdot} 
Let $\cD$ be an augmented stable double category. Then $\sdot(\cD)$ is unital.
\end{prop}

\begin{proof}
By \cref{prop:PplusminusSdot} and \cref{lem:easyunital}, to prove that $\sdot(\cD)$ is unital it is enough to check that 
\begin{center}
\begin{tikzcd}
\sdotn{1}(\cD) \arrow[r, "d_0"] \arrow[d, "s_1"'] & \sdotn{0}(\cD) \arrow[d, "s_0"]\\
\sdotn{2}(\cD) \arrow[r, "d_0"']  & \sdotn{1}(\cD)
\end{tikzcd}
\quad\mbox{and}\quad
\begin{tikzcd}
\sdotn{1}(\cD) \arrow[r, "d_1"] \arrow[d, "s_0"'] & \sdotn{0}(\cD) \arrow[d, "s_0"]\\
\sdotn{2}(\cD) \arrow[r, "d_2"']  & \sdotn{1}(\cD)
\end{tikzcd}
\end{center}
are pullbacks. We prove that the square on the right is a pullback; the argument for the one on the left is similar.

An arbitrary element of the pullback is given by a pair $(a,F)$, where $a\in \aug=\sdotn{0}(\cD)$ and $F$ is an element of $\sdotn{2}(\cD)$, namely, an augmented double functor $\cW_{2}\to\cD$, which is of the form
\begin{center}
\begin{tikzpicture}[scale=1.4]
\draw (1,0.3) node(a00){$a$};
\draw (2,0.3) node(a01){$a$};
\draw (2,-0.7) node(a11) {$a$};
\draw (3, -0.7) node(a12){$F(12)$};
\draw (3, 0.3) node(a02){$F(02)$};
\draw (3,-1.7) node (a22){$F(22).$};

\draw[mono] (a00)-- node[anchor=south] {\scriptsize$\id_a^h$} (a01);
\draw[epi] (a01)-- node[anchor=east] {\scriptsize$\id_a^v$}(a11);
\draw[mono] (a11)--(a12);
\draw[mono] (a01)--(a02);
\draw[epi] (a12)--(a22);
\draw[epi] (a02)--(a12);

\draw[twoarrowlonger] (2.2,0.1)--(2.8,-0.5);
\end{tikzpicture}
\end{center}
By stability of $\cD$, the only square with horizontal source $\id_a^v$ is a vertical identity morphism, so $F(02)=F(12)$ and $F$ is of the form
\begin{center}
\begin{tikzpicture}[scale=1.4]
\draw (1,0.3) node(a00){$a$};
\draw (2,0.3) node(a01){$a$};
\draw (2,-0.7) node(a11) {$a$};
\draw (3, -0.7) node(a12){$F(12)$};
\draw (3, 0.3) node(a02){$F(12)$};
\draw (3,-1.7) node (a22){$F(22),$};

\draw[mono] (a00)-- node[anchor=south] {\scriptsize$\id_a^h$} (a01);
\draw[epi] (a01)-- node[anchor=east] {\scriptsize$\id_a^v$}(a11);
\draw[mono] (a11)--(a12);
\draw[mono] (a01)--(a02);
\draw[epi] (a12)--(a22);
\draw[epi] (a02)-- node[anchor=west] {\scriptsize$\id_{F(12)}^v$} (a12);

\draw[twoarrowlonger] (2.2,0.1)--(2.8,-0.5);
\end{tikzpicture}
\end{center}
and therefore $F$ is uniquely determined by $d_0F\in \sdotn{1}(\cD)$.
\end{proof}

Now we combine the previous results to show that $\sdot$ indeed defines the desired functor.
\begin{proof}[Proof of \cref{prop: first half}]
 \cref{prop:PplusminusSdot} implies by \cref{thm:PathSpaceCriterion} that $\sdot(\cD)$ is a $2$-Segal set for any augmented stable double category $\cD$. \cref{lem:UnitalitySdot} shows that  $\sdot(\cD)$ is also unital.  Functoriality is immediate by the definition of $\sdot$.

Finally, if $\D$ is a stable pointed double category, i.e., the augmentation set $\aug$ consists of exactly one point, then $\sdot \D$ is a reduced unital $2$-Segal set, since in this case $\sdotn{0}\D$ is a single point.
\end{proof}

\section{The path construction}\label{sec:path}

In this section we construct a functor in the other direction,
\[\cP \colon \untwoseg \to \asdc.\]
We first describe a construction on a unital 2-Segal set, then prove that its output is an augmented stable double category using \cref{stableassociative}. Finally, we establish that this assignment is functorial.

\begin{const}\label{const:path}
Let $X$ be a unital 2-Segal set. By \cref{thm:PathSpaceCriterion}, the path spaces $\Pplus X$ and $\Pminus X$ are 1-Segal, and hence, by \cref{fundcat}, we have categories $\cH = \tau_1 \Pplus X$ and $\cV = \tau_1 \Pminus X$, both with $X_1$ and $X_2$ as the sets of objects and morphisms, respectively. The source, target, and identity maps for $\cH$ are given by $d_2$, $d_1$, and $s_1$, respectively, and for $\cV$, they are given by $d_1$, $d_0$, and $s_0$. Composition  in $\cH$ and $\cV$ can be defined by
\[
X_2\times_{X_1} X_2 \underset{\cong}{\xrightarrow{(d_3,d_1)^{-1}}}  X_3 \xrightarrow{d_2}   X_2 \text{ and } X_2\times_{X_1} X_2 \underset{\cong}{\xrightarrow{(d_2,d_0)^{-1}}} X_3 \xrightarrow{d_1}   X_2,
\]
respectively. 

Let $\Sq = X_3$. We define the horizontal source and the horizontal target of a square by using the face maps $d_3$ and $d_2$, respectively, as shown in the diagram
$$\xymatrix{\Mor \cV\ar@{=}[d]&\Sq \ar@{-->}[l]_-{s_h}\ar@{-->}[r]^-{t_h}\ar@{=}[d]&\Mor \cV \ar@{=}[d]\\X_2&X_3\ar[l]^-{ d_3}\ar[r]_-{ d_2}&X_2,}$$
and the vertical source and vertical target of a square using the face maps $d_1$ and $d_0$, respectively, as shown in the diagram
$$\xymatrix{\Mor \cH \ar@{=}[d]&\Sq \ar@{-->}[l]_-{s_v}\ar@{-->}[r]^-{t_v}\ar@{=}[d]&\Mor \cH \ar@{=}[d]\\X_2&X_3\ar[l]^-{ d_1}\ar[r]_-{ d_0}&X_2.}$$
More explicitly, for a $3$-simplex $x\in X_3$ whose 2-dimensional faces are $m,m', e,$ and $e'$ as given in the diagram
\begin{center}
 \begin{tikzpicture}[scale=0.7]
 \def\l{3cm}
    \begin{scope} 
   \draw[fill] (-\l,0) node (b0){$0$};
   \draw[fill] (0,0)  node (b1){$1$};
   \draw[fill] (0,\l) node (b2){$2$};
   \draw[fill] (-\l,\l)  node (b3){$3$};

   \draw (b1)--(b2);
   \draw (b2)--(b3);
   \draw (b3)--(b0);
   \draw (b1)--(b0);
   \draw (b1)--(b3);
   
   \draw[epi] (-0.95*\l, 0.45*\l)-- node[below]{$e'$}(-0.55*\l, 0.45*\l);
   \draw[mono] (-0.05*\l, 0.55*\l)--node[above]{$m'$}(-0.45*\l, 0.55*\l);
   \end{scope}

    \begin{scope} [xshift=4cm]
   \draw[fill] (-\l,0) node (b0){$0$};
   \draw[fill] (0,0)  node (b1){$1,$};
   \draw[fill] (0,\l) node (b2){$2$};
   \draw[fill] (-\l,\l)  node (b3){$3$};

   \draw (b1)--(b2);
   \draw (b2)--(b3);
   \draw (b3)--(b0);
   \draw (b1)--(b0);
   \draw (b0)--(b2);
   
    \draw[epi] (-0.45*\l, 0.45*\l)-- node[below]{$e$}(-0.05*\l, 0.45*\l);
   \draw[mono] (-0.55*\l, 0.55*\l)--node[above]{$m$}(-0.95*\l, 0.55*\l);
   \end{scope}
   
 \end{tikzpicture}
\end{center}
the corresponding square in $\Sq$ is given by
\begin{center}
 \begin{tikzpicture}
 \def\l{3cm}
 \def\h{3cm}
 %nodes
   \draw (0,0) node (a1){$d_2d_1(x)$};
   \draw (\l,0) node (a2){$d_1d_1(x)$};
   \draw(-0.3*\l,-0.3*\h) node (b1){$d_1d_3(x)$};
   \draw(1.3*\l,-0.3*\h) node (b2){$d_1d_2(x)$};
      \draw(-0.3*\l,-1*\h) node (c1){$d_0d_3(x)$};
   \draw(1.3*\l,-1*\h) node (c2){$d_0d_2(x)$};
      \draw (0,-1.3*\h) node (d1){$d_2d_0(x)$};
   \draw (\l,-1.3*\h) node (d2){$d_1d_0(x).$};
   %middle
   \draw (0.6*\l, -0.6*\h)node(middle){$x$};
   \draw[twoarrowlonger] (0.3*\l, -0.4*\h)--(0.7*\l, -0.9*\h);
  %arrows
  \draw[mono] (a1)--node[above](u){$m=d_1(x)$} (a2);
  \draw[epi] (b1)--node[left](l){$e=d_3(x)$} (c1);
  \draw[epi] (b2)--node[right](r){$e'=d_2(x)$} (c2);
  \draw[mono] (d1)--node[below](b){$m'=d_0(x)$} (d2);
  %equalities
  \draw[double, double distance=1pt] (a1)--(b1);
  \draw[double, double distance=1pt] (a2)--(b2);
  \draw[double, double distance=1pt] (c1)--(d1);
  \draw[double, double distance=1pt] (c2)--(d2);
 \end{tikzpicture}
\end{center}
Note that this particular description of $x$ also establishes that the horizontal source and target are compatible with the source and target in $\cH$, and analogously for $\cV$.
\end{const}

\begin{prop}
Let $X$ be a unital 2-Segal set. Then the data from \cref{const:path} above defines a stable double category $\pcat X$. Furthermore, the inclusion $s_0X_0 \subseteq X_1$ defines an augmentation for $\pcat X$.
\end{prop}

\begin{proof}
To apply \cref{stableassociative} we need to show that:
\begin{enumerate}
\item condition \eqref{eqn_stability} holds, and

\item the two horizontal ``compositions'' of squares from (\ref{h1h2}) agree with each other, and similarly, the two analogous vertical ``compositions'' of squares  coincide.
\end{enumerate}

Condition \eqref{eqn_stability} follows from 2-Segality of $X$: every $3$-simplex of $X$ is completely determined by the pair $(d_2x, d_0x)$ or by the pair $(d_3x, d_1x)$, which in turn are the cospan of targets, or the span of sources, 
$$\xymatrix{
\ar@{=}[d]\Mor \cV \times_{\Ob} \Mor \cH &\ar@{=}[d]\Sq\ar[r]^-{(s_h,s_v)}\ar[l]_-{(t_h,t_v)}&\ar@{=}[d]\Mor\cV\times_{\Ob}\Mor\cH\\
X_2\times_{X_1}X_2&X_3\ar[r]^-{( d_3, d_1)}_-{\cong}\ar[l]^-{\cong}_-{( d_2, d_0)}&X_2\times_{X_1}X_2.}$$
Thus, once we show that $\pcat X$ is indeed a double category, we know that it is stable.

For horizontal composition, consider the following diagram. Starting at $X_3\times_{X_2} X_3$ at the left of the diagram and following along the top to $X_3$ is the first horizontal composition $\circ_h^1$ as defined in (\ref{h1h2}); following along the bottom to $X_3$ is the second horizontal composition $\circ_h^2$. 
\begin{equation}\label{octagon}
\begin{tikzpicture}[baseline=(x.base)]
\newdimen\R
\R=4.5cm
\draw (0,0) node (x) {$X_4$}
(0:\R) node (x_0) {$X_3$}
(45: \R) node (x_1) {$X_2\times_{X_1} X_2$}
(90: \R) node (x_2) {$X_2 \times_{X_1} X_3$}
(135: \R) node (x_3) {$(X_2\times_{X_1} X_2) \times_{X_1} X_2$}
(180: \R) node (x_4) {$X_3\times_{X_2} X_3$}
(225: \R) node (x_5) {$X_2\times_{X_1} (X_2\times_{X_1} X_2)$}
(270: \R) node (x_6) {$X_3\times_{X_1} X_2$}
(315: \R) node (x_7) {$X_2\times_{X_1} X_2$}
;
\draw[->] (x) -- node[above]{\scriptsize$d_3$}(x_0) ;
%\draw[->] (x) -- (x_1) ;
\draw[->] (x) -- node[left]{\scriptsize$(d_3d_4, d_1)$}(x_2) ;
%\draw[->] (x) -- (x_3) ;
\draw[->] (x) -- node[anchor=south] {\scriptsize$(d_4, d_2)$} (x_4) ;
%\draw[->] (x) -- (x_5) ;
\draw[->] (x) -- node[left]{\scriptsize$(d_0,d_2d_3)$}(x_6) ;
%\draw[->] (x) -- (x_7) ;

\draw[->] (x_4) -- node[anchor=east] {\scriptsize$(d_3,d_1)\times d_1$} (x_3);
\draw[->] (x_4) -- node[anchor=east] {\scriptsize$d_0\times (d_0, d_2)$} (x_5);
\draw[->] (x_2) -- node[anchor=north] {\scriptsize$\cong$} node[anchor=south, xshift=-0.5cm]{\scriptsize$\id\times(d_3,d_1)$}(x_3);
\draw[->] (x_2) -- node[above, xshift=0.2cm]{\scriptsize$\id\times d_2$}(x_1);
\draw[->] (x_0) -- node[anchor=west] {\scriptsize$\cong$} node[anchor=east]{\scriptsize$(d_3,d_1)$}(x_1);
\draw[->] (x_6) -- node[anchor=north, xshift=-0.7cm]{\scriptsize$(d_3,d_1) \times \id$}  node[anchor=south]{\scriptsize$\cong$}(x_5);
\draw[->] (x_6) --node[below, xshift=0.2cm]{\scriptsize$d_2\times \id$} (x_7);
\draw[->] (x_0) -- node[anchor=west] {$(d_0,d_2)$} node[anchor=east]{\scriptsize$\cong$}(x_7);

\end{tikzpicture}
\end{equation}

The 2-Segality of $X$ implies that the map $(d_4,d_2) \colon X_4\to X_3\times_{X_2}X_3$ is a bijection. Indeed, we have the commutative diagram 
\[
\xymatrixcolsep{5pc}\xymatrix{
X_4\ar[r]^-{(d_4,d_2)}\ar[d]_-{(d_0d_4,d_2d_4,d_1d_2)}^-{\cong} & X_3\times_{X_2} X_3\ar[d]_-{\cong}^-{ (d_0,d_2) \times (d_3,d_1)}\\
X_2\times_{X_1} X_2\times_{X_1} X_2 \ar[r]^-{\cong} & (X_2 \times_{X_1} X_2) \times_{X_2} (X_2\times_{X_1} X_2)}
\]
where the two vertical maps are bijections because $X$ is 2-Segal, and the bottom map is an isomorphism by a general limit argument.  The left vertical map is the $\cT$-Segal map corresponding to the following triangulation of the pentagon:
$$\begin{tikzpicture}[scale=1.5]
\draw (0:1) node[dot] (0) {} node[anchor=west] {0}
-- (72: 1) node[dot] (1) {} node[anchor=south west] {1}
-- (144:1) node[dot] (2) {} node[anchor=south east] {2}
-- (216:1) node[dot] (3) {} node[anchor=north east] {3}
-- (288:1) node[dot] (4) {} node[anchor= north west] {4.}
-- cycle;
\draw (0) -- (3) -- (1);
\end{tikzpicture}$$

We have shown that the diagram in \eqref{octagon} commutes and thus, that $$\circ_h=\circ_h^1=\circ_h^2 :X_3\times_{X_2} X_3 \underset{\cong}{\xrightarrow{(d_4,d_2)^{-1}}} X_4 \xrightarrow{d_3} X_3.$$

A similar argument shows that the two vertical compositions coincide and are given by
$$\circ_v=\circ_v^1=\circ_v^2 :X_3\times_{X_2} X_3 \underset{\cong}{\xrightarrow{(d_2,d_0)^{-1}}} X_4 \xrightarrow{d_1} X_3.$$

Lastly, to prove that $\cP X$ is augmented by $s_0X_0$, note that unitality of $X$ implies that the inner squares are pullbacks in the diagrams 
$$\xymatrix{
\Ob\ar@{-->}[rrr]\ar@{-->}[ddd]\ar@{=}[rd]&&&\aug\ar@{^{(}->}[ddd]\\
&X_1\ar[r]^-{d_1}\ar[d]_-{ s_0}&X_0\ar[d]^-{ s_0}\ar[ru]^-{s_0}_-{\cong}&\\
&X_2\ar[r]_-{d_2}&X_1&\\
\Mor\cH\ar@{=}[ru]\ar[rrr]_-{s_h}&&&\Ob\ar@{=}[lu]}\quad 
\xymatrix{
\Ob\ar@{-->}[rrr]\ar@{-->}[ddd]\ar@{=}[rd]&&&\aug\ar@{^{(}->}[ddd]\\
&X_1\ar[r]^-{d_0}\ar[d]_-{ s_1}&X_0\ar[d]^-{ s_0}\ar[ru]^-{s_0}_-{\cong}&\\
&X_2\ar[r]_-{d_0}&X_1&\\
\Mor\cV\ar@{=}[ru]\ar[rrr]_-{t_v}&&&\Ob\ar@{=}[lu]}
$$
from which it follows that the outer diagrams are pullbacks as well. The simplicial identities $d_1s_0=\id$ and $d_1s_1=\id$ show that the vertical maps interact appropriately with the target in $\Mor \cH$ and the source in $\Mor \cV$. The conclusion thus follows from \cref{augpb}.
\end{proof}

We conclude by proving that the path construction we have defined in this section is functorial.

\begin{prop}\label{prop: second half}
The construction $\pcat$ defines a functor
$$\pcat \colon \untwoseg\to\asdc,$$
which restricts to
$$\pcat \colon \untwoseg_*\to\psdc.$$
\end{prop}

\begin{proof}
Let $f \colon X \rightarrow X'$ be a map between unital $2$-Segal sets. For ease of notation, we denote by $\Horz$, $\Verz$, and $\Sq$ the horizontal category of objects, the vertical category of objects, and the set of squares of $\pcat X$, and by $\Horz '$, $\Verz '$, and $\Sq '$ the corresponding ones for $\pcat X'$.

Since $\Pplus$ and $\Pminus$ are functors, we obtain maps of simplicial sets $\Pplus X\to  \Pplus X'$ and $\Pminus X\to  \Pminus X'$, which in turn uniquely determine functors
$$\Horz \xrightarrow{f_{h}} \Horz' \quad\text{and}\quad \Verz \xrightarrow{f_{v}} \Verz'$$
on the horizontal and vertical categories of objects, respectively.  Moreover, both $f_{h}$ and $f_{v}$ agree with $f_1$ on the set of objects $\Ob(\pcat X)=X_1$. Note that the assignment is compatible with the augmentation, since $f_1s_0=s'_0f_0$.  The component $f_3 \colon X_3\to X_3'$ determines a function $\Sq\to\Sq'$ between the sets of squares.

It remains to check compatibility of the assignment with vertical and horizontal source and target, identities, and composition.  All the arguments are similar, depending only on the compatibility of $f_3$ with the simplicial maps; to illustrate, we establish compatibility with composition using the commutativity of the following diagram: 
\[\xymatrix@C=1.9pc{\ar[ddd]_-{\pcat f|_{\Sq}\times\pcat f|_{\Sq}}\Sq\times_{\Ver} \Sq\ar@{=}[rd]\ar[rrrr]^-{\circ_h}&&&&\ar[ddd]^-{\pcat f|_{\Sq}}\Sq\ar@{=}[ld]\\
&X_3\times_{X_2} X_3\ar[d]_-{f_3\times f_3}&X_4\ar[d]|-{f_4}\ar[l]^-{\cong}_-{( d_4, d_2)}\ar[r]^-{ d_3}&X_3\ar[d]^-{f_3}&\\
&X'_3\times_{X'_2} X'_3&X'_4\ar[l]_-{\cong}^-{( d'_4, d'_2)}\ar[r]_-{ d'_3}&X'_3&\\
\ar@{=}[ru]\Sq'\times_{\Ver'} \Sq'\ar[rrrr]_-{\circ'_h}&&&&\ar@{=}[lu]\Sq'.}\]

Note that the restriction to reduced unital 2-Segal sets has image in pointed stable double categories and pointed double functors.
\end{proof}

\section{The equivalence}

In this section we prove our main theorem, which states that the functors $\sdot$, defined in \cref{defn:sdot}, and $\pcat$, defined in \cref{const:path}, are inverse to one another.

\begin{thm}\label{thm:main_thm}
The functors $\sdot$ and $\pcat$ define an equivalence of categories 
\[ \sdot \colon \asdc \overset{\simeq}\longleftrightarrow \untwoseg \colon \pcat \]
which restricts to an equivalence of categories
\[ \sdot \colon \sdc_* \overset{\simeq}\longleftrightarrow \untwoseg_* \colon \pcat. \]
\end{thm}

We need to show that there are natural isomorphisms
$$\id_{\untwoseg} \overset{\cong}{\longrightarrow} \sdot \pcat \quad\mbox{and}\quad \pcat \sdot \overset{\cong}{\longrightarrow} \id_{\asdc}.$$
The first isomorphism is proven in \cref{prop:eta} and the second in \cref{prop:equiv2}.

\begin{prop}\label{prop:eta}
Let $X$ be a unital 2-Segal set. There is a natural isomorphism of simplicial sets
$$\eta_X \colon X\xrightarrow{\,\,\cong\,\,}\sdot(\pcat X).$$
\end{prop}

\begin{proof}
By \cref{prop:3iso}, since $X$ and $\sdot(\pcat X)$ are 2-Segal sets, it is enough to define $\eta_X$ at the level of 0-, 1-, 2-, and 3-simplices, and prove that it is an isomorphism at levels 0, 1, and 2.

Recall that $\sdotn{0}(\pcat X)$ is the augmentation of $\pcat X$, which by construction is the subset $s_0 X_0 $ of $X_1$. Thus, we can define $\eta_X$ at the level of 0-simplices as 
\[s_0\colon X_0 \to \sdotn{0}(\pcat X)=s_0X_0.\]
It is an isomorphism since $s_0$ is injective.

We define $\eta_X$ on 1-simplices as
\begin{center}
\begin{tikzpicture}
\draw (0,-2) node(x1){$x\in X_1$};
\draw (2,-2) node(s1){};
 \begin{scope}[yshift=-1.2cm, xshift=0.5cm]
\draw (2,0) node(a00){$s_0d_1x$};
\draw (4,0) node(a01){$x$};
\draw (4,-1.5) node(a11) {$s_0d_0x.$};
\draw[mono] (a00)--node[above]{$s_0x$}(a01);
\draw[epi] (a01)--node[right]{$s_1x$}(a11);
\end{scope}
\draw[mapstikz] (x1)--node[below]{}(s1);
\end{tikzpicture}
\end{center}
\noindent One can check that this output is an allowed 1-simplex in $\sdot(\pcat X)$, and that the sources and targets are as stated. Moreover, it is not hard to see that this map is injective. The fact that $\pcat X$ is an augmented double category, which follows from the unitality of $X$, implies surjectivity.

We define $\eta_X$ on 2-simplices as
\begin{center}
\begin{tikzpicture}
\draw (0, -5.7) node(x2){$x\in X_2$};
\draw (2,-5.7)node(s2){};
\draw[mapstikz](x2)--node[below]{}(s2);
 \begin{scope}[yshift=-5cm, xshift=2cm, scale=1.5]
   \draw (0.5,0.3) node(a00){$s_0d_1d_2x$};
\draw (2,0.3) node(a01){$d_2x$};
\draw (2,-0.7) node(a11) {$s_0d_1d_0x$};
\draw (3.5, -0.7) node(a12){$d_0x$};
\draw (3.5, 0.3) node(a02){$d_1x$};
\draw (3.5,-1.7) node (a22){$s_0d_0d_0x.$};
\draw[mono] (a00)--node[above]{$s_0d_2x$}(a01);
\draw[epi] (a01)--node[left]{$s_1d_2x$}(a11);
\draw[mono] (a11)--node[below]{$s_0d_0x$}(a12);
\draw[mono] (a01)--node[above]{$x$}(a02);
\draw[epi] (a12)--node[right]{$s_1d_0x$}(a22);
\draw[epi] (a02)--node[right]{$x$}(a12);
   
   %filling the square
   \draw[twoarrowlonger] (2.2,0.1)--node[above]{$s_1x$}(3.3,-0.5); 
\end{scope}
\end{tikzpicture}
\end{center}
\noindent Again, it is routine to check that this diagram defines an allowed 2-simplex in $\sdot(\pcat X)$. Injectivity is immediate, and surjectivity follows from the fact that $\pcat X$ is stable and augmented.

Finally, we define $\eta_X$ on 3-simplices as 

{\small{
\begin{center}
 \begin{tikzpicture}[scale=0.8]

\draw (0, -11) node(x3){$x\in X_3$};
\draw (2,-11)node(s3){};
\draw[mapstikz](x3)--node[below]{}(s3);
 \begin{scope}[yshift=-10cm, xshift=0cm, scale=2.6]

   \draw (1,0.3) node(a00){$s_0d_1d_2d_3x$};
\draw (2,0.3) node(a01){$d_2d_3x$};
\draw (2,-0.7) node(a11) {$s_0d_1d_0d_3x$};
\draw (3, -0.7) node(a12){$d_2d_0x$};
\draw (3, 0.3) node(a02){$d_2d_1x$};
\draw (3,-1.7) node (a22){$s_0d_1d_0d_0x$};
\draw (4, 0.3) node (a03){$d_1d_1x$};
\draw (4, -0.7) node (a13){$d_1d_0x$};
\draw (4, -1.7) node (a23){$d_0d_0x$};
\draw (4, -2.7) node (a33){$s_0d_0d_0d_0x.$};
%horizontal
\draw[mono] (a00)--node[above]{$s_0d_2d_3x$}(a01);
\draw[mono] (a11)--node[below]{$s_0d_0d_3x$}(a12);
\draw[mono] (a01)--node[above]{$d_3x$}(a02);
\draw[mono] (a02)--node[above]{$d_1x$}(a03);
\draw[mono] (a12)--node[below]{$d_0x$}(a13);
\draw[mono] (a22)--node[below]{$s_0d_0d_0x$}(a23);
%vertical
\draw[epi] (a01)--node[left]{$s_1d_2d_3x$}(a11);
\draw[epi] (a12)--node[left]{$s_1d_2d_0x$}(a22);
\draw[epi] (a02)--node[right]{$d_3x$}(a12);
\draw[epi] (a03)--node[right]{$d_2x$}(a13);
\draw[epi] (a13)--node[right]{$d_0x$}(a23);
\draw[epi] (a23)--node[right]{$s_1d_0d_0x$}(a33);
   %filling the square
   \draw[twoarrowlonger] (2.2,0.1)--node[above,xshift=0.3cm]{$s_1d_3x$}(2.8,-0.5); 
   \draw[twoarrowlonger] (3.2,0.1)--node[above]{$x$}(3.8,-0.5); 
   \draw[twoarrowlonger] (3.2,-0.9)--node[above,xshift=0.3cm]{$s_1d_0x$}(3.8,-1.5); 
\end{scope}

 \end{tikzpicture}
\end{center}
}}

A standard argument using the simplicial identities shows that these maps are compatible with the simplicial maps, thus showing that we have constructed an isomorphism of simplicial sets.
\end{proof}

Next, we want to show that the composite $ \pcat\sdot \colon \asdc \to \asdc$ is naturally isomorphic to the identity functor.

\begin{prop}\label{prop:equiv2}
Let $\cD$ be an augmented stable double category. There is a natural isomorphism of double categories
$$\varepsilon_{\cD}: \pcat\sdot (\cD)\xrightarrow{\,\,\cong\,\,}\cD.$$
\end{prop}

\begin{proof} 
Recall from \cref{fundcat} that the fundamental category functor $\tau_1$ is the left adjoint of the nerve functor, with the counit being an isomorphism.  Given an augmented stable double category $\cD$, the maps constructed in \cref{prop:PplusminusSdot} have adjoints $\tau_1\Pplus\sdot(\cD)\to \Horz{\cD}$ and $\tau_1\Pminus\sdot(\cD)\to \Verz{\cD}$ which are isomorphisms of categories and are natural in $\cD$. 
By \cref{rmk stable double functor}, if there is a double functor $\varepsilon_{\cD}$, it is already determined by these two isomorphisms. However, it is not immediate that they assemble to a double functor. Rather than checking their compatibility, we will construct a map on the squares of $\cP \sdot(\cD)$ which is compatible with these functors.

Recall that the set of squares of $\pcat\sdot(\cD)$ is given by $\sdotn{3}(\cD)$, which in turn is defined to be
\[
\Hom_{\adc}(\cW_3, \cD).
\]
To such a double functor $F\colon \cW_3\to \cD$, let $\varepsilon_{\cD}$ assign the square in $\cD$ which is the image of the square
 \begin{center}
  \begin{tikzpicture}[scale=0.6]
    \def\l{2cm}
    \begin{scope}
   \draw[fill] (0,0) node (b0){$13$};
   \draw[fill] (-\l,\l) node (b2){$02$};
   \draw[fill] (-\l,0) node (b3){$12$};
   \draw[fill] (0,\l) node (b1){$03$};

   \draw[epi] (b1)--node[anchor=west](x01){}(b0);
   \draw[mono] (b2)--node[anchor=south](x12){}(b1);
   \draw[epi] (b2)--node[anchor=east](x23){}(b3);
   \draw[mono] (b3)--node[anchor=north](x03){}(b0);

   \draw[twoarrowlonger] (-0.8*\l, 0.8*\l)--(-0.2*\l,0.2*\l);
   \end{scope}
  \end{tikzpicture}
 \end{center}
under $F$, a construction which is again natural in $\cD$. Since both $\pcat\sdot(\cD)$ and $\cD$ are stable, to check that we have defined a double functor, it is enough to check that this assignment is compatible with vertical and horizontal sources and targets.  As explained in \cref{const:path}, the horizontal source and target vertical morphisms are given by $d_3$ and $d_2$, respectively. Recall that these maps restrict $F$ to the subdiagrams
\begin{center}
 \begin{tikzpicture}[scale=1.2]
 \begin{scope}
   \draw (1,0.3) node(a00){$00$};
\draw (2,0.3) node(a01){$01$};
\draw (2,-0.7) node(a11) {$11$};
\draw (3, -0.7) node(a12){$12$};
\draw (3, 0.3) node(a02){$02$};
\draw (3,-1.7) node (a22){$22$};
\draw[mono] (a00)--(a01);
\draw[epi] (a01)--(a11);
\draw[mono] (a11)--(a12);
\draw[mono] (a01)--(a02);
\draw[epi] (a12)--(a22);
\draw[epi] (a02)--(a12);

   \draw[twoarrowlonger] (2.2,0.1)--(2.8,-0.5);
\end{scope}

\draw (4,0.3)node(x){and};
 \begin{scope}[xshift=4cm]
   \draw (1,0.3) node(a00){$00$};
\draw (2,0.3) node(a01){$01$};
\draw (2,-0.7) node(a11) {$11$};
\draw (3, -0.7) node(a12){$13$};
\draw (3, 0.3) node(a02){$03$};
\draw (3,-1.7) node (a22){$33$};
\draw[mono] (a00)--(a01);
\draw[epi] (a01)--(a11);
\draw[mono] (a11)--(a12);
\draw[mono] (a01)--(a02);
\draw[epi] (a12)--(a22);
\draw[epi] (a02)--(a12);

   \draw[twoarrowlonger] (2.2,0.1)--(2.8,-0.5);
\end{scope}
 \end{tikzpicture}
\end{center}
and the identification with $\Verz$ restricts these two diagrams to the maps $02\epi 12$ and $03\epi 13$, respectively.  But these maps  are exactly the horizontal source and target vertical morphisms of the square indicated above in $\cD$. Similarly, the described map is also compatible with vertical source and target.

Next, we have to check that the double functor $\varepsilon_{\cD} \colon \pcat\sdot(\cD)\to \cD$ is augmented. Recall that the augmentation of $\pcat\sdot(\cD)$ is given by the image of the degeneracy map $s_0\colon \sdotn{0}(\cD)\to \sdotn{1}(\cD)$. An augmented double functor $F\colon\cW_1\to \cD$ is in the image of $s_0$ if and only if it is of the form
\begin{center}
 \begin{tikzpicture}
\begin{scope}
\draw (-0.5,0) node(a00){$F(00)$};
\draw (1.5,0) node(a01){$F(01)$};
\draw (1.5,-1.5) node(a11) {$F(11).$};
\draw[double, double distance=1pt] (a00)--(a01);
\draw[double, double distance=1pt] (a01)--(a11);
\end{scope}
 \end{tikzpicture}
\end{center}
In particular, since $F$ is augmented, the object $F(01)$ of $\cD$ to which this object of $\pcat\sdot(\cD)$ is sent is an object in the augmentation set of $\cD$. Thus the double functor $\pcat\sdot(\cD)\to \cD$ is augmented.

We already know that $\varepsilon_{\cD}$ induces isomorphisms on $\Horz$ and $\Verz$, so it remains to show that it induces a bijection on the set of squares.  Given any square
 \begin{center}
  \begin{tikzpicture}[scale=0.6]
    \def\l{2cm}
    \begin{scope}
   \draw[fill] (0,0) node (b0){$w$};
   \draw[fill] (-\l,\l) node (b2){$x$};
   \draw[fill] (-\l,0) node (b3){$z$};
   \draw[fill] (0,\l) node (b1){$y$};

   \draw[epi] (b1)--node[anchor=west](x01){}(b0);
   \draw[mono] (b2)--node[anchor=south](x12){}(b1);
   \draw[epi] (b2)--node[anchor=east](x23){}(b3);
   \draw[mono] (b3)--node[anchor=north](x03){}(b0);

   \draw[twoarrowlonger] (-0.8*\l, 0.8*\l)--(-0.2*\l,0.2*\l);
   \end{scope}
  \end{tikzpicture}
 \end{center}
in $\cD$, we have to construct an augmented double functor $G\colon \cW_3\to \cD$ which maps
 \begin{center}
  \begin{tikzpicture}[scale=0.6]
    \def\l{2cm}
    \begin{scope}
   \draw[fill] (0,0) node (b0){$13$};
   \draw[fill] (-\l,\l) node (b2){$02$};
   \draw[fill] (-\l,0) node (b3){$12$};
   \draw[fill] (0,\l) node (b1){$03$};

   \draw[epi] (b1)--node[anchor=west](x01){}(b0);
   \draw[mono] (b2)--node[anchor=south](x12){}(b1);
   \draw[epi] (b2)--node[anchor=east](x23){}(b3);
   \draw[mono] (b3)--node[anchor=north](x03){}(b0);

   \draw[twoarrowlonger] (-0.8*\l, 0.8*\l)--(-0.2*\l,0.2*\l);
   \end{scope}
  \end{tikzpicture}
 \end{center}
 to it. Define $G(11\mono 12)$ and $G(12\epi 22)$ to be the unique elements $a_{11}\mono z$ and $z\epi a_{22}$ in $\Hor{\cD}$ and $\Ver{\cD}$, respectively, with $a_{11}$, $a_{22}$ in the augmentation set $\aug$.   By stability of $G$, the following span and cospan can be completed uniquely into squares, as indicated in the diagrams
 \begin{center} 
  \begin{tikzpicture}[scale=0.6]
    \def\l{2cm}
    \begin{scope}
   \draw[fill] (0,0) node (b0){$z$};
   \draw[fill] (-\l,\l) node (b2){$x_{01}$};
   \draw[fill] (-\l,0) node (b3){$a_{11}$};
   \draw[fill] (0,\l) node (b1){$x$};

   \draw[epi] (b1)--node[anchor=west](x01){}(b0);
   \draw[mono, dashed] (b2)--node[anchor=south](x12){}(b1);
   \draw[epi, dashed] (b2)--node[anchor=east](x23){}(b3);
   \draw[mono] (b3)--node[anchor=north](x03){}(b0);

   \draw[twoarrowlonger, dashed] (-0.8*\l, 0.8*\l)--(-0.2*\l,0.2*\l);
   \end{scope}

   \draw (\l, \l) node(x){and};

       \begin{scope}[xshift=3*\l]
   \draw[fill] (0,0) node (b0){$x_{23}.$};
   \draw[fill] (-\l,\l) node (b2){$z$};
   \draw[fill] (-\l,0) node (b3){$a_{22}$};
   \draw[fill] (0,\l) node (b1){$w$};

   \draw[epi, dashed] (b1)--node[anchor=west](x01){}(b0);
   \draw[mono] (b2)--node[anchor=south](x12){}(b1);
   \draw[epi] (b2)--node[anchor=east](x23){}(b3);
   \draw[mono, dashed] (b3)--node[anchor=north](x03){}(b0);

   \draw[twoarrowlonger, dashed] (-0.8*\l, 0.8*\l)--(-0.2*\l,0.2*\l);
   \end{scope}
  \end{tikzpicture}
 \end{center}
 We set these squares to be the images, respectively, of the squares
 \begin{center}
  \begin{tikzpicture}[scale=0.6]
    \def\l{2cm}
    \begin{scope}
   \draw[fill] (0,0) node (b0){$12$};
   \draw[fill] (-\l,\l) node (b2){$01$};
   \draw[fill] (-\l,0) node (b3){$11$};
   \draw[fill] (0,\l) node (b1){$02$};

   \draw[epi] (b1)--node[anchor=west](x01){}(b0);
   \draw[mono] (b2)--node[anchor=south](x12){}(b1);
   \draw[epi] (b2)--node[anchor=east](x23){}(b3);
   \draw[mono] (b3)--node[anchor=north](x03){}(b0);

   \draw[twoarrowlonger] (-0.8*\l, 0.8*\l)--(-0.2*\l,0.2*\l);
   \end{scope}

   \draw (\l, \l) node(x){and};

       \begin{scope}[xshift=3*\l]
   \draw[fill] (0,0) node (b0){$23$};
   \draw[fill] (-\l,\l) node (b2){$12$};
   \draw[fill] (-\l,0) node (b3){$22$};
   \draw[fill] (0,\l) node (b1){$13$};

   \draw[epi] (b1)--node[anchor=west](x01){}(b0);
   \draw[mono] (b2)--node[anchor=south](x12){}(b1);
   \draw[epi] (b2)--node[anchor=east](x23){}(b3);
   \draw[mono] (b3)--node[anchor=north](x03){}(b0);

   \draw[twoarrowlonger] (-0.8*\l, 0.8*\l)--(-0.2*\l,0.2*\l);
   \end{scope}
  \end{tikzpicture}
 \end{center}
under $G$. Finally, we use the augmentation again to produce unique maps $a_{00}\mono x_{01}$ and $x_{23}\epi a_{33}$ which are declared to be the images of $00\mono 01$ and $23\epi 33$ under $G$, respectively. Thus we have defined the desired double functor $G\colon \cW_3\to \cD$. Note that there were no choices involved in the construction of $G$, so $G$ is indeed the unique preimage of the given square under the map of squares of the double functor $\varepsilon_{\cD}:\pcat\sdot\cD\to \cD$. 
\end{proof}

\section{Three examples, revisited} \label{sec:examples}

In this section, we return to the examples of 2-Segal sets described in \cref{sec:twosegex} and construct their corresponding augmented stable double categories.  

\begin{ex}
Recall the 2-Segal set $M_\bullet$ which is the nerve of a partial monoid from \cref{ex: partial monoids}.   Let us consider the image of this 2-Segal set under the path construction.

The objects of the associated double category $\mathcal P M$ are the elements of the monoid $a\in M$. The set of horizontal morphisms, which is equal to the set of vertical morphisms, is the set of composable pairs $M_2$. However, their interpretation is different: for $(a,b) \in M_2$, the corresponding horizontal arrow has source $a$ and target $a\cdot b$, i.e., it can be interpreted as multiplication on the right by $b$,
$$\begin{tikzcd}
a \arrow[tail]{r}{\cdot b} & a\cdot b.
\end{tikzcd}$$
The vertical arrow corresponding to $(a,b)\in M_2$ has target $b$ and source $a\cdot b$, so it can be thought of as multiplication on the left by $a$, but with the arrow pointing in the other direction.
$$\begin{tikzcd}
a\cdot b \arrow[two heads, swap]{d}{a \cdot} \\  b.
\end{tikzcd}$$
The set of squares is $M_3$ and reflects associativity of the multiplication: for $(a,b,c)\in M_3$, its associated square is
$$\begin{tikzcd}
a\cdot b \arrow[two heads, swap]{d}{a \cdot} \arrow[tail]{r}{\cdot c} & (a\cdot b) \cdot c = a\cdot (b\cdot c) \arrow[two heads]{d}{a \cdot}  \\
b \arrow[tail, swap]{r}{\cdot c} & b\cdot c.
\end{tikzcd}$$
The double category $\pcat M$ is pointed by the unit $1\in M$, since for every $a\in M$, the elements $(1,a)\in M_2$ and $(a,1)\in M_2$, which can be visualized as
$$\begin{tikzcd}
1 \arrow[tail]{r}{\cdot a} & 1\cdot a = a
\end{tikzcd}
\hspace{1cm}\mbox{and}\hspace{1cm}
\begin{tikzcd}
a=a\cdot 1 \arrow[two heads, swap]{d}{a \cdot} \\  1,
\end{tikzcd}$$
exhibit 1 as initial with respect to the horizontal category and terminal with respect to the vertical category.  Finally, stability is again given by associativity, since both
$$\begin{tikzcd}
a\cdot b \arrow[two heads, swap]{d}{a \cdot} \arrow[tail]{r}{\cdot c} & (a\cdot b) \cdot c   \\
b 
\end{tikzcd}
\quad\mbox{and}\quad
\begin{tikzcd}
\mbox{} & a\cdot (b\cdot c) \arrow[two heads]{d}{a \cdot}  \\
b \arrow[tail, swap]{r}{\cdot c} & b\cdot c
\end{tikzcd}$$
can be completed uniquely to a square as above.
\end{ex}

\begin{ex}
Let us now revisit the 2-Segal set $2\mathrm{Cob}^{\leq g}$ from \cref{ex:cob}.

The objects of the associated double category are elements in $(2\mathrm{Cob}^{\leq g})_1$, which are diffeomorphism classes of 2-dimensional cobordisms $\Sigma$ with genus at most $g$. Horizontal and vertical morphisms are elements in $(2\mathrm{Cob}^{\leq g})_2$, which are given by diffeomorphism classes of cobordisms $\Sigma$ with genus at most $g$, together with a choice of decomposition $\Sigma \cong  \Sigma_1\amalg_{N} \Sigma_2$, where  $N=\partial_{\textrm{out}} \Sigma_1 = \partial_{\textrm{in}} \Sigma_2$. We view $\Sigma$ as a horizontal morphism and as a vertical morphism via
$$\begin{tikzcd}[column sep=large]
\Sigma_1 \arrow[tail]{r}{(-)\amalg_{N} \Sigma_2} & \Sigma_1\amalg_{N} \Sigma_2\cong \Sigma 
\end{tikzcd}
\quad\text{and}\quad
\begin{tikzcd}[column sep=large]
\Sigma \cong \Sigma_1 \amalg_{N} \Sigma_2 \arrow[two heads]{d}{\Sigma_1\amalg_{N} (-)}  \\ \Sigma_2,
\end{tikzcd}$$
respectively.

The augmentation is given by cylinders on 1-dimensional closed manifolds viewed as trivial cobordisms, since the set of such cylinders is the image of $(2\mathrm{Cob}^{\leq g})_0$ under the degeneracy map $s_0$. 

Given an object in the double category $\Sigma\in (2\mathrm{Cob}^{\leq g})_1$, there is a unique object in the augmentation, namely the cylinder on its incoming boundary, together with a unique horizontal morphism to $\Sigma$:
$$\begin{tikzcd}[column sep=large]
\partial_{\textrm{in}}\Sigma\times[0,1] \arrow[tail]{r}{(-)\amalg_{\partial_{\textrm{in}}\Sigma} \Sigma} & (\partial_{\textrm{in}}\Sigma\times[0,1]) \amalg_{\partial_{\textrm{in}}\Sigma} \Sigma \cong \Sigma.
\end{tikzcd}$$
For example, for the pair of pants, we have the horizontal morphism
$$
\raisebox{-0.5\height}{\includegraphics{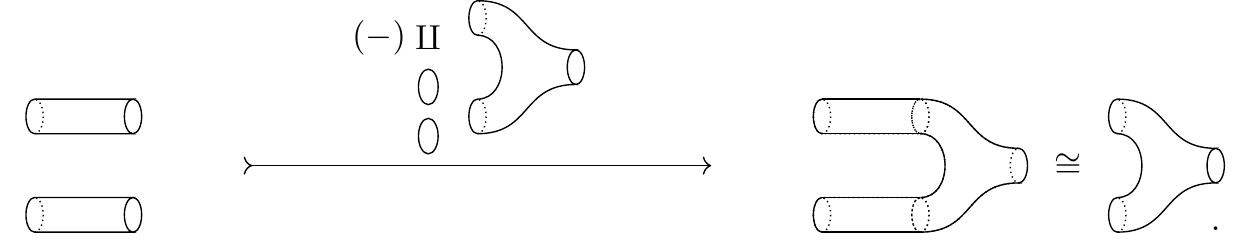}}
$$
Similarly, there is a unique object in the augmentation, namely the cylinder on its outgoing boundary, together with a unique vertical morphism from $\Sigma$:
$$\begin{tikzcd}[column sep=large]
\Sigma \cong \Sigma \amalg_{\partial_{\textrm{out}}\Sigma} (\partial_{\textrm{out}}\Sigma \times[0,1])\arrow[two heads]{d}{\Sigma \,\amalg_{\partial_{\textrm{out}}\Sigma} (-)}  \\ \partial_{\textrm{out}} \Sigma \times[0,1].
\end{tikzcd}$$
In the example of the pair of pants, the vertical morphism is (drawn 
horizontally)
$$
\raisebox{-0.5\height}{\includegraphics{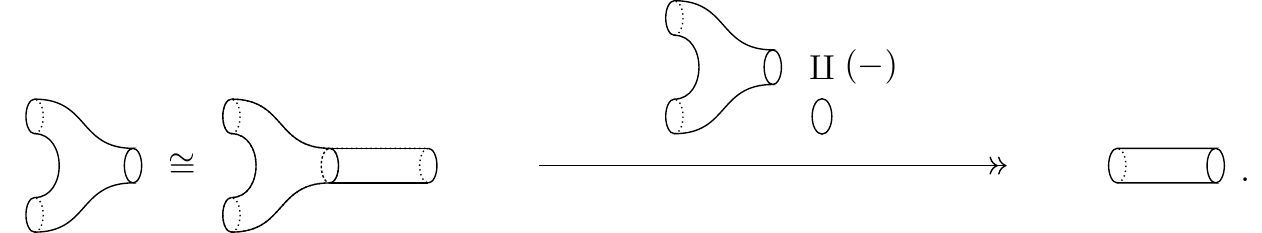}}
$$
\end{ex}

\begin{ex} 
Finally, we apply our construction to the 2-Segal set $X$ associated to a graph $G$ as described in \cref{ex:graph}.

An element in $X_2$, such as 
\[
\begin{tikzpicture}
\def\l{1cm}
\draw (\l,0) node [vnode=$a$](a){};
\draw (2*\l,0) node [vnode=$b$](b){};
\draw (\l,-0.3*\l) node[vellipsefirstone=0.8*\l]{};
\draw (2*\l, -0.3*\l) node[vellipsesecondone=0.8*\l]{};
\draw (a)--(b);
\end{tikzpicture}
\]
represents a horizontal morphism
\begin{center}
\begin{tikzpicture}
\def\l{1cm}
\draw (\l,0) node [vnode=$a$](a){};
\draw (\l,-0.3*\l) node[vellipsefirstone=0.8*\l]{};

\draw[mono] (1.5*\l,-0.3*\l) -- (5.5*\l,-0.3*\l);

%label of the mono
\begin{scope}[yshift=0.8*\l, xshift=3*\l]
\draw (0,0) node [vnode=$a$](a){};
\draw (\l,0) node [vnode=$b$](b){};
\draw (0,-0.3*\l) node[vellipsefirstone=0.8*\l]{};
\draw (\l, -0.3*\l) node[vellipsesecondone=0.8*\l]{};
\draw (a)--(b);
\end{scope}

\begin{scope}[xshift=5*\l]
\draw (\l,0) node [vnode=$a$](a){};
\draw (2*\l,0) node [vnode=$b.$](b){};
\draw (a)--(b);
\draw (1.5*\l, -0.3*\l) node[vellipsefirsttwo=0.8*\l]{};
\end{scope}

\end{tikzpicture}
\end{center}

The same element gives the vertical morphism
\begin{center}
\begin{tikzpicture}[scale=0.75]
\def\l{1cm}
\def\w{6cm}
\def\h{-5cm}

\begin{scope}
\draw (\l,0) node [vnode=$a$](a){};
\draw (2*\l,0) node [vnode=$b$](b){} ++(0.5*\l,-0.5*\l) node (t1) {};
\draw (1.5*\l,-0.3*\l) node[vellipsefirsttwo=0.8*\l]{};
\draw (a)-- node (l11) {} (b);
\path (l11) ++(0,-\l) node (l1) {};
\end{scope}

\begin{scope}[shift={(0.5, \h)}]
\draw (\l,0) node [vnode=$b.$](b){} node[anchor=south, yshift=0.2cm] (l2) {} ++(0.5*\l,-0.5*\l) node (b1) {};
\draw (\l,-0.3*\l) node[vellipsefirstone=0.8*\l]{};
\end{scope}

\draw[epi] (l1) -- (l2);
%label for leftmost arrow
\begin{scope}[shift={(-0.2*\l, 0.5*\h)}]
\draw (0,0) node [vnode=$a$](a){};
\draw (\l,0) node [vnode=$b$](b){};
\draw (0,-0.3*\l) node[vellipsefirstone=0.8*\l]{};
\draw (\l, -0.3*\l) node[vellipsesecondone=0.8*\l]{};
\draw (a)--(b);
\end{scope}
\end{tikzpicture}
\end{center}

To give an idea of the different way squares can be formed from such morphisms, we give a few examples.  A 3-simplex such as
\[
\begin{tikzpicture}
\def\l{1cm}
\draw (\l,0) node [vnode=$a$](a){};
\draw (2*\l,0) node [vnode=$b$](b){};
\draw (3*\l,0) node [vnode=$c$](c){};
\draw (\l,-0.3*\l) node[vellipsefirstone=0.8*\l]{};
\draw(2*\l, -0.3*\l) node[vellipsesecondone=0.8*\l]{};
\draw (3*\l, -0.3*\l) node[vellipsethirdone=0.8*\l]{};
\draw (a)--(b);
\draw (b)--(c);
\end{tikzpicture}
\]
gives rise to a square

\begin{center}
\begin{tikzpicture}[scale=0.75]
\def\l{1cm}
\def\w{6cm}
\def\h{-5cm}

\begin{scope}
\draw (\l,0) node [vnode=$a$](a){};
\draw (2*\l,0) node [vnode=$b$](b){} ++(0.5*\l,-0.5*\l) node (t1) {};
\draw (1.5*\l,-0.3*\l) node[vellipsefirsttwo=0.8*\l]{};
\draw (a)-- node (l11) {} (b);
\path (l11) ++(0,-\l) node (l1) {};
\end{scope}

\begin{scope}[xshift = \w]
\draw (\l,0) node [vnode=$a$](a){} ++(-0.8*\l,-0.5*\l) node (t2) {};
\draw (2*\l,0) node [vnode=$b$](b){} ++(0,-\l) node (r1) {};
\draw (3*\l,0) node [vnode=$c$](c){};
\draw (2*\l, -0.3*\l) node[vellipsefirstthree=0.8*\l]{};
\draw (a)--(b);
\draw (b)--(c);
\end{scope}

\draw[mono] (t1) -- node (T) {} (t2);
%label of the top mono 
\begin{scope}[shift={(0.5*\w++0.5*\l, 1.2*\l)}]
\draw (0,0) node [vnode=$a$](a){};
\draw (\l,0) node [vnode=$b$](b){};
\draw (2*\l,0) node [vnode=$c$](c){};
\draw (b)--(c);
\draw (a)--(b);
\draw (0.5*\l,-0.3*\l) node[vellipsefirsttwo=0.8*\l]{};
\draw (2*\l, -0.3*\l) node[vellipsesecondone=0.8*\l]{};
\end{scope}

\begin{scope}[shift={(0.5, \h)}]
\draw (\l,0) node [vnode=$b$](b){} node[anchor=south, yshift=0.2cm] (l2) {} ++(0.5*\l,-0.5*\l) node (b1) {};
\draw (\l,-0.3*\l) node[vellipsefirstone=0.8*\l]{};
\end{scope}

\draw[epi] (l1) -- (l2);
%label for leftmost arrow
\begin{scope}[shift={(-0.2*\l, 0.5*\h)}]
\draw (0,0) node [vnode=$a$](a){};
\draw (\l,0) node [vnode=$b$](b){};
\draw (0,-0.3*\l) node[vellipsefirstone=0.8*\l]{};
\draw (\l, -0.3*\l) node[vellipsesecondone=0.8*\l]{};
\draw (a)--(b);
\end{scope}

\begin{scope}[shift={(\w ++ 0.5cm, \h)}]
\def\l{1cm}
\draw (\l,0) node [vnode=$b$](b){} ++(-0.5*\l,-0.5*\l) node (b2) {};
\draw (2*\l,0) node [vnode=$c.$](c){};
\draw (b)-- node[anchor=south, yshift=0.2cm] (r2) {} (c);
\draw (1.5*\l,-0.3*\l) node[vellipsefirsttwo=0.8*\l]{};
\end{scope}

\draw[epi] (r1) -- node (R) {} (r2);
%label for the rightmost arrow
\path (R) ++ (0.5cm, 0) node (R1) {};
\begin{scope}[shift={(R1)}]
\draw (0,0) node [vnode=$a$](a){};
\draw (\l,0) node [vnode=$b$](b){};
\draw (2*\l,0) node [vnode=$c$](c){};
\draw (b)--(c);
\draw (a)--(b);
\draw (0,-0.3*\l) node[vellipsefirstone=0.8*\l]{};
\draw (1.5*\l, -0.3*\l) node[vellipsesecondtwo=0.8*\l]{};
\end{scope}

\draw[mono] (b1) -- node (B) {} (b2);
\path (B) ++ (-0.5, -1) node (B1) {};
\begin{scope}[shift={(B1)}]
\draw (0,0) node [vnode=$b$](a){};
\draw (\l,0) node [vnode=$c$](b){};
\draw (0,-0.3*\l) node[vellipsefirstone=0.8*\l]{};
\draw (\l, -0.3*\l) node[vellipsesecondone=0.8*\l]{};
\draw (a)--(b);
\end{scope}

\end{tikzpicture}
\end{center}

However, squares can look fairly different even if we permute the sets in the partition. For example, the 3-simplex
\[ 
\begin{tikzpicture}
\def\l{0.8cm}
\draw (\l,0) node [vnode=$a$](a){};
\draw (2*\l,0) node [vnode=$b$](b){};
\draw (3*\l,0) node [vnode=$c$](c){};
\draw (\l,-0.3*\l) node[vellipsethirdone=0.8*\l, ,label={below:{$3$}}]{};
\draw (2*\l, -0.3*\l) node[vellipsefirstone=0.8*\l,label={below:{$1$}}]{};
\draw (3*\l, -0.3*\l) node[vellipsesecondone=0.8*\l,label={below:{$2$}}]{};
\draw (a)--(b);
\draw (b)--(c);
\end{tikzpicture}
\]
corresponds instead to the square
\begin{center}
\begin{tikzpicture}[scale=0.8]
\def\l{1cm}
\def\w{5*\l}
\def\h{3.5*\l}

%corners

\begin{scope}
\draw (\l,0) node [vnode=$b$](b){};
\draw (2*\l,0) node [vnode=$c$](c){};
\draw (b)--(c);
\draw (1.5*\l,-0.3*\l) node[vellipsefirsttwo=0.8*\l](ulc){};
\end{scope}

\begin{scope}[xshift=\w]
\draw (\l,0) node [vnode=$a$](a){};
\draw (2*\l,0) node [vnode=$b$](b){};
\draw (3*\l,0) node [vnode=$c$](c){};
\draw (a)--(b);
\draw (b)--(c);
\draw (2*\l,-0.3*\l) node[vellipsefirstthree=0.8*\l](urc){};
\end{scope}

\begin{scope}[yshift=-\h, xshift=0.5*\l]
\draw (\l,0) node [vnode=$c$](c){};
\draw (\l,-0.3*\l) node[vellipsefirstone=0.8*\l](llc){};
\end{scope}

\begin{scope}[xshift={\w+0.5*\l}, yshift=-\h]
\draw (\l,0) node [vnode=$a$](a){};
\draw (2*\l,0) node [vnode=$c$](c){};
\draw (1.5*\l,-0.3*\l) node[vellipsefirsttwo=0.8*\l](lrc){};
\end{scope}

%arrows 
\draw[mono, shorten >=0.1cm, shorten <=0.1cm] (ulc)--node[above](d1){}(urc);
\draw[mono, shorten >=0.1cm, shorten <=0.1cm] (llc)--node[below](d0){}(lrc);
\draw[epi, shorten >=0.1cm, shorten <=0.1cm] (urc)--node[right](d2){}(lrc);
\draw[epi, shorten >=0.1cm, shorten <=0.1cm] (ulc)--node[left](d3){}(llc);

%arrow labels
\begin{scope}[shift={(-0.3*\w,-0.5*\h)}]
\draw (\l,0) node [vnode=$b$](b){};
\draw (2*\l,0) node [vnode=$c$](c){};
\draw (b)--(c);
\draw (1*\l,-0.3*\l) node[vellipsefirstone=0.8*\l]{};
\draw (2*\l, -0.3*\l) node[vellipsesecondone=0.8*\l]{};
\end{scope}

\begin{scope}[shift={(1.3*\w,-0.5*\h)}]
\draw (\l,0) node [vnode=$a$](a){};
\draw (2*\l,0) node [vnode=$b$](b){};
\draw (3*\l,0) node [vnode=$c$](c){};
\draw (a)--(b);
\draw (b)--(c);
\draw (2*\l,-0.3*\l) node[vellipsefirstone=0.8*\l]{};
\draw (1*\l,-0.3*\l) node[vellipsesecondone=0.8*\l]{};
\draw (3*\l,-0.3*\l) node[vellipsesecondone=0.8*\l]{};
\end{scope}

\begin{scope}[shift={(0.35*\w, 0.3*\h)}]
\draw (\l,0) node [vnode=$a$](a){};
\draw (2*\l,0) node [vnode=$b$](b){};
\draw (3*\l,0) node [vnode=$c$](c){};
\draw (a)--(b);
\draw (b)--(c);
\draw (2.5*\l,-0.3*\l) node[vellipsefirsttwo=0.8*\l]{};
\draw (\l, -0.3*\l) node[vellipsesecondone=0.8*\l]{};
\end{scope}

\begin{scope}[xshift={0.5*\w}, yshift=-1.3*\h]
\draw (\l,0) node [vnode=$a$](a){};
\draw (2*\l,0) node [vnode=$c$](c){};
\draw (2*\l,-0.3*\l) node[vellipsefirstone=0.8*\l]{};
\draw (1*\l,-0.3*\l) node[vellipsesecondone=0.8*\l]{};
\end{scope}

\end{tikzpicture}
\end{center}

Lastly, we give an example of a degenerate 3-simplex 
\[
\begin{tikzpicture}
\def\l{0.8cm}
\draw (\l,0) node [vnode=$a$](a){};
\draw (2*\l,0) node [vnode=$b$](b){};
\draw (3*\l,0) node [vnode=$c$](c){};
\draw (4*\l,-0.2*\l) node{$\varnothing$};
\draw (\l,-0.3*\l) node[vellipsefirstone=0.8*\l]{};
\draw (2.5*\l, -0.3*\l) node[vellipsesecondtwo=0.8*\l]{};
\draw (4*\l, -0.3*\l) node[vellipsethirdone=0.8*\l]{};
\draw (a)--(b);
\draw (b)--(c);
\end{tikzpicture}
\]
which gives rise to a square which represents an identity in the horizontal category.

\begin{center}
\begin{tikzpicture}[scale=0.8]
\def\l{1cm}
\def\w{6*\l}
\def\h{3.5*\l}

%corners

\begin{scope}
\draw (\l,0) node [vnode=$a$](a){};
\draw (2*\l,0) node [vnode=$b$](b){};
\draw (3*\l,0) node [vnode=$c$](c){};
\draw (a)--(b);
\draw (b)--(c);
\draw (2*\l,-0.3*\l) node[vellipsefirstthree=0.8*\l](ulc){};
\end{scope}

\begin{scope}[xshift=\w]
\draw (\l,0) node [vnode=$a$](a){};
\draw (2*\l,0) node [vnode=$b$](b){};
\draw (3*\l,0) node [vnode=$c$](c){};
\draw (a)--(b);
\draw (b)--(c);
\draw (2*\l,-0.3*\l) node[vellipsefirstthree=0.8*\l](urc){};
\end{scope}

\begin{scope}[yshift=-\h, xshift=0.5*\l]
\draw (\l,0) node [vnode=$b$](b){};
\draw (2*\l,0) node [vnode=$c$](c){};
\draw (b)--(c);
\draw (1.5*\l,-0.3*\l) node[vellipsefirsttwo=0.8*\l](llc){};
\end{scope}

\begin{scope}[xshift={\w+0.5*\l}, yshift=-\h]
\draw (\l,0) node [vnode=$b$](b){};
\draw (2*\l,0) node [vnode=$c.$](c){};
\draw (b)--(c);
\draw (1.5*\l,-0.3*\l) node[vellipsefirsttwo=0.8*\l](lrc){};
\end{scope}

%arrows 
\draw[mono, shorten >=0.1cm, shorten <=0.1cm] (ulc)--node[above](d1){}(urc);
\draw[mono, shorten >=0.1cm, shorten <=0.1cm] (llc)--node[below](d0){}(lrc);
\draw[epi, shorten >=0.1cm, shorten <=0.1cm] (urc)--node[right](d2){}(lrc);
\draw[epi, shorten >=0.1cm, shorten <=0.1cm] (ulc)--node[left](d3){}(llc);

%arrow labels
\begin{scope}[shift={(-0.3*\w,-0.5*\h)}]
\draw (\l,0) node [vnode=$a$](a){};
\draw (2*\l,0) node [vnode=$b$](b){};
\draw (3*\l,0) node [vnode=$c$](c){};
\draw (\l,-0.3*\l) node[vellipsefirstone=0.8*\l]{};
\draw (2.5*\l, -0.3*\l) node[vellipsesecondtwo=0.8*\l]{};
\draw (a)--(b);
\draw (b)--(c);
\end{scope}

\begin{scope}[shift={(1.3*\w,-0.5*\h)}]
\draw (\l,0) node [vnode=$a$](a){};
\draw (2*\l,0) node [vnode=$b$](b){};
\draw (3*\l,0) node [vnode=$c$](c){};
\draw (a)--(b);
\draw (b)--(c);
\draw (1*\l,-0.3*\l) node[vellipsefirstone=0.8*\l]{};
\draw (2.5*\l,-0.3*\l) node[vellipsesecondtwo=0.8*\l]{};
\end{scope}

\begin{scope}[shift={(0.35*\w, 0.4*\h)}]
\draw (\l,0) node [vnode=$a$](a){};
\draw (2*\l,0) node [vnode=$b$](b){};
\draw (3*\l,0) node [vnode=$c$](c){};
\draw (4*\l,-0.2*\l) node{$\varnothing$};
\draw (2*\l,-0.3*\l) node[vellipsefirstthree=0.7*\l]{};
\draw (4*\l, -0.3*\l) node[vellipsesecondone=0.7*\l]{};
\draw (a)--(b);
\draw (b)--(c);
\end{scope}

\begin{scope}[xshift={0.5*\w}, yshift=-1.3*\h]
\draw (\l,0) node [vnode=$b$](b){};
\draw (2*\l,0) node [vnode=$c$](c){};
\draw (3*\l,-0.2*\l) node{$\varnothing$};
\draw (1.5*\l,-0.3*\l) node[vellipsefirsttwo=0.8*\l]{};
\draw (3*\l, -0.3*\l) node[vellipsesecondone=0.8*\l]{};
\draw (b)--(c);
\end{scope}

\end{tikzpicture}
\end{center}
\end{ex}

\bibliographystyle{abbrv}
\bibliography{ref}

\begin{thebibliography}{10}

\bibitem{BarKthy}
C.~Barwick.
\newblock On the algebraic {$K$}-theory of higher categories.
\newblock {\em J. Topol.}, 9(1):245--347, 2016.

\bibitem{BGT}
A.~J. Blumberg, D.~Gepner, and G.~Tabuada.
\newblock A universal characterization of higher algebraic {$K$}-theory.
\newblock {\em Geom. Topol.}, 17(2):733--838, 2013.

\bibitem{Catalan}
E.~Catalan.
\newblock Note sur une {\'e}quation aux diff{\'e}rences finies.
\newblock {\em Journal de math{\'e}matiques pures et appliqu{\'e}es},
  3:508--516, 1838.

\bibitem{DK}
T.~Dyckerhoff and M.~Kapranov.
\newblock {H}igher {S}egal spaces {I}.
\newblock {\em ArXiv e-prints}, Dec. 2012.
\newblock arXiv:1212.3563.

\bibitem{Ehresmann}
C.~Ehresmann.
\newblock Cat\'egories structur\'ees.
\newblock {\em Ann. Sci. \'Ecole Norm. Sup. (3)}, 80:349--426, 1963.

\bibitem{FiorePaoliPronk}
T.~M. Fiore, S.~Paoli, and D.~Pronk.
\newblock Model structures on the category of small double categories.
\newblock {\em Algebr. Geom. Topol.}, 8(4):1855--1959, 2008.

\bibitem{FLP}
T.~M. {Fiore} and M.~{Pieper}.
\newblock {Waldhausen Additivity: Classical and Quasicategorical}.
\newblock {\em ArXiv e-prints}, July 2012.

\bibitem{GabrielZisman}
P.~Gabriel and M.~Zisman.
\newblock {\em Calculus of fractions and homotopy theory}.
\newblock Ergebnisse der Mathematik und ihrer Grenzgebiete, Band 35.
  Springer-Verlag New York, Inc., New York, 1967.

\bibitem{GalvezKockTonks}
I.~{G{\'a}lvez-Carrillo}, J.~{Kock}, and A.~{Tonks}.
\newblock {Decomposition spaces, incidence algebras and M\"obius inversion I:
  basic theory}.
\newblock {\em ArXiv e-prints}, Dec. 2015.
\newblock arXiv:1512.07573.

\bibitem{GalvezKockTonks_comb}
I.~{G{\'a}lvez-Carrillo}, J.~{Kock}, and A.~{Tonks}.
\newblock {Decomposition spaces in combinatorics}.
\newblock {\em ArXiv e-prints}, Dec. 2016.
\newblock arXiv:1612.09225.

\bibitem{GJ}
P.~G. Goerss and J.~F. Jardine.
\newblock {\em Simplicial Homotopy Theory}, volume 174 of {\em Progress in
  Mathematics}.
\newblock Birkh\"auser Verlag, Basel, 1999.

\bibitem{GP}
M.~Grandis and R.~Pare.
\newblock Limits in double categories.
\newblock {\em Cahiers Topologie G\'eom. Diff\'erentielle Cat\'eg.},
  40(3):162--220, 1999.

\bibitem{rezk}
C.~Rezk.
\newblock A model for the homotopy theory of homotopy theory.
\newblock {\em Trans. Amer. Math. Soc.}, 353(3):973--1007 (electronic), 2001.

\bibitem{SV}
C.~Scheimbauer and A.~Valentino.
\newblock Cobordism categories with constraints.
\newblock In preparation.

\bibitem{Schmitt}
W.~R. Schmitt.
\newblock Hopf algebras of combinatorial structures.
\newblock {\em Canad. J. Math.}, 45(2):412--428, 1993.

\bibitem{Segal}
G.~Segal.
\newblock Configuration-spaces and iterated loop-spaces.
\newblock {\em Invent. Math.}, 21:213--221, 1973.

\bibitem{segalgamma}
G.~Segal.
\newblock Categories and cohomology theories.
\newblock {\em Topology}, 13:293--312, 1974.

\bibitem{waldhausen}
F.~Waldhausen.
\newblock Algebraic {$K$}-theory of spaces.
\newblock In {\em Algebraic and geometric topology ({N}ew {B}runswick,
  {N}.{J}., 1983)}, volume 1126 of {\em Lecture Notes in Math.}, pages
  318--419. Springer, Berlin, 1985.

\end{thebibliography}

\end{document}